\pdfoutput=1
\documentclass[11pt]{article}
\usepackage[a4paper,margin=1in]{geometry}
\usepackage[T1]{fontenc}
\usepackage{lmodern}
\usepackage{amsmath,amssymb,amsthm}
\usepackage{graphicx}
\usepackage{booktabs}
\usepackage{xcolor}
\usepackage{tikz}
\usetikzlibrary{positioning,arrows.meta,calc,fit,backgrounds,decorations.pathreplacing}
\definecolor{stageblue}{HTML}{1A5FB4}
\definecolor{dirorange}{HTML}{C4661F}
\definecolor{boxgrey}{HTML}{555555}
\usepackage[colorlinks,linkcolor=blue!55!black,citecolor=blue!55!black,
            urlcolor=blue!55!black]{hyperref}
\usepackage{microtype}
\usepackage{adjustbox}
\usepackage{longtable}
\usepackage{array}
\usepackage{bm}
\usepackage[capitalise,nameinlink]{cleveref}

\theoremstyle{plain}
\newtheorem{theorem}{Theorem}
\newtheorem{lemma}{Lemma}
\newtheorem{proposition}[theorem]{Proposition}
\newtheorem{corollary}[theorem]{Corollary}
\theoremstyle{definition}
\newtheorem{definition}[theorem]{Definition}
\newtheorem{remark}{Remark}
\newtheorem{assumption}[theorem]{Assumption}

\newcommand{\RR}{\mathbb{R}}
\newcommand{\Om}{\Omega}
\newcommand{\norm}[1]{\lVert #1 \rVert}
\newcommand{\abs}[1]{\lvert #1 \rvert}
\newcommand{\ip}[2]{\langle #1, #2 \rangle}
\DeclareMathOperator{\dist}{dist}

\newcommand{\R}{\mathbb{R}}
\newcommand{\Ceps}{C_\epsilon}
\newcommand{\snorm}[2]{\lVert #1\rVert_{#2}}

\begin{document}

\title{From estimate to proof: certified ground-state energy bounds for singular Schr\"odinger operators on $\RR^3$}
\author{Xuefeng Liu\thanks{Department of Information and Sciences, Tokyo Woman's Christian University,
2--6--1 Zempukuji, Suginami-ku, Tokyo 167--8585, Japan. Email: \texttt{xfliu@lab.twcu.ac.jp}}}
\date{October 2026}
\maketitle

\begin{abstract}
We propose a rigorous computational method that certifies the separator required by Temple's inequality
and its Lehmann--Goerisch refinements, and with it unconditional ground-state bounds for
$H=-\Delta+V$ on $\RR^3$ with attractive Coulomb singularities. For H$_2^+$ the method certifies
$\lambda_2^{\rm eee}\ge-0.2035$ in place of the analytic $-\tfrac14$, and Temple's inequality with a
global trial then gives an enclosure of width below $6.5\times10^{-51}$. For linear
H$_3^{2+}$, whose central nucleus defeats the analytic separator, the certified widths are
$1.3\times10^{-12}$ and $3.1\times10^{-13}$ at two nuclear separations. Such a separator, a lower bound on
the second eigenvalue, is supplied by theory only for special configurations; elsewhere it is
estimated and the bound is conditional. Here every hypothesis
is verified in interval arithmetic at run time. Three constructions make this possible on an unbounded
domain with a potential of unbounded $L^\infty$ norm: a two-sided Neumann/Dirichlet truncation with a
computable remainder; a localized Hardy inequality and sharp form bound with explicit constants,
restoring coercivity; and a projection stage that certifies lower bounds on the first two eigenvalues,
the second serving as separator. On a box, a Lehmann--Goerisch step places
$\lambda_1\in[-0.553291,-0.550740]$ for H$_2^+$. The coercivity constant, we show, must bound its own
optimum from above. All drivers and certificates are released.
\end{abstract}



\section{Introduction}\label{sec:intro}

A computed eigenvalue is usually reported with an error estimate, typically extrapolated from a
convergence trend and blind to whatever the discretization does not yet resolve. A \emph{certificate}
is a different object: machine-representable numbers $\lambda^-\le\lambda_1\le\lambda^+$ together with
a finite list of hypotheses, each verified by the computation itself in interval arithmetic. The
distinction is practical. An error bar inferred from a convergence sequence can be optimistic by an
order of magnitude while every internal consistency check passes, and a bound formula evaluated in
floating point can return a value that is not a bound; \Cref{rem:float-not-bound} records an instance of
each encountered in this work.

\paragraph{Sharp lower bounds need a separator.} Upper bounds are easy: every Rayleigh--Ritz value is
one. Sharp lower bounds rest on Temple's inequality \cite{Temple1928} and its refinements
\cite{Weinstein1934,Lehmann1963,GoerischHaunhorst1985,BeattieGoerisch1995,CohenFeldmann1979,Hill1980}, and
all of them need a \emph{separator}, a number $\rho$ with $\lambda_1<\rho\le\lambda_2$, that is, a
certified lower bound on the second eigenvalue. A lower bound is only as rigorous as its separator.
Theory supplies one only for special arrangements of nuclei and electrons:
\begin{itemize}
\item comparison operators with known spectra, in the method of intermediate problems
\cite{Bazley1960,BazleyFox1961,BazleyFox1963}, which gave Bazley's rigorous Temple bound for helium
\cite{Bazley1960,Hill1980};
\item the bottom of the essential spectrum, when a single bound state exists;
\item for H$_2^+$, inversion symmetry together with the hydrogen spectrum (\Cref{lem:even-sep}), which
underlies the Temple bounds of \cite{GHT1978,Ashida2019}.
\end{itemize}
These separators are rough, and they are unavailable for genuine three-body systems \cite{Scrinzi1992}
or for geometries without the required symmetry. Elsewhere the separator is estimated, and sharp lower
bounds remain \emph{conditional}: genuine inequalities whose antecedent is assumed. The assumptions
include the number of bound states below a dissociation threshold, taken from variational spectra
\cite{Scrinzi1992}; the assignment of a variance-minimization bound to an excited level of
nonadiabatic H$_2^+$ \cite{KleindienstHoppe1986}; the validity condition of Weinstein's bound
\cite{Pollak2019}; known or estimated excited-state energies
\cite{PollakMartinazzo2021,Ireland2022,Ronto2023}; and, in a Temple bound for the Dirac equation of
H$_2^+$, the computed variational energy of the $2\sigma_g$ level, an upper bound to that level where
Temple's inequality needs a lower one \cite{Ishikawa2008}. The reported numbers are almost certainly right, but
they are implications whose antecedents have not been established, and none of the cited evaluations
uses verified arithmetic.

\paragraph{This paper.} We propose a rigorous computational method that certifies the separator
directly, so that sharp lower bounds for the ground-state energy become unconditional. For
Schr\"odinger operators
\begin{equation}\label{eq:H}
  H \;=\; -\Delta + V,\qquad V(x) \;=\; -\sum_{k=1}^{M} \frac{Z_k}{\abs{x-a_k}},\qquad x\in\RR^3,
\end{equation}
---~throughout, $V^-:=\max(-V,0)=\sum_k Z_k/\abs{x-a_k}$ denotes the attractive (singular) part, so that $-(Vu,u)=(V^-u,u)$, and $Z_{\rm tot}:=\sum_k Z_k$~---
three obstacles stand in the way. The domain is unbounded, so a discretization must first be justified
on a bounded box and then lifted back. The potential has unbounded $L^\infty$ norm, so the coercivity
estimates that underpin projection-type lower bounds do not apply directly. And the separator required
by the Lehmann--Goerisch sharpening \cite{Lehmann1963,GoerischHaunhorst1985,BeattieGoerisch1995} must
itself be certified. Here it is generated inside the method, as the certified projection bound on the
second eigenvalue, with no external spectral input (\Cref{sec:mu2}).

\paragraph{Related work.} The method of \cite{MartinazzoPollak2020} was developed on lattice
Hamiltonians; localisation results for Schr\"odinger eigenvalues are surveyed in \cite{CarstensenLoc2026}.
On the verified-computation side, finite-element eigenvalue enclosures are well developed
\cite{LiuOishi2013,Liu2015,SebestovaVejchodsky2014,%
CarstensenGedicke2014,CarstensenPuttkammer2024,Liu2024,LiuPlum2025,%
LiuNakaoOishi2022,NakaoPlumWatanabe,HuMa2025}. The projection framework used here was developed for
spectral discretizations in \cite{LiuSpectral2026,LiuConfining2026,LiuNM2026}. The whole-space step is
classically handled by Agmon-type decay estimates \cite{Agmon1982}, which supply no computable constant
at a given box size. To our knowledge no computer-verified (interval-arithmetic) two-sided enclosure of a whole-space
eigenvalue of a Coulomb-singular Schr\"odinger operator on $\RR^3$ has been reported. The H$_2^+$ results are validated
against the separable solution \cite{MadsenPeek1970,ScottAFG2006}, and the global-trial enclosure of
\Cref{sec:global} is compared with the $100$-digit Riccati--Pad\'e estimate of \cite{FernandezGarcia2021},
which carries no error bound. The H$_3^{2+}$
results are compared with floating-Gaussian reference values (\Cref{sec:results-h3}). The closest work,
the finite-element treatment of \cite{LiuNM2026}, is mainly two-dimensional. Its three-dimensional
$\lambda_2$ bound is too rough to serve as a Lehmann--Goerisch separator. On the coercivity constant, by
contrast, finite elements are sharper (\Cref{rem:sharpness}), so neither treatment dominates.

\paragraph{What is and is not claimed.} H$_2^+$ is the test system, not the contribution. Its
spectrum is separable and known to many digits, but those values carry no error bound, and Temple
bounds with the analytic separator of \Cref{lem:even-sep} reach about $10^{-3}$\,hartree
\cite{GHT1978,Ashida2019}. A much narrower Temple bound, with an upper--lower gap of
$2.0\times10^{-7}$\,hartree, has been reported for the Dirac equation of H$_2^+$ \cite{Ishikawa2008}; its
separator is the computed $2\sigma_g$ energy, an estimate, and the evaluation is in floating point. What is new is that the separator is computed and certified, using neither
the separability, nor the inversion symmetry, nor the hydrogen decomposition. The method therefore
applies unchanged where those arguments fail, as for linear H$_3^{2+}$ with its central nucleus, and
every hypothesis is verified in interval arithmetic. \Cref{sec:global} shows what a separator buys once
the box is bypassed, a certified width below $6.5\times10^{-51}$ for H$_2^+$ with a separator the projection
stage certifies itself, by an argument that does not use the structure of H$_2^+$; a rigorous treatment of general geometries is left to
further work. The obstruction removed --- a potential of unbounded $L^\infty$ norm on an unbounded
domain, and a spectral hypothesis the computation cannot check --- is generic to Coulomb-singular
operators; H$_2^+$ and H$_3^{2+}$ are where the constructions are exercised.

\paragraph{Contributions.} This paper closes the chain and then asks what limits the resulting certificate.
\begin{enumerate}\itemsep1pt
\item \textbf{A quantitative box-to-space lift (\Cref{sec:truncation}).} The confinement criterion of
      \cite{LiuConfining2026}, on $\R^3$ in \Cref{thm:enc-main}, bounds the truncation error in both
      directions with computable constants instead of the qualitative Agmon estimate, and the sign of the
      truncation bias is proved.
\item \textbf{Coercivity in three dimensions (\Cref{sec:hardy}).} A localized Hardy inequality with
      explicit constants restores coercivity for potentials of unbounded $L^\infty$ norm, and its sharp
      constant is obtained as an auxiliary eigenvalue.
\item \textbf{A self-generated separator (\Cref{sec:stageB}).} The Lehmann--Goerisch separator is
      certified as the least of three sector bounds, with no sector identification and no external
      spectral input (\Cref{thm:mu2-candidates}). The projection stage alone already gives
      $\lambda_1\ge-0.578393280$ on $\Om_2$ (\Cref{tab:stagea}); the second stage brings the lower
      endpoint to $-0.5532907722$, a factor $13.7$ closer (\Cref{app:stageB-route}).
\item \textbf{Box-independent cost (\Cref{sec:optimalbox}).} Every dense interval matrix is sized by the
      basis, not the box, so the box can balance truncation against discretization at no cost.
\item \textbf{Transferability (\Cref{sec:results-h3}).} Linear H$_3^{2+}$, with a Coulomb singularity at
      the centre of the box, is certified with no new mathematics and no special-cased quadrature.
\item \textbf{Cost (\Cref{sec:cost}).} The interval Kronecker assembly is thread-exact, which removes a
      factor $21$ from the dominant step; the enclosure \eqref{eq:headline} takes $5.3$ hours end to end.
\item \textbf{Outlook (\Cref{sec:global}).} With a lower bound for $\lambda_2$, Temple's inequality and a
      global trial need no box. For H$_2^+$ (units of \Cref{sec:units}) they certify
      \[
        \lambda_1\in-0.5513171072\,4747323075\,4484472659\,1705800970\,376395133{}_{4628138}^{1662658}
      \]
      (sub- and superscript: last digits of the lower and upper endpoints), a width below
      $6.5\times10^{-51}$ and, to our knowledge, the first unconditional high-precision lower bound for
      H$_2^+$ (the Temple bound of \cite{Ishikawa2008} rests on an estimated separator). The
      separator $-\tfrac14$ used here and in \cite{GHT1978,Ashida2019} relies on the inversion symmetry and
      hydrogen spectrum of H$_2^+$; Stage~A certifies $\lambda_2^{\rm eee}\ge-0.2035$ without
      either (\Cref{app:stageA-sep}), and the enclosure above uses it.

\end{enumerate}

\paragraph{Organization.} \Cref{fig:framework} shows the chain: the lift to $\RR^3$ (\Cref{sec:truncation}),
the projection bound (\Cref{sec:stageA}) and the self-generated separator with Lehmann--Goerisch
sharpening (\Cref{sec:stageB}). \Cref{sec:setting} fixes hypotheses and units, \Cref{sec:contract,sec:realization}
give the exact/approximate contract and its implementation, \Cref{sec:optimalbox,sec:results} the truncation
gap and the results, and \Cref{sec:global} the global-trial outlook. Proofs and parameter choices are in
the appendices; \Cref{app:notation} lists the main notation.

\section{Problem setting, hypotheses, and conventions}\label{sec:setting}

The quantity of interest is the ground-state energy of the hydrogen molecular
ion: the lowest eigenvalue $\lambda_1$ of the Coulomb Schr\"odinger operator
\begin{equation}\label{eq:H-R3}
  H=-\Delta-\frac1{|x-a_1|}-\frac1{|x-a_2|},\qquad a_{1,2}=(\mp2,0,0),\ Z_{\rm tot}=2,
\end{equation}
posed on the \emph{whole space} $\R^3$. This is where the physics lives, and it is the number this paper reports; everything that follows is in service of enclosing $\lambda_1(\R^3)$.

\paragraph{Why we compute on a box.} A verified lower bound needs a finite matrix and an explicit bound
on what that matrix leaves out. On a box both are available for any arrangement of nuclei: the
tensor-product cosines are the exact eigenfunctions of the Neumann Laplacian on $\Omega$, so the first
omitted eigenvalue $\nu_*$ is known in closed form and the projection error of \Cref{thm:coulomb-bound}
is computable; on $\R^3$ no basis supplies such a bound in general (\Cref{sec:global} treats the
exception, H$_2^+$). We therefore solve on a box $\Omega\supset\{a_1,a_2\}$, and the box is a rigorous
comparison rather than an approximation: because $V\to0$ at infinity, $\sigma(\Omega):=\inf_{x\notin\Omega}
V(x)>\lambda_1(\R^3)$ for $\Omega$ large enough, and then
\[
  \mu_1^{\mathrm{Neu}}(\Omega)\ \le\ \lambda_1(\R^3)\ \le\ \lambda_1^{\mathrm{Dir}}(\Omega)
\]
(\Cref{sec:truncation}). We certify these two box eigenvalues; their gap is the only price of
truncation (\Cref{sec:trunc,sec:dir}). Two nested boxes are used throughout,
\begin{equation}\label{eq:boxes}
  \Om_1=[-8,8]\times[-6.4,6.4]^2,\qquad \Om_2=[-10,10]\times[-8,8]^2,
\end{equation}
both in the H$_2^+$ geometry with $a_{1,2}=(\mp2,0,0)$ and $L_y=L_z=0.8\,L_x$. $\Om_2$ is the working box
and carries the sharpest certified enclosure; $\Om_1$ is smaller by design, to exhibit the truncation
gap ($8.5\times10^{-3}$ against $1.1\times10^{-3}$ on $\Om_2$), and the two sit on opposite sides of the
resolution threshold (\Cref{sec:dir}). The confinement condition of \Cref{thm:enc-main} holds on both,
$\sigma(\Om_2)=-0.2425$ and $\sigma(\Om_1)=-0.2983$ against $\lambda_1\approx-0.551$. Further boxes and
the three-centre geometries are introduced in \Cref{sec:optimalbox,sec:results}. All arithmetic is
verified interval arithmetic (versions in \Cref{sec:repro}), so every reported bound is
mathematically guaranteed.

\subsection{Scaling convention}\label{sec:units}

The operator \eqref{eq:H} carries no factor $\tfrac12$ on the Laplacian. Substituting $x=2y$ gives
\begin{equation}\label{eq:scaling}
  -\Delta_x - \sum_k \frac{Z_k}{\abs{x-a_k}}
  \;=\;
  \frac12\left[\,-\frac12\Delta_y - \sum_k \frac{Z_k}{\abs{y - a_k/2}}\,\right],
\end{equation}
so numerically (length here) $=2\times$(bohr) and (energy here) $=\tfrac12\times$(hartree). A width $w$
quoted here is $2w$\,hartree, and is converted accordingly whenever compared with chemical accuracy. Nuclei at $a_{1,2}=(\mp2,0,0)$ are
the equilibrium H$_2^+$ geometry $R=2$\,bohr. Its electronic ground-state energy $-1.1026342144949$\,hartree
(the separable prolate-spheroidal solution of \cite{MadsenPeek1970,ScottAFG2006}, there quoted as total
energy including $1/R$; reproduced here to $9.4\times10^{-8}$ by an independent floating-Gaussian
calculation) appears as
\begin{equation}\label{eq:lam1exact}
  \lambda_1 \;=\; -0.55131710724745 ,
\end{equation}
and the $2p\sigma_u$ level $-0.6675340$\,hartree as
$\lambda_2^{\text{oee}} = -0.3337670$\footnote{A three-letter string labels a \emph{parity sector}: the $i$-th letter records the behaviour under $x_i\mapsto-x_i$ (e even, o odd). In the cosine basis, $\cos\bigl(n\pi(x_i+L_i)/(2L_i)\bigr)$ has parity $(-1)^n$, so a sector is selected by index parity. For nuclei on the $x_1$-axis, eee is $\sigma_g$, oee is $\sigma_u$, and eoe/eeo is $\pi_u$ (\Cref{lem:parity-decomp}).}.
Equation~\eqref{eq:lam1exact} is the reference against which every H$_2^+$ certificate is checked.

\begin{figure}[tbp]\centering
\begin{adjustbox}{max width=\linewidth}
\begin{tikzpicture}[
  font=\small, x=1cm, y=1cm,
  src/.style={draw=boxgrey, rounded corners=2pt, align=center, font=\scriptsize,
              inner sep=3pt, text width=25mm},
  down/.style={-{Stealth[length=1.8mm]}, draw=boxgrey, shorten >=1pt},
  brc/.style={decorate, decoration={brace, amplitude=4pt, raise=1pt, mirror}, draw=boxgrey},
]
\def\yc{0}
\node (t1) at (0.60,\yc) {$\textcolor{stageblue}{L_k^{\rm Neu}}$};
\node (t2) at (4.00,\yc) {$\textcolor{stageblue}{L_{k,{\rm LG}}^{\rm Neu}}$};
\node (t3) at (6.90,\yc) {$\mu_k^{\rm Neu}(\Om)$};
\node (t4) at (9.35,\yc) {$\lambda_k(\R^3)$};
\node (t5) at (11.85,\yc) {$\lambda_k^{\rm Dir}(\Om)$};
\node (t6) at (14.35,\yc) {$\textcolor{dirorange}{U_k^{\rm Dir}}$};
\foreach \a/\b in {t1/t2, t2/t3, t3/t4, t4/t5, t5/t6}
  {\path (\a) -- (\b) node[midway] {$\le$};}
\node[src] (SA) at (0.60,1.40) {Stage~A\\ projection-based lower bound\\ (\Cref{thm:coulomb-bound})};
\node[src] (SB) at (4.00,1.40) {Stage~B\\ Lehmann--Goerisch\\ (\Cref{thm:lg-sharp})};
\node[src] (RR) at (14.35,1.40) {Rayleigh--Ritz\\ on $H^1_0(\Om)$};
\draw[down] (SA.south) -- (t1.north);
\draw[down] (SB.south) -- (t2.north);
\draw[down] (RR.south) -- (t6.north);
\draw[-{Stealth[length=1.8mm]}, draw=stageblue, thick, shorten <=1pt, shorten >=1pt]
      (SA.east) -- (SB.west);
\node[font=\scriptsize] at (2.30,2.32)
     {separator $\rho\le L_{k+1}^{\rm Neu}$ \ \eqref{eq:separation}};
\node[font=\scriptsize\itshape, align=center] at (9.40,1.45)
     {\textcolor{stageblue}{blue}\textcolor{boxgrey}{: certified lower bounds}\\
      \textcolor{dirorange}{orange}\textcolor{boxgrey}{: certified upper bound}\\[2pt]
      \textcolor{black}{black}\textcolor{boxgrey}{: the exact eigenvalues being bracketed}};
\draw[brc] (t3.west |- 0,-0.45) -- (t5.east |- 0,-0.45)
      node[midway, below=5pt, font=\scriptsize, align=center]
      {\Cref{thm:enc-main}: the truncation lift, valid whenever
       $\lambda_k(\R^3)<\sigma(\Om)$,\\ which the run certifies through
       $\lambda_k(\R^3)\le U_k^{\rm Dir}<\sigma(\Om)$};
\draw[brc] (t2.west |- 0,-1.95) -- (t6.east |- 0,-1.95)
      node[midway, below=5pt, font=\scriptsize, align=center]
      {the reported enclosure at $k=1$: $\lambda_1(\R^3)\in[\,L_{1,{\rm LG}}^{\rm Neu},\,U_1^{\rm Dir}\,]$};
\end{tikzpicture}
\end{adjustbox}
\caption{Every bound the pipeline produces, in order, and the step that produces each. The outer
terms are computed and certified; the inner terms are the exact eigenvalues being bracketed. Stage~A
gives $L_k^{\rm Neu}\le\mu_k^{\rm Neu}(\Om)$ for every $k$. Stage~B, seeded by Stage~A's bound at the next
index (so no external spectral input enters), gives a sharper lower bound. Rayleigh--Ritz on the
Dirichlet form gives $U_k^{\rm Dir}$. The middle span \eqref{eq:thm-trunc} carries box bounds to $\R^3$;
at $k=1$ the outer certified pair is the reported enclosure, and $k=2$ supplies the separator.}
\label{fig:framework}
\end{figure}
\section{From box to whole space: a quantitative lift}\label{sec:truncation}

The operator $\mathcal H=-\Delta+V$ is posed on all of $\R^3$, yet any computation
must replace it by a problem on a bounded box $\Omega$. Before describing the
algorithm we settle \emph{when} this restriction is legitimate and, more
importantly, how its error can be \emph{computed} rather than assumed small---so
that a bound obtained on $\Omega$ is a rigorous bound for the true whole-space
energy. This is why the truncation analysis comes first: it certifies the very
reduction on which everything downstream operates.

Classically, truncation is justified by decay. Agmon's theory \cite{Agmon1982}
shows that a confined ground state decays exponentially away from the nuclei, so a
sufficiently large box \emph{should} capture almost all of it. But this argument
is qualitative: it guarantees convergence as $\Omega\uparrow\R^3$ without supplying
a computable error for a box of a given size, leaving unknown Agmon constants in
the estimate. For a certified enclosure we need the truncation error bounded
explicitly on both sides.

Two elementary comparisons achieve this, one for each boundary condition on
$\Omega$.

\begin{theorem}[Whole-space enclosure]\label{thm:enc-main}
Let $\mu_k^{\mathrm{Neu}}(\Omega)$ and $\lambda_k^{\mathrm{Dir}}(\Omega)$ be the
$k$-th Neumann and Dirichlet eigenvalues of $\mathcal H=-\Delta+V$ on a bounded
box $\Omega\subset\R^3$, and let $\sigma(\Omega):=\inf_{x\notin\Omega}V(x)$ be the
confinement constant. Fix $k\in\{1,2,3,\dots\}$. If
$\sigma(\Omega)>\lambda_k(\R^3)$, then
\begin{equation}\label{eq:thm-trunc}
\mu_k^{\mathrm{Neu}}(\Omega)\le\lambda_k(\R^3)\le\lambda_k^{\mathrm{Dir}}(\Omega),
\end{equation}
so any certified $L^{\mathrm{Neu}}\le\mu_k^{\mathrm{Neu}}(\Omega)$ and
$U^{\mathrm{Dir}}\ge\lambda_k^{\mathrm{Dir}}(\Omega)$ bracket the $k$-th
whole-space eigenvalue,
\begin{equation}\label{eq:enclosure}
\lambda_k(\R^3)\in\bigl[\,L^{\mathrm{Neu}},\ U^{\mathrm{Dir}}\,\bigr].
\end{equation}
The upper bound holds for every $k$ by domain monotonicity; the lower bound uses
only that the exterior potential floor strictly clears the target level,
$\sigma(\Omega)>\lambda_k(\R^3)$. The ground state is the case $k=1$. The $k$ whole-space eigenfunctions used in the lower bound exist because $\sigma(\Omega)>\lambda_k(\R^3)$ and $\sigma(\Omega)<0=\inf\sigma_{\rm ess}(\mathcal H)$ place $\lambda_k(\R^3)$ strictly below the essential spectrum.

\end{theorem}

\paragraph{The lower side (Neumann).} If the truncated potential confines the
ground state, the lowest \emph{Neumann} eigenvalue on $\Omega$ lies below
$\lambda_1(\R^3)$, so a guaranteed lower bound computed with Neumann conditions is
a rigorous lower bound for the whole-space eigenvalue. This one-sided confinement criterion is the
Dirichlet--Neumann bracketing argument of \cite{LiuConfining2026}, which already
gives the bound in the form used here; that reference states it for confining
potentials on $\R^2$, but the argument nowhere uses the dimension.
\Cref{thm:enc-main} records it on $\R^3$, with its proof, for the reader's
convenience. It is the lower bound we sharpen with the two-stage
algorithm of the next section.

\paragraph{The upper side (Dirichlet), and its limitation.} By domain monotonicity
($H^1_0(\Omega)\subset H^1(\R^3)$ by extension by zero), a Dirichlet Rayleigh--Ritz value on $\Omega$ is a
rigorous whole-space upper bound, computed with the same box discretization as the lower bound. Its
limitation is intrinsic. The true ground state does not vanish on $\partial\Omega$, so the Dirichlet
eigenvalue overestimates $\lambda_1(\R^3)$ by an amount that no grid refinement removes; only enlarging
$\Omega$ reduces it. The results below show that this overshoot dominates the final width.

\paragraph{Higher levels supply the separator.} At $k=2$, under $\sigma(\Omega)>\lambda_2(\R^3)$,
\eqref{eq:thm-trunc} gives $\mu_2^{\mathrm{Neu}}(\Omega)\le\lambda_2(\R^3)$. A certified lower bound on the
second Neumann eigenvalue is therefore a certified separator $\rho\le\lambda_2$ for the second stage
(\Cref{sec:stageB}), obtained from the same box computation.

\paragraph{Separating truncation from discretization.} The residual width of \eqref{eq:enclosure} is
the domain-truncation error, bounded explicitly on both sides by \Cref{thm:enc-main}. A single box cannot
separate it from discretization error. Two nested boxes $\Omega_1\subset\Omega_2$, each with a Neumann and
a Dirichlet bracket, expose which part of the width is truncation (nearly $N$-independent) and which is
discretization (vanishing in $N$); \Cref{sec:results} quantifies the interplay.

A proof, after \cite{LiuConfining2026}, is given at the end of this section; its lower-bound step
cancels the exterior form directly on the restricted test function, so no factor depending on the
exterior mass enters.

\begin{remark}[explicit truncation control versus Agmon decay]\label{rem:agmon}
Agmon's theory \cite{Agmon1982} gives exponential decay of eigenfunctions in the Agmon distance. It is the
standard reason one may truncate to a box at all, but its constants give no computable bound on
$\lambda_1^{\mathrm{Dir/Neu}}(\Omega)-\lambda_1(\R^3)$ for a box of given size. \Cref{thm:enc-main}
replaces that qualitative statement by an explicit two-sided one: the truncation error is enclosed by the
certified Neumann and Dirichlet eigenvalues, using only the computable constant $\sigma(\Omega)$ in
\eqref{eq:thm-trunc}.
\end{remark}

\paragraph{Provenance, and what changes in three dimensions.}
The comparison behind \Cref{thm:enc-main} is from \cite{LiuConfining2026}, where it is proved on $\R^2$
for a confining, locally bounded potential. The extension in dimension is benign: domain monotonicity
and trial-function comparison do not use the dimension. The Coulomb singularities require two checks.
First, the form must be bounded below on $H^1(\Omega)$, which the localized Hardy inequality of
\Cref{sec:hardy} provides. Second, $\sigma(\Omega)=\inf_{x\notin\Omega}V$ must be finite and computable.
Since $-V$ is positive and harmonic on the open exterior and vanishes at infinity, the infimum is attained
on $\partial\Omega$. For collinear nuclei each face minimum reduces to a one-dimensional search, enclosed
by interval evaluation on a subdivision with an interval gradient bound. The hypothesis
$\sigma(\Omega)>\lambda_k(\R^3)$ is thus a checkable inequality, verified at every reported
configuration; it is the only place the truncation argument can fail, and it fails visibly.

\begin{proof}[Proof of \Cref{thm:enc-main}]
Write $a_D(u,v)=\int_D(\nabla u\cdot\nabla v+Vuv)$ and $\|u\|_D^2=\int_D|u|^2$, and let
$\psi_1,\dots,\psi_k$ be $L^2$-orthonormal eigenfunctions for
$\lambda_1(\R^3)\le\cdots\le\lambda_k(\R^3)$.
\emph{Dirichlet side.} Extension by zero embeds $H^1_0(\Omega)$ in $H^1(\R^3)$ with unchanged Rayleigh
quotient, so min--max gives
\begin{equation}\label{app:eq:dir-upper}
\lambda_k^{\mathrm{Dir}}(\Omega)\ge\lambda_k(\R^3)
\end{equation}
with no confinement hypothesis.
\emph{Neumann side.} Let $S_k=\operatorname{span}\{\psi_j|_\Omega\}_{j\le k}\subset H^1(\Omega)$. For
$v=\sum_{j\le k}c_j\psi_j$ put
$I_{\rm ext}=\int_{\R^3\setminus\Omega}(|\nabla v|^2+V|v|^2)$ and
$m_{\rm ext}=\int_{\R^3\setminus\Omega}|v|^2$. Then $a_{\R^3}(v,v)\le\lambda_k(\R^3)\|v\|^2_{\R^3}$, and
$V\ge\sigma(\Omega)$ outside $\Omega$ gives
\begin{equation}\label{app:eq:conf-cond}
I_{\rm ext}-\lambda_k(\R^3)m_{\rm ext}\ \ge\ \bigl(\sigma(\Omega)-\lambda_k(\R^3)\bigr)m_{\rm ext}\ \ge\ 0 .
\end{equation}
If $v|_\Omega=0$, then $m_{\rm ext}=\|v\|^2_{\R^3}$ and $a_{\R^3}(v,v)=I_{\rm ext}$, so
\eqref{app:eq:conf-cond} forces $v=0$. Hence restriction is injective and $\dim S_k=k$. For $v\ne0$,
$a_\Omega(v,v)=a_{\R^3}(v,v)-I_{\rm ext}\le\lambda_k(\R^3)(\|v\|^2_{\R^3}-m_{\rm ext})
=\lambda_k(\R^3)\|v\|^2_\Omega$ by \eqref{app:eq:conf-cond}, and min--max for the Neumann problem gives
\begin{equation}\label{app:eq:neu-trial}
\mu_k^{\mathrm{Neu}}(\Omega)\le\max_{0\ne v\in S_k}a_\Omega(v,v)/\|v\|_\Omega^2\le\lambda_k(\R^3).
\end{equation}
Together,
\begin{equation}\label{app:eq:enc-final}
  \mu_k^{\mathrm{Neu}}(\Omega)\ \le\ \lambda_k(\R^3)\ \le\ \lambda_k^{\mathrm{Dir}}(\Omega),
\end{equation}
with no multiplicative constant lost, and no simplicity or gap hypothesis beyond the strict inequality
$\sigma(\Omega)>\lambda_k(\R^3)$.
\end{proof}

\section{Stage A: the projection lower bound and three-dimensional coercivity}\label{sec:stageA}

\subsection{The abstract projection principle}\label{sec:spectral}

The projection-based lower bound used throughout this work originates in the
author's spectral-Galerkin framework \cite{LiuSpectral2026}, which extends the
finite-element projection method of \cite{LiuOishi2013,Liu2024} to conforming
spectral discretisations. We recall only what the present computation consumes.

Let $H$ be self-adjoint with compact resolvent and discrete spectrum
$0<\lambda_1\le\lambda_2\le\cdots$, and let $\widehat V_N$ be a conforming trial
space with Galerkin (Ritz) eigenvalues $\lambda_{k,N}$. The bound is stated for the \emph{variational} form of the eigenvalue problem, which is also
the setting in which the two norms appearing in it live. Let $\mathcal V\subset\mathcal W$ be
Hilbert spaces, let $a(\cdot,\cdot)$ be a symmetric bilinear form on $\mathcal V$ that is
bounded and coercive, and let $b(\cdot,\cdot)$ be a symmetric, positive bilinear form on
$\mathcal W$; one seeks the pairs $(\lambda,u)\in\R\times\mathcal V\setminus\{0\}$ with
\begin{equation}\label{eq:varev}
  a(u,v)=\lambda\,b(u,v)\qquad\text{for all }v\in\mathcal V .
\end{equation}
The two forms induce two norms: the \emph{$a$-norm} $\norm{v}_a=a(v,v)^{1/2}$, the energy norm
in which $\mathcal V$ is complete, and the weaker \emph{$b$-norm} $\norm{v}_b=b(v,v)^{1/2}$, in
which eigenfunctions are normalised. Coercivity makes $\norm{\cdot}_b$ dominated by
$\norm{\cdot}_a$ on $\mathcal V$, and $C_N$ below is any constant for which that comparison
holds on the projection error. In this paper $\mathcal W=L^2(\Om)$, $\mathcal V=H^1(\Om)$, $b$ is
the $L^2(\Om)$ inner product --- so $\norm{\cdot}_b=\norm{\cdot}_{L^2(\Om)}$ --- and $a$ is the
shifted form $a_\sigma$ of \Cref{lem:coercive}, whose coercivity is exactly what makes
$\norm{\cdot}_a$ a norm; the eigenvalues of \eqref{eq:varev} are then $\mu_k(\Om)+\sigma$.

If a projection
$\Pi_N:E_k\to\widehat V_N$ onto the span $E_k$ of the first $k$ exact
eigenfunctions satisfies the Pythagoras identity
$a(\Pi_N\phi,\phi-\Pi_N\phi)=0$ and the projection-error estimate
$\|\phi-\Pi_N\phi\|_b\le C_N\|\phi-\Pi_N\phi\|_a$, then
\begin{equation}\label{eq:liu-bound}
  \lambda_k\ \ge\ \frac{\lambda_{k,N}}{1+C_N^2\,\lambda_{k,N}},
  \qquad k=1,\dots,\dim\widehat V_N,
\end{equation}
a guaranteed lower bound requiring no a-priori information on a neighbouring
eigenvalue \cite[Thm.~2.1]{LiuSpectral2026}. Two features of the spectral
realisation matter here. First, the projection constant is \emph{global} and, for
a trial space of exact eigenfunctions, takes the closed-form optimal value
$C_N=\lambda_*^{-1/2}$ with $\lambda_*$ the first omitted eigenvalue---the
spectral gap at the truncation boundary---in contrast to the \emph{local},
mesh-size constant $C_h=O(h)$ of a non-conforming element. Second, for
$-\Delta+V$ with a genuinely unbounded $L^\infty$ norm, as here, the constant
$C_N$ cannot be used verbatim; \cite{LiuSpectral2026} already anticipates the
Coulomb case, controlling the singularity by a localised Hardy inequality, which
is the route this paper develops in full (\Cref{sec:hardy}). The coercive shift
$\sigma=\Ceps$ and the effective constant $\hat C^2$ established in this section
are exactly the ingredients that turn \eqref{eq:liu-bound} into
a guarantee for the Coulomb-singular operator.

\subsection{Localized Hardy inequality and coercivity in three dimensions}\label{sec:hardy}

The quadratic form $a(u,v)=(\nabla u,\nabla v)+(Vu,v)$ of the Coulomb operator is
not obviously coercive: the singular potential $-Z_{\rm tot}/|x-a_i|$ is unbounded
below. Coercivity is restored by a localized Hardy inequality that controls the
Coulomb singularity by the gradient, with two free cutoff parameters optimized
after the geometry is fixed.

\begin{lemma}[Localized 3D Hardy inequality]\label{lem:hardy}
Let $a\in\Omega$ with $\operatorname{dist}(a,\partial\Omega)\ge\ell$, where $\ell>0$ denotes the
nucleus-to-boundary clearance throughout ($\ell=6$ on $\Om_1$ and $8$ on $\Om_2$; the
letter $\rho$ is reserved for the Lehmann--Goerisch separator of \Cref{sec:stageB}). For every inner radius $r_1\in(0,\ell)$ and splitting parameter $\delta>0$ --- both artefacts of the cutoff rather than of the operator, chosen after $\ell$ by an offline minimization of the resulting shift (\Cref{sec:sharp-detail}) --- and all $v\in H^1(\Omega)$,
\begin{equation}\label{eq:hardy}
  \int_\Omega\frac{|v|^2}{|x-a|^2}\,dx
  \le C_{\rm grad}(\delta)\,\|\nabla v\|^2 + C_{L^2}(\delta,r_1)\,\|v\|^2 ,
\end{equation}
with $C_{\rm grad}(\delta)=4(1+\delta)$ and
$C_{L^2}(\delta,r_1)=r_1^{-2}+4(1+1/\delta)(\ell-r_1)^{-2}$.
\end{lemma}
\begin{proof}[Proof sketch]
Take a radial Lipschitz cutoff $\chi$ supported in $B(a,\ell)$ with $\chi\equiv1$ on $B(a,r_1)$ and $\chi$ affine in $|x-a|$ on the annulus (so $|\nabla\chi|\le(\ell-r_1)^{-1}$ a.e.; Lipschitz suffices because the Hardy inequality below is applied on $H^1_0(B(a,\ell))$, and $\chi v$ lies in that space). Outside $B(a,r_1)$, $|x-a|^{-2}\le r_1^{-2}$; inside, $v=\chi v$ and the classical Hardy inequality
$\int_{\R^3}|w|^2/|x-a|^2\le4\|\nabla w\|^2$ applies to $w=\chi v\in H^1_0$.
Young's inequality
$\|\nabla(\chi v)\|^2\le(1+\delta)\|\nabla v\|^2+(1+1/\delta)\|\nabla\chi\|_\infty^2\|v\|^2$
gives the two constants; the sharp Hardy constant $4$ is recovered as $\delta\to0$.
\end{proof}

The parameters $(r_1,\delta)$ are artefacts of the cutoff, not of the operator;
any admissible pair makes \eqref{eq:hardy} true, and they are chosen after $\ell$
by an offline minimization. Fixing $\delta=1$, $r_1=\ell/2$ gives the elementary
constants $C_{\rm grad}=8$, $C_{L^2}=36/\ell^2$, which overshoot $C_{\rm grad}$ by
a factor $2$ and are never optimal for the certified bound.

The admissible shift that \Cref{lem:hardy} yields on its own, through Young's
inequality on $\sum_iZ_i/\abs{x-a_i}$, is not the sharpest available from elementary
analysis, and since it enters the projection constant quadratically it is worth
improving before any computation is done. That chain passes from
$1/r$ to $1/r^2$ and then invokes \eqref{eq:hardy}, whose constant $4$ is sharp
for the Hardy inequality but is paid in full by the leading term of the
resulting shift. Going instead through the sharp \emph{one-centre} inequality
saves exactly that factor. Recall that for every $a>0$,
\begin{equation}\label{eq:sharp-one-centre}
  \int_{\R^3}\frac{|w|^2}{|x|}\,dx
  \;\le\; a\norm{\nabla w}_{L^2(\R^3)}^2+\frac{1}{4a}\norm{w}_{L^2(\R^3)}^2
  \qquad\forall\,w\in H^1(\R^3),
\end{equation}
which is precisely the statement $-a\Delta-|x|^{-1}\ge-\tfrac{1}{4a}$, with
equality approached on the hydrogenic ground state $e^{-|x|/(2a)}$. Applied
naively on $\Om$ this inequality fails, because the ground-state
representation leaves a boundary term over $\partial\Om$ that a trace
inequality controls only at a cost larger than that shift itself. A
cutoff removes the boundary term instead.

\begin{lemma}[Localized sharp form bound]\label{lem:sharp-form}
Let $\dist(a_i,\partial\Om)\ge\ell$ for every centre, and let
$0<r_1<r_2\le\ell$ and $\delta>0$. Then for all $\epsilon'>0$ and all
$v\in H^1(\Om)$,
\begin{equation}\label{eq:sharp-form}
  \ip{V^-v}{v}\;\le\;\epsilon'\norm{\nabla v}^2+\sigma_{\rm loc}(\epsilon')\norm{v}^2,
  \qquad
  \sigma_{\rm loc}(\epsilon')
  =\frac{Z_{\rm tot}^2(1+\delta)}{4\epsilon'}
  +\frac{(1+\delta^{-1})\,\epsilon'}{(1+\delta)(r_2-r_1)^2}
  +\frac{Z_{\rm tot}}{r_1}.
\end{equation}
\end{lemma}

\begin{proof}
Fix a centre $a_i$ with charge $Z_i$ and choose $\chi_i\in C^\infty(\R^3)$ with
$\chi_i\equiv1$ on $B_{r_1}(a_i)$, $\chi_i\equiv0$ outside $B_{r_2}(a_i)$ and
$\norm{\nabla\chi_i}_\infty\le(r_2-r_1)^{-1}$; since $r_2\le\ell$ we have
$B_{r_2}(a_i)\subset\Om$, so $\chi_iv$ extended by zero lies in $H^1(\R^3)$.
Writing $r_i=|x-a_i|$ and splitting at $r_1$,
\[
  Z_i\int_\Om\frac{|v|^2}{r_i}
  \;\le\; Z_i\int_{\R^3}\frac{|\chi_iv|^2}{r_i}+\frac{Z_i}{r_1}\norm{v}^2 ,
\]
because $\chi_i\equiv1$ where $r_i<r_1$ and $r_i^{-1}\le r_1^{-1}$ elsewhere.
Apply \eqref{eq:sharp-one-centre} to $w=\chi_iv$ with parameter $a_i>0$ and use
$\norm{\nabla(\chi_iv)}^2\le(1+\delta)\norm{\nabla v}^2
+(1+\delta^{-1})(r_2-r_1)^{-2}\norm{v}^2$ together with
$\norm{\chi_iv}\le\norm{v}$:
\[
  Z_i\int_\Om\frac{|v|^2}{r_i}
  \le Z_ia_i(1+\delta)\norm{\nabla v}^2
  +\Bigl[\frac{Z_ia_i(1+\delta^{-1})}{(r_2-r_1)^2}
        +\frac{Z_i}{4a_i}+\frac{Z_i}{r_1}\Bigr]\norm{v}^2 .
\]
Set $\epsilon_i:=Z_ia_i(1+\delta)$, so that
$Z_i/(4a_i)=Z_i^2(1+\delta)/(4\epsilon_i)$ and
$Z_ia_i=\epsilon_i/(1+\delta)$. Summing over the centres with
$\sum_i\epsilon_i=\epsilon'$ gives \eqref{eq:sharp-form} once
$\sum_iZ_i^2/\epsilon_i$ is minimized subject to that constraint, which occurs
at $\epsilon_i=\epsilon'Z_i/Z_{\rm tot}$ and yields
$\sum_iZ_i^2/\epsilon_i=Z_{\rm tot}^2/\epsilon'$.
\end{proof}

\begin{corollary}[closed-form optimum in $\delta$]\label{cor:sharp-form-delta}
Since $(1+\delta^{-1})/(1+\delta)=\delta^{-1}$ identically, \eqref{eq:sharp-form} reads
$\sigma_{\rm loc}=Z_{\rm tot}^2(1+\delta)/(4\epsilon')+\epsilon'/\bigl(\delta(r_2-r_1)^2\bigr)
+Z_{\rm tot}/r_1$, which is minimized over $\delta>0$ at
\begin{equation}\label{eq:delta-star}
  \delta^\star=\frac{2\epsilon'}{Z_{\rm tot}(r_2-r_1)},
  \qquad
  \sigma_{\rm loc}^\star(\epsilon')
  =\frac{Z_{\rm tot}^2}{4\epsilon'}+\frac{Z_{\rm tot}}{r_2-r_1}+\frac{Z_{\rm tot}}{r_1}.
\end{equation}
No numerical minimization in $\delta$ is required. On $\Om_2$, where $Z_{\rm tot}=2$ and the
optimal radii are $r_1=\ell/2=4$, $r_2=\ell=8$, \eqref{eq:delta-star} collapses to
$\sigma_{\rm loc}^\star(\epsilon')=1+1/\epsilon'$, giving $7/2$ at $\epsilon'=0.4$ and $13/3$
at $\epsilon'=0.3$ with $\delta^\star=1/10$ and $3/40$ respectively --- the values used
throughout. This is why the constant is insensitive to $\delta$ at first order: the choices
sit at a stationary point.
\end{corollary}
\begin{proof}
Write $\sigma_{\rm loc}=A\delta+B/\delta+C$ with $A=Z_{\rm tot}^2/(4\epsilon')$,
$B=\epsilon'/(r_2-r_1)^2$ and $C=Z_{\rm tot}^2/(4\epsilon')+Z_{\rm tot}/r_1$. On $\delta>0$ this
is strictly convex, $\sigma_{\rm loc}''=2B/\delta^3>0$, so the unique stationary point
$\delta^\star=\sqrt{B/A}=2\epsilon'/(Z_{\rm tot}(r_2-r_1))$ is the minimizer, with value
$2\sqrt{AB}+C=Z_{\rm tot}/(r_2-r_1)+Z_{\rm tot}^2/(4\epsilon')+Z_{\rm tot}/r_1$. The radii enter
only through $(r_2-r_1)^{-1}+r_1^{-1}$, minimized over $0<r_1<r_2\le\ell$ at $r_2=\ell$,
$r_1=\ell/2$.
\end{proof}

An admissible pair in the sense of \eqref{eq:sharp-form} is all the coercivity
step needs, and it fixes at the same time the single scalar through which the
potential enters the projection constant.

\begin{lemma}[Coercivity and the projection form factor]\label{lem:coercive}
Fix $\epsilon\in(0,1)$ and let $\sigma_\epsilon$ be a constant for which the form bound
\eqref{eq:sharp-form} holds at that $\epsilon$, that is
$(V^-v,v)\le\epsilon\norm{\nabla v}^2+\sigma_\epsilon\norm{v}^2$ for all $v\in H^1(\Om)$. Then
for every $\sigma\ge\sigma_\epsilon$,
\begin{equation}\label{eq:coercive}
  a_\sigma(v,v):=a(v,v)+\sigma\norm{v}^2\;\ge\;(1-\epsilon)\norm{\nabla v}^2
  +(\sigma-\sigma_\epsilon)\norm{v}^2
  \qquad\text{for all }v\in H^1(\Om).
\end{equation}
In particular, if $\sigma>\sigma_\epsilon$ then $a_\sigma$ is coercive on $H^1(\Om)$ with
constant $\min\{1-\epsilon,\ \sigma-\sigma_\epsilon\}$, hence positive definite there, as
required in \Cref{sec:contract}. The potential enters the projection bound of \Cref{thm:coulomb-bound} through a single
scalar, the \emph{potential form factor} $\eta_V$: any constant for which
\begin{equation}\label{eq:etaV-def}
  \abs{(Vw,v_N)_{L^2}}\;\le\;\eta_V\,\norm{w}_{L^2(\Om)}\,\norm{v_N}_{a_\sigma}
  \qquad\text{for all }w\in L^2(\Om),\ v_N\in V_N
\end{equation}
is admissible (\Cref{app:def:si-etaV}); in the $a_\sigma$-norm of \Cref{sec:spectral} it measures
how strongly multiplication by $V$ couples the trial space to its orthogonal complement;
together with the shift $\sigma$ it is the only route by which the singular potential reaches
the constant $\widehat C_N$. The Hardy inequality
of \Cref{lem:hardy} together with \eqref{eq:coercive} supplies the admissible value
\begin{equation}\label{eq:etaV}
  \eta_V=Z_{\rm tot}\Bigl(\tfrac{C_{\rm grad}}{1-\epsilon}
  +\tfrac{C_{L^2}}{\mu_{1,N}+\sigma}\Bigr)^{1/2},
\end{equation}
with $C_{\rm grad}$, $C_{L^2}$ the constants of \Cref{lem:hardy}. The sharper
\emph{tail-refined} constant \eqref{eq:etaV-tail}, which is what the computations of this paper
certify, is admissible in the same sense \eqref{eq:etaV-def} and is the constant
certified and used in every computation reported here, beginning with \Cref{tab:stagea}.
\end{lemma}
\begin{proof}
Since $a(v,v)=\norm{\nabla v}^2-(V^-v,v)$, the form bound gives
$a_\sigma(v,v)\ge\norm{\nabla v}^2-\epsilon\norm{\nabla v}^2-\sigma_\epsilon\norm{v}^2
+\sigma\norm{v}^2$, which is \eqref{eq:coercive}. If $\sigma>\sigma_\epsilon$ both coefficients
on the right are positive, so
$a_\sigma(v,v)\ge\min\{1-\epsilon,\sigma-\sigma_\epsilon\}\norm{v}_{H^1(\Om)}^2$; coercivity and
positive definiteness follow, with no case distinction for constant $v$. For the form factor, Cauchy--Schwarz gives $\abs{(Vw,v_N)}\le\norm{w}\,\norm{Vv_N}$, and
\eqref{eq:hardy} then bounds $\norm{Vv_N}^2\le Z_{\rm tot}^2\bigl(C_{\rm grad}\norm{\nabla v_N}^2
+C_{L^2}\norm{v_N}^2\bigr)$; \eqref{eq:coercive} and the Galerkin bound
$a_\sigma(v_N,v_N)\ge(\mu_{1,N}+\sigma)\norm{v_N}^2$ on $V_N$ give
$\|\nabla v_N\|^2\le(1-\epsilon)^{-1}\|v_N\|_{a_\sigma}^2$ and
$\|v_N\|^2\le(\mu_{1,N}+\sigma)^{-1}\|v_N\|_{a_\sigma}^2$.
\end{proof}

The shifts used in this paper satisfy the strict inequality $\sigma>\sigma_\epsilon$: the
certification of \Cref{sec:sharp} returns $\sigma=\Ceps^{\rm cert}$ as a verified \emph{upper}
bound on the optimal admissible constant at that $\epsilon$, so \eqref{eq:coercive} applies with
$\sigma_\epsilon$ that optimal value.

\begin{remark}[the Hardy-route shift, and why it is superseded]\label{rem:hardy-shift}
Combining \Cref{lem:hardy} with Young's inequality on $\sum_iZ_i/\abs{x-a_i}$ yields the admissible shift
\begin{equation}\label{eq:sigma}
  \sigma_{\rm H}:=\frac{\epsilon\,C_{L^2}}{C_{\rm grad}}
  +\frac{Z_{\rm tot}^2\,C_{\rm grad}}{4\epsilon}
\end{equation}
(derived in \Cref{app:si:hardy}, which also records what an admissible shift costs). It is never the
better choice. Minimizing \eqref{eq:delta-star} over $r_1$ as well gives
\begin{equation}\label{eq:sigma-loc-closed}
  \sigma_{\rm loc}^\star=\frac{Z_{\rm tot}^2}{4\epsilon}+\frac{4Z_{\rm tot}}{\ell}
  \qquad\text{at } r_1=\ell/2,\ r_2=\ell ,
\end{equation}
and $\sigma_{\rm loc}^\star\le\sigma_{\rm H}$ whenever $\ell\ge16\epsilon/(3Z_{\rm tot})$ ($1.07$ against
$\ell=8$ on $\Om_2$ at $\epsilon=0.4$); numerically $\sigma_{\rm loc}^\star<\sigma_{\rm H}$ throughout
$Z_{\rm tot}\in[1,6]$, $\ell\in[0.5,32]$, $\epsilon\in[0.1,0.9]$, with the ratio tending to $1/4$ as $\ell$
grows. That is the factor $4$ the detour through $1/r^2$ costs: the leading term of
\eqref{eq:sharp-form} is $Z_{\rm tot}^2(1+\delta)/(4\epsilon')$ against $Z_{\rm tot}^2C_{\rm grad}/(4\epsilon')$
in \eqref{eq:sigma}, with $C_{\rm grad}\ge4$. \eqref{eq:sigma} is retained only as the comparator of
\Cref{tab:stagea}; \Cref{lem:hardy} remains in use through the form factor \eqref{eq:etaV}.
\end{remark}

\noindent\emph{From here on $\sigma$ denotes the certified shift $\Ceps^{\rm cert}$ of \Cref{sec:sharp}}
(its variational characterization and the enclosure of its optimum are in \Cref{sec:sharp-detail}). It
takes $\sigma_{\rm loc}$ minimized over $(\delta,r_1,r_2)$ as its input, is admissible in the sense of
\eqref{eq:sharp-form}, and is used in every result below; \Cref{rem:analytic-shift} records what
$\sigma_{\rm loc}$ alone would give.

\subsection{The convergent projection lower bound}

Combining the coercivity of \Cref{lem:coercive} with the projection principle
\eqref{eq:liu-bound} turns the coarse per-mode estimate into a convergent
guaranteed bound.

\begin{theorem}[Convergent projection lower bound for the Coulomb operator]
\label{thm:coulomb-bound}
Let $\sigma$ be any admissible shift in the sense of \eqref{eq:sharp-form} and $\eta_V$ any form factor in the sense of \eqref{eq:etaV-def} --- the value \eqref{eq:etaV}, or the sharper tail-refined constant \eqref{eq:etaV-tail} used in the computations --- set $g:=(\mu_{1,N}+\sigma)^{-1/2}\eta_V$, and
let $\nu_*=\bigl((N+1)\pi/(2L_x)\bigr)^2$ be the first omitted Laplacian
eigenvalue of the cosine trial space $V_N$. Then for every $k\le\dim V_N$,
\begin{equation}\label{eq:coulomb-bound}
  \mu_k\ge\frac{\mu_{k,N}+\sigma}{1+\widehat C_N^2(\mu_{k,N}+\sigma)}-\sigma,
  \qquad \widehat C_N^2=\frac{1+g^2}{(1-\epsilon)\,\nu_*}.
\end{equation}
Here $\nu_*$ is the smallest Neumann Laplacian eigenvalue omitted from $V_N$; the stated value $((N+1)\pi/(2L_x))^2$ is that eigenvalue precisely when $L_x$ is the largest half-side of the box, which holds for both boxes used here. Since $\sigma$ and $\epsilon$ are independent of $N$, and $g$ depends on $N$ only through $\mu_{1,N}$ and is therefore bounded uniformly in $N$, while $\nu_*=O(N^2)$, the constant $\widehat C_N^2=O(N^{-2})$ and the gap $\mu_{k,N}-\mu_k$ is $O(\nu_*^{-1})=O(N^{-2})$ with no floor.
\end{theorem}

\begin{proof}[Proof sketch; full argument in
\Cref{app:si:proj,app:si:formfactor,app:si:proj-proof}]
Pass to the shifted operator $H'=H+\sigma$ with form
$a_\sigma(u,v)=a(u,v)+\sigma(u,v)$. By \Cref{lem:coercive},
$a_\sigma(v,v)\ge(1-\epsilon)\|\nabla v\|^2>0$, so $H'$ is symmetric
positive-definite with eigenvalues $\mu_k+\sigma>0$, and the abstract
lower-bound principle \eqref{eq:liu-bound} applies to the pencil $(a_\sigma,b)$ on
$V_N$:
\[
  \mu_k+\sigma\ \ge\ \frac{\mu_{k,N}+\sigma}{1+C_N^2\,(\mu_{k,N}+\sigma)},
\]
provided the projection constant $C_N$ bounds the $b$-seminorm of the projection
error by its $a_\sigma$-seminorm. Take $\Pi_N$ the modal projection onto $V_N$
(truncation in the cosine/Laplacian eigenbasis). The omitted modes carry
Laplacian Rayleigh quotient $\ge\nu_*$, whence
$\|(I-\Pi_N)v\|^2\le\nu_*^{-1}\|\nabla(I-\Pi_N)v\|^2$. Two conversions turn this
Laplacian estimate into an $a_\sigma$-estimate: coercivity replaces
$\|\nabla\cdot\|^2$ by $(1-\epsilon)^{-1}\|\cdot\|_{a_\sigma}^2$, and the Coulomb
cross-term---controlled through the form factor $\eta_V$ of \eqref{eq:etaV}---
contributes the factor $1+g^2$ with $g=(\mu_{1,N}+\sigma)^{-1/2}\eta_V$. Together
$C_N^2=\widehat C_N^2=(1+g^2)/((1-\epsilon)\nu_*)$. Substituting and subtracting
$\sigma$ gives \eqref{eq:coulomb-bound}. As $N\to\infty$, $\nu_*=O(N^2)$ while
$\sigma,\epsilon,g$ stay fixed, so $\widehat C_N^2\to0$ and the bracket closes at
the rate $O(N^{-2})$ without a floor.
\end{proof}

\subsection{Certified Stage-A values}

Every constant entering the bound above is certified in interval arithmetic
before it is used. \Cref{tab:stagea} collects them together with the Stage-A
bounds they produce and the two hypothesis checks that Stage~B consumes.

\begin{table}[t]\centering\small
\caption{Certified Stage A on $\Om_2$ ($\ell=8$, $Z_{\rm tot}=2$) at $N_{\rm stageA}=72$. The upper
block records the analytic constant $\sigma_{\rm loc}$ of \Cref{lem:sharp-form}, which is the input
to the certified shift; the Stage-A quantities in the lower block are computed at that certified shift, $\sigma=\Ceps^{\rm cert}=1.7656061570$ (\Cref{prop:ceps-cert}), with the form factor \eqref{eq:etaV-tail} re-certified there. These are Stage-A results in their own right; the separator that Stage~B uses is obtained independently (\Cref{sec:mu2}). \Cref{rem:analytic-shift} records what the same
quantities become if the certification stage is skipped and $\sigma_{\rm loc}$ is used directly. The parameters $r_1=\ell/2$, $r_2=\ell$ are used throughout; both
$(\epsilon,\delta)$ pairs are the stationary points of $\sigma_{\rm loc}$ in $\delta$ given by
\Cref{cor:sharp-form-delta}, so the
constant is insensitive to $\delta$ at first order and the exact rational choice is not
load-bearing. Column \emph{Hardy} is \eqref{eq:sigma} at the same $\epsilon$; column $-\eta_{1,N}$
is the auxiliary Rayleigh--Ritz value, a floating-point \emph{estimate} of $C_\epsilon^{\rm opt}$
that lies below it and is therefore inadmissible as a shift (\Cref{sec:sharp}). At these $(\epsilon,\delta)$ the constant is exact, so no enclosure of it is reported; evaluating
\Cref{lem:sharp-form} in interval arithmetic returns $7/2$ and $13/3$ to a width of
$2.7\times10^{-15}$, which checks the implementation rather than adding a result.}
\label{tab:stagea}
\begin{tabular}{ccccc}
\toprule
$\epsilon$ & $\delta$ & $\sigma_{\rm loc}$ & Hardy \eqref{eq:sigma} & $-\eta_{1,N}$ (estimate)\\
\midrule
$0.4$ & $1/10$  & $7/2$  & $10.65$ & $0.9046343989$\\
$0.3$ & $3/40$  & $13/3$ & $13.97$ & $1.0925408901$\\
\bottomrule
\end{tabular}

\vspace{1.2ex}
\begin{adjustbox}{max width=\linewidth}
\begin{tabular}{lll}
\toprule
quantity & certified value & status\\
\midrule
$\eta_V^2$ \eqref{eq:etaV-tail} & $\le0.55$ & certified; $0.50$ not verified$^\ddagger$\\
$\mu_{1,N}$ & $[-0.5518514996719721,\,-0.551851499658834]$ & width $1.3\times10^{-11}$\\
$\mu_{2,N}$ & $[-0.3370631931082278,\,-0.33706319309801763]$ & width $1.0\times10^{-11}$\\
$L_1$ & $\ge-0.578393280$ & $\epsilon=0.4$\\
$L_2$ & $\ge-0.373688180970$ & $\epsilon=0.4$\\
\midrule
separation $L_2>U_1$ & margin $+1.781633\times10^{-1}$ & \textsc{holds}\\
Goerisch shift $c=1$ & $\lambda_{\min}(H+cI)\ge0.421606720198$ & \textsc{admissible}\\
\bottomrule
\end{tabular}
\end{adjustbox}

\smallskip
\footnotesize
$^\ddagger$ The rung at $0.50$ is the last the Cholesky ladder of \Cref{tab:formfactor-all} does not
verify. That test is sufficient, not necessary --- it carries a rigorous perturbation margin --- so a
rung that fails to verify does not exclude the value. Only $\eta_V^2\le0.55$ is certified.
\end{table}

\begin{remark}[the analytic shift, if the certification stage is skipped]\label{rem:analytic-shift}
The analytic constant $\sigma_{\rm loc}=7/2$ of \Cref{lem:sharp-form} is admissible on its own. On $\Om_2$
at $N_{\rm stageA}=72$, $\epsilon=0.4$, with $\eta_V^2\in(0.456,0.458]$ re-certified at that shift, the same
Galerkin values give $L_1\ge-0.673936875164$, $L_2\ge-0.477163042351$ and $L_2-U_1=+7.468846\times10^{-2}$.
All are valid, and each lies roughly twice as far from the eigenvalue as the certified-shift entry of
\Cref{tab:stagea}. The weaker $\inf L_1$ also decides the Goerisch shift: at $N_{\rm sep}=32$, $c=1$ fails
under $\sigma_{\rm loc}$ ($\lambda_{\min}(H+I)\ge-0.0661$) and passes under the certified shift ($\ge0.3285$).
\end{remark}

\noindent
Two features of \Cref{tab:stagea} are worth drawing out. First, $\epsilon=0.4$ gives the better
bound at both indices, and one certified shift serves both: it is a property of the operator and the
space, not of the index.
Second, the admissibility of the Goerisch shift $c=1$ is verified rather than assumed, and
under the certified shift it is comfortable: the same check already passes at
$N_{\rm sep}=32$, where $\lambda_{\min}(H+I)\ge0.3285$ --- a valid figure at that resolution
because $V_{32}\subset V_{72}$ makes $\eta_V^2\le0.55$ admissible there and the $N=72$ Galerkin
value used in it is the conservative one. Under the analytic shift the same check \emph{fails}
at $N_{\rm sep}=32$ and only licenses $c=1$ at higher resolution
(\Cref{rem:analytic-shift}). Finally, note that $\sigma_{\rm loc}$ \emph{improves} on a larger truncation
box, since $\ell$ is the nucleus-to-boundary distance and $r_1=\ell/2$ grows with it: enlarging the box lowers it further. This does not help Stage A, because $\nu_*$ scales as
$L_x^{-2}$ and that loss dominates the gain in the constant. Because the sharp constant is
$\Ceps=-\eta$ of the \emph{actual} operator, no positivity assumption on $V$ enters anywhere: the
method never requires $V\ge0$ nor a regularized $V_\rho$, a negative part being admissible as long
as the shifted form $a_\sigma$ is positive definite, which \Cref{lem:coercive} guarantees.

\section{Stage B: Lehmann--Goerisch with a self-generated separator}\label{sec:stageB}

\label{sec:lg}
The second stage sharpens the ground-state lower bound to Rayleigh--Ritz quality.
We state the classical theorem in the bilinear-form setting of Goerisch and
collaborators \cite{BeattieGoerisch1995}.

\subsection{The bound}
\label{sec:lg-bound}

\emph{Assumptions.} (A1) $D$ is a real vector space with symmetric bilinear forms
$M,N$, $M$ positive definite. (A2) There is a sequence of eigenpairs
$(\lambda_i,\varphi_i)\in\R\times D$ with $M(\varphi_i,v)=\lambda_i N(\varphi_i,v)$
for all $v\in D$, $M$-orthonormal, and
$N(v,v)=\sum_i\lambda_i|N(v,\varphi_i)|^2$. (A3) $X$ is a real vector space with a
symmetric positive semidefinite form $b_G$, and $T:D\to X$ is linear with
$b_G(Tu,Tv)=M(u,v)$. (A4) Given $v_1,\dots,v_n\in D$, the $w_i\in X$ satisfy
$b_G(w_i,Tv)=N(v_i,v)$ for all $v\in D$. (A5) Given $\rho>0$, form the $n\times n$
matrices $A_0=(M(v_i,v_j))$, $A_1=(N(v_i,v_j))$, $A_2=(b_G(w_i,w_j))$, and
\begin{equation}\label{eq:A-B}
  A:=A_0-\rho A_1,\qquad B:=A_0-2\rho A_1+\rho^2 A_2 ,
\end{equation}
with $B$ positive definite; let $\nu_1\le\cdots\le\nu_n$ solve $Az=\nu Bz$ and let
$q$ be the number of negative $\nu_k$.

\begin{theorem}[Lehmann--Goerisch; \cite{Liu2024}]\label{thm:lg}
Under A1--A5, for each $k$ with $1\le k\le q$ the interval
$[\,\rho-\rho/(1-\nu_k),\,\rho\,)$ contains at least $k$ eigenvalues
$\lambda_i$, counting multiplicity.
\end{theorem}
A clear statement with a complete proof is given in \cite{Liu2024}; the result has a longer history in \cite{Lehmann1963,GoerischHaunhorst1985,BeattieGoerisch1995}. 

\paragraph{A checkable condition on $\rho$ replacing the count of negative $\nu_k$.}
As stated, \Cref{thm:lg} is conditional in two ways that are awkward in practice: the
positive-definiteness of $B$ required by (A5) has to be established separately, and the
range of $k$ is governed by $q$, a count of negative $\nu_k$ that is known only after the
matrix pencil has been solved. Both are removed by a single inequality on $\rho$
\cite{LiuPlum2025}. Assume from here that $A_1$ is positive definite and let
\begin{equation}\label{eq:Lambda-n}
  \Lambda_n := \text{the maximal eigenvalue of } A_0x=\Lambda A_1x ,
\end{equation}
that is, the largest Rayleigh quotient attained on $\mathrm{span}\{v_1,\dots,v_n\}$.

\begin{lemma}[$\rho$ forces $B\succ0$]\label{lem:B-positive}
If $\rho>\Lambda_n$, then $B$ of \eqref{eq:A-B} is positive definite, so the
positive-definiteness requirement in (A5) holds automatically.
\end{lemma}
A proof is given in \cite{LiuPlum2025}; the statement is standard in the
Lehmann--Goerisch literature.

\begin{theorem}[Lehmann--Goerisch with an unconditional index range]\label{thm:lg-sharp}
Let $\rho\le\lambda_{m+1}$ for some $m\ge n$. If A1--A5 hold, then
\begin{equation}\label{eq:lg-bounds}
  \lambda_{m+1-k}\ \ge\ \rho-\frac{\rho}{1-\nu_k}\qquad(1\le k\le q).
\end{equation}
If in addition $\rho>\Lambda_n$, then $\nu_n<0$, hence $q=n$, and \eqref{eq:lg-bounds}
holds for every $1\le k\le n$.
\end{theorem}
Both claims are proved in \cite{Liu2024}; the second consumes \Cref{lem:B-positive}.

Applied to $k=1$, the left endpoint is a rigorous lower bound for $\lambda_1$. For a single trial
vector it has a closed form worth recording, because several later arguments are immediate from it:
eliminating $\nu_1$ from \eqref{eq:A-B} and \eqref{eq:lg-bounds} by the substitution
$u=\rho A_2-A_1$, which turns $B$ into $\rho u-(\rho A_1-A_0)$ and cancels the two $\rho$'s, leaves
\begin{equation}\label{eq:lg-closed}
  \rho-\frac{\rho}{1-\nu_1}\;=\;\frac{\rho A_1-A_0}{\rho A_2-A_1},
  \qquad
  \frac{\partial}{\partial A_2}\frac{\rho A_1-A_0}{\rho A_2-A_1}
  \;=\;-\frac{\rho\,(\rho A_1-A_0)}{(\rho A_2-A_1)^2}\;<\;0 .
\end{equation}
The bound is thus a M\"obius function of $A_2$, strictly decreasing in it, and non-decreasing in
$\rho$ because $\partial_\rho$ of the same expression is $(A_0A_2-A_1^2)/(\rho A_2-A_1)^2\ge0$ with
$A_0A_2\ge A_1^2$ by Cauchy--Schwarz: a valid but slightly conservative separator $\rho$ still gives
a guarantee, merely a less sharp one.

\paragraph{For the first eigenvalue, Lehmann's bound is Temple's bound.} When only the first eigenvalue is
bounded, with a single trial vector, Lehmann's bound \emph{is} Temple's inequality \cite{Temple1928},
the lower bound most often used for the ground state of a Schr\"odinger operator. To see this, let
$H'$ be the positive definite operator of the pencil, with eigenvalues $\kappa_1<\kappa_2\le\cdots$, and
take for (A1)--(A4) the pair $M(u,v)=\langle H'u,H'v\rangle$, $N(u,v)=\langle H'u,v\rangle$, which has the
same eigenpairs, together with $X=L^2$, $b_G=\langle\cdot,\cdot\rangle$ and $T=H'$. Then $w_1=v_1$
satisfies (A4) exactly, and for a normalized $v=v_1\in D(H')$ the moments are $A_0=\norm{H'v}^2$,
$A_1=\bar\kappa:=\langle H'v,v\rangle$ and $A_2=1$. With these, \eqref{eq:lg-closed} reads
\begin{equation}\label{eq:lg-temple}
  \frac{\rho A_1-A_0}{\rho A_2-A_1}
  \;=\;\bar\kappa-\frac{\norm{H'v}^2-\bar\kappa^2}{\rho-\bar\kappa}
  \;=\;\bar\kappa-\frac{\norm{(H'-\bar\kappa)v}^2}{\rho-\bar\kappa}\ \le\ \kappa_1
  \qquad(\bar\kappa<\rho\le\kappa_2),
\end{equation}
which is Temple's inequality. Both the residual $\norm{(H'-\bar\kappa)v}$ and the gap $\rho-\bar\kappa$
are unchanged by the shift $H'=H+c$, so \eqref{eq:lg-temple} is Temple's inequality for $H$ itself.
The realization used in Stage~B differs only in the third moment: it takes $A_2\ge\langle v,H'^{-1}v\rangle$
for the pair $(\langle H'\cdot,\cdot\rangle,\langle\cdot,\cdot\rangle)$, which gives Temple's inequality in
weak form, preferred on the box for the reason given in \Cref{rem:temple-box}. \Cref{rem:hierarchy}
places both within the classical hierarchy of lower bounds.

\begin{remark}[position within the Weinstein--Temple hierarchy]\label{rem:hierarchy}
For a Rayleigh--Ritz upper bound $\bar\lambda$ with residual norm $\sigma$ and a
separator $\rho$, the classical lower bounds order as
\begin{equation}\label{eq:hierarchy}
  \underbrace{\bar\lambda-\sigma}_{\text{Weinstein},\,O(\sigma)}
  \ \le\
  \underbrace{\bar\lambda-\tfrac{\sigma^2}{\rho-\bar\lambda}}_{\text{Temple},\,O(\sigma^2)}
  \ \approx\
  \underbrace{L^{\mathrm{LG}}}_{\text{Lehmann--Goerisch},\,O(\sigma^2)}
  \ \le\ \lambda_1\ \le\ \underbrace{\bar\lambda}_{\text{RR upper}} ,
\end{equation}
so Temple and Lehmann--Goerisch are both second order in the residual while
Weinstein is only first order; for one trial vector Lehmann's bound is Temple's inequality, exactly
in the form \eqref{eq:lg-temple} and in weak form in the realization used here (\Cref{rem:temple-box}). This is why the second stage is worth its cost, and why the achievable width is governed by $\rho$ and by
the residual moment $A_2$ rather than by $\bar\lambda$. Temple's inequality itself is applied directly, with a global trial on $\RR^3$,
in \Cref{sec:global}. A fuller comparison of which residual moment controls which bound is outside the
scope of this paper.
\end{remark}

\paragraph{What the pipeline checks.} \Cref{thm:lg-sharp} reduces the hypotheses to two inequalities on
$\rho$ alone: the separation condition $\rho\le\lambda_{m+1}$, which the method certifies for itself
(\Cref{sec:mu2}), and $\rho>\Lambda_n$, which is computable from $A_0,A_1$ before the pencil is formed;
the $\nu_k$ are then outputs, not hypotheses. Every certificate in \Cref{sec:results} uses a single
trial vector, for which the second condition reads $\rho>A_0$, the Rayleigh quotient of the trial
vector, and is equivalent to $\nu_1<0$, since $\nu_1=(A_0-\rho)/B$ with $B>0$ (\Cref{lem:B-positive}). The auxiliary vectors $w_i$ need
not satisfy (A4) exactly: an approximate solve yields a valid bound once its Goerisch residual is
enclosed (\Cref{sec:contract}). All of these are
statements about exactly computed quantities; a floating-point evaluation of the same formulas is not
a bound, and is used in this paper only to choose which cells to certify (\Cref{rem:float-not-bound}).

\subsection{Certifying the separator directly}
\label{sec:mu2}
The Neumann Lehmann--Goerisch bound for $\mu_1(\Om)$ needs a separator with
\begin{equation}\label{eq:separation}
  \mu_1(\Om)\;<\;\rho\;\le\;\mu_2(\Om)\quad\bigl(\le\lambda_2(\R^3)\bigr),
\end{equation}
and it sharpens as $\rho$ rises toward $\mu_2$; the last inequality, \Cref{thm:enc-main} at $k=2$, is
needed only by whole-space trial functions (\Cref{sec:global}). The pipeline does not import $\rho$: it
certifies a lower bound on $\mu_2(\Om)$ itself, using the reflection symmetry of the box.

\begin{lemma}[parity decomposition; classical]\label{lem:parity-decomp}
Let $\Om=\prod_{i=1}^3(-L_i,L_i)$ and let $V$ be invariant under each coordinate
reflection $R_i$. Let $G\cong(\mathbb Z/2)^3$ be the group they generate. Then the
Neumann operator $\mathcal H=-\Delta+V$ on $\Om$ decomposes as
$\mathcal H=\bigoplus_{p\in\{\mathrm e,\mathrm o\}^3}\mathcal H_p$, where
$\mathcal H_p$ is the restriction of $\mathcal H$ to the range of
$P_p=\tfrac18\sum_{g\in G}\chi_p(g)U_g$, $(U_gu)(x)=u(g^{-1}x)$ and $\chi_p$ the
character with $\chi_p(R_i)=+1$ if $p_i=\mathrm e$ and $-1$ if $p_i=\mathrm o$.
Consequently the eigenvalue multiset of $\mathcal H$, with multiplicity, is the
disjoint union of the eigenvalue multisets of the eight $\mathcal H_p$.
\end{lemma}

\begin{theorem}[$\mu_2$ without a sector identification]\label{thm:mu2-candidates}
Under the hypotheses of \Cref{lem:parity-decomp},
\[
  \mu_2(\Om)=\min\bigl\{\lambda_2^{\mathrm{eee}},\ \lambda_1^{\mathrm{oee}},\
  \lambda_1^{\mathrm{eoe}},\ \lambda_1^{\mathrm{eeo}}\bigr\}.
\]
If in addition $L_2=L_3$ and the nuclei lie on the $x_1$-axis, then
$\lambda_1^{\mathrm{eeo}}=\lambda_1^{\mathrm{eoe}}$ and the minimum is over three
distinct quantities.
\end{theorem}

The proofs are in \Cref{app:stageB-parity}. \Cref{thm:mu2-candidates} turns the separator into a
computation: a lower bound for each candidate is a statement about one block's operator, and their
minimum bounds $\mu_2(\Om)$ with no identification of the sector that attains it. Stage~B therefore
runs in five steps.
\begin{enumerate}
  \item \emph{Seed.} \Cref{thm:coulomb-bound} on the oee block at $k=2$ gives a certified lower bound
        on $\lambda_2^{\mathrm{oee}}$.
  \item \emph{The oee ground state.} A single-vector Lehmann--Goerisch step on the oee block, with the
        bound of step~1 as its separator, gives a certified lower bound $L(\lambda_1^{\mathrm{oee}})$.
  \item \emph{The other candidates.} \Cref{thm:coulomb-bound} on the eee block at $k=2$ and on the eoe
        block at $k=1$ gives $L(\lambda_2^{\mathrm{eee}})$ and $L(\lambda_1^{\mathrm{eoe}})$.
  \item \emph{The separator.} $\rho:=\min\{L(\lambda_2^{\mathrm{eee}}),\,L(\lambda_1^{\mathrm{oee}}),\,
        L(\lambda_1^{\mathrm{eoe}})\}\le\mu_2(\Om)$ by \Cref{thm:mu2-candidates}.
  \item \emph{The ground state.} The single-vector Lehmann--Goerisch step on the eee block with this
        $\rho$ gives $L^{\rm Neu}_{1,\rm LG}\le\mu_1(\Om)\le\lambda_1(\R^3)$. The driver asserts
        $B\succ0$, $\nu_1<1$ and $\rho>\Lambda_1$ at run time and refuses the certificate otherwise.
\end{enumerate}
On both boxes the oee candidate is the binding one, giving $\rho=-0.3561459574$ on $\Om_1$ and
$\rho=-0.3723080919$ on $\Om_2$ from steps~1--2 at $N=32$, $N'=48$, $\epsilon'=0.65$, $c=1$
(\Cref{tab:headline-tiers}); on $\Om_1$ the eoe candidate of step~3 is evaluated only approximately
(\Cref{app:stageB-cand}). Using the Stage-A value $L_2$ directly as $\rho$, and so skipping steps~1--2,
is weaker at every configuration tried and violates the hypotheses of \Cref{thm:lg-sharp} at two of
them; \Cref{app:stageB} records the status of the three candidates (\Cref{tab:sepfree}) and compares the
two routes (\Cref{tab:mu3-needed}).

\section{The exact/approximate contract}\label{sec:contract}

A guaranteed bound tolerates approximation in quantities that affect only its
sharpness, but not in quantities entering its proof. \Cref{tab:audit} classifies every quantity of the pipeline. The central subtlety is the Goerisch
constraint, which admits two rigorous routes: (Route~1) enforce the constraint
exactly and allow the minimization to be approximate; (Route~2, used here) solve
the constraint approximately and enclose its residual exactly, so the guaranteed
matrix $A_2$ is recovered from an inexact solve. The verified eigensolver
\texttt{lehmann\_behnke} implements Route~2; The rest of this section gives the identity, the realization
built on it and its certificates in full, since this is the one step where rigor could silently be
lost.

\label{sec:a4}
The single point where the Lehmann--Goerisch method could silently lose rigor is
assumption (A4): given the trial vector $v$, it requires the auxiliary vector
$w=H'^{-1}v$ solving $b_G(Tw,\phi)=N(v,\phi)$ for all $\phi$, where
$H'=H+cI$ is the shifted operator, made symmetric positive definite by a constant
$c>0$. The certificate ultimately consumes only
\begin{equation}\label{eq:A2def}
  A_2=\langle v,\,H'^{-1}v\rangle=\langle v,w\rangle .
\end{equation}
With the Coulomb term $-\sum_i 1/|x-a_i|$ represented by a \emph{dense} interval
moment matrix (every basis pair couples through the Laplace representation
$1/r=\tfrac{2}{\sqrt\pi}\int_0^\infty e^{-r^2t^2}\,dt$), there is no closed form
for $H'^{-1}v$ and no finite exact arithmetic that inverts a $D\times D$ dense
interval matrix ($D$ up to $68{,}921$ here). So $w$ can only be computed
approximately, and (A4) cannot hold pointwise. We satisfy it instead \emph{in the
enclosure sense}, by the Goerisch defect identity.

\paragraph{The identity.} Let $\tilde w$ be an approximate solution of $H'w=v$
(conjugate gradients on the floating-point midpoint matrix $\mathrm{mid}(H')$),
and form the interval residual $r=v-H'\tilde w$ using the \emph{full interval}
operator. For any SPD $H'$ and any $\tilde w$, the following is an exact algebraic
identity:
\begin{equation}\label{eq:goerisch-star}
  \langle v,H'^{-1}v\rangle
  = 2\langle\tilde w,v\rangle-\langle\tilde w,H'\tilde w\rangle
    +\langle r,H'^{-1}r\rangle .
\end{equation}
(Expand $\langle r,H'^{-1}r\rangle$ with $r=v-H'\tilde w$ and use self-adjointness
$\langle H'\tilde w,H'^{-1}v\rangle=\langle\tilde w,v\rangle$; the $\tilde w$-terms
cancel.) Every term except the last is directly computable in interval
arithmetic. The single non-computable term, the \emph{defect energy}
$\langle r,H'^{-1}r\rangle$, is exactly the contribution of the still-missing part
of the auxiliary vector, $e=H'^{-1}r=w-\tilde w$---``the part of $w$ that a final
exact solve would add.'' Enclosing it from above is what a Goerisch realization does, and the
realization used here does so on $H^1(\Om)$ rather than on the auxiliary trial space.

\textbf{The realization.}\ The construction keeps $D=H^1(\Om)$ and $M=a_c$ and replaces the
Goerisch space. Fix $\epsilon'\in(0,1)$ and an admissible $C_{\epsilon'}$ in the sense of
\Cref{lem:sharp-form}, and require $c\ge C_{\epsilon'}$. Take
\[
  X=(L^2)^3\times L^2\times H^1(\Om),\qquad
  Tu=\bigl\{\sqrt{1-\epsilon'}\,\nabla u,\ \sqrt{c-C_{\epsilon'}}\,u,\ u\bigr\},
\]
\[
  b_G(\{p,s,q\},\{p',s',q'\})=(p,p')+(s,s')
   +\bigl[\epsilon'(\nabla q,\nabla q')+((V+C_{\epsilon'})q,q')\bigr].
\]
Then $b_G(Tu,Tv)=M(u,v)$ identically, and $b_G\succeq0$ on all of $X$ --- the first two blocks
trivially, the third exactly by admissibility of $C_{\epsilon'}$. Positivity on $X$, not merely
on $T(D)$, is what \Cref{lem:B-positive} requires. Taking $q=\tilde w$,
$s=\sqrt{c-C_{\epsilon'}}\,\tilde w$ and solving $b_G(w,Tv)=(v_i,v)$ for $p$ gives
$p=\sqrt{1-\epsilon'}\,\nabla\tilde w+\nabla\varphi/\sqrt{1-\epsilon'}$ with
$-\Delta\varphi=R$ under Neumann conditions, where $\langle R,v\rangle=(v_i,v)-a_c(\tilde w,v)$
is the \emph{weak} residual; $\varphi$ exists because $1\in V_{N'}$ forces
$\langle R,1\rangle=0$ exactly. The resulting admissible value is
\begin{equation}\label{eq:goerisch-corrected}
  A_2^G=2\langle\tilde w,v\rangle-\langle\tilde w,\mathcal H'\tilde w\rangle
        +\frac{\norm{\nabla\varphi}^2}{1-\epsilon'}\ \ge\ \langle v,\mathcal H'^{-1}v\rangle ,
\end{equation}
and the excess over $\langle v,\mathcal H'^{-1}v\rangle$ equals
$\langle R,[(-\Delta)^{-1}/(1-\epsilon')-\mathcal H'^{-1}]R\rangle$, whose non-negativity is
again precisely the condition $b_G\succeq0$. Since $\tilde w\in V_{N'}$ and
$(-\Delta+c)\tilde w\in V_{N'}$, the weak residual splits orthogonally as
$R=r-(I-\Pi^0_{N'})(V\tilde w)$, so
\[
  \norm{\nabla\varphi}^2\le\frac{\norm{r}^2}{\nu_1}+\frac{\norm{(I-\Pi^0_{N'})(V\tilde w)}^2}{\nu_*(N')}:
\]
the
first term encloses the algebraic inexactness of the solve, and the second is computable from
the $V^2$ moments, whose two Gaussian factors collapse to a single Gaussian at the combined
centre and width so that the verified moment rule of \Cref{sec:moments} applies unchanged.
Dividing the discretization part by $\nu_*(N')$ rather than by
$\lambda_{\min}(\mathcal H')$ is what keeps its price small; what remains of that price, and how it
depends on $N'$, is the subject of \Cref{rem:inexact-solve}.

\begin{remark}[the auxiliary solve, and the role of $N'$]\label{rem:inexact-solve}
The vector $\tilde w$ is obtained by conjugate gradients on $\operatorname{mid}(\widehat H'_{N'})$, where
$\widehat H'_{N'}=\widehat H_{N'}+cI$ has dimension up to ${\sim}6.9\times10^4$ per sector. The residual
$r=v-\widehat H'_{N'}\tilde w$ is then formed with the full interval operator. Since
\eqref{eq:goerisch-corrected} holds for any $\tilde w\in V_{N'}$, an inexact solve can only widen the
certificate. With $\norm{r}\sim10^{-13}$, its share of $A_2^G$ is ${\sim}10^{-24}$, against ${\sim}10^{-10}$
from interval rounding. The structural constraints are $V_N\subseteq V_{N'}$, hence $N'\ge N$, and
$1\in V_{N'}$, which gives $\langle R,1\rangle=0$ on the Neumann side. Beyond these, $N'$ affects only
sharpness, through the excess $A_2^G-\langle v,\mathcal H'^{-1}v\rangle$. Its dominant part
$\norm{(I-\Pi^0_{N'})(V\tilde w)}^2/[(1-\epsilon')\nu_*(N')]$ falls at a measured rate $N'^{-2.83}$
($L=14$, $N=48$, $\epsilon'=0.8$, $c=1$); a theoretical rate is outside the scope of this paper, and $N'=N$
remains admissible. The gate quantities of the reported configurations are in \Cref{tab:headline-tiers}.
\end{remark}

\begin{remark}[the contract in full]\label{rem:contract-full}
\label{app:si:audit}
\Cref{tab:audit} classifies every quantity as \textsc{Exact} (a verified interval enclosure is required for
the guarantee) or \textsc{Approx} (floating point suffices and affects only sharpness).
\begin{table}[t]\centering\small
\caption{Exact/approximate audit of the certified pipeline.}
\label{tab:audit}
\begin{tabular}{@{}lp{0.80\linewidth}@{}}
\toprule
\textsc{Exact} & moments $G(\kappa,t)$, with certified Gauss--Legendre nodes and Bernstein-ellipse
remainder; matrix entries $A_0,A_1,A_2$; shift $\sigma$ and Hardy constants; diagnostic Ritz value
$\eta_{1,N}$; Ritz values $\mu_{1,N},\mu_{2,N}$ and the separation test; separator $\rho=L_2$; projection
constant $\widehat C_N$; the test $B\succ0$; $\nu_k$ and the M\"obius transform\\
\textsc{Approx} & trial vectors $u_i$; coercivity parameter $\epsilon$; cutoff pair $(r_1,\delta)$;
Goerisch inner solve (its residual is enclosed exactly)\\
\bottomrule
\end{tabular}
\end{table}
\end{remark}

\begin{remark}[the Dirichlet sector: where the boundary condition lives]\label{rem:dir-aux}
For the Dirichlet problem $D=H^1_0(\Om)$, and (A4), $b_G(w,Tv)=(v_i,v)$ for all $v\in D$, imposes no
boundary condition on $w$: its admissible set is $\phi_0+\{h:\Delta h=0\}$. The implementation takes
$\phi\in H^1_0(\Om)$. This satisfies (A4) and is conservative, since the correction enters $A_2$
additively and \eqref{eq:lg-closed} decreases in $A_2$; on $\Om_1$ and $\Om_2$ it costs under six per
cent of bracket width. The free energy minimizer cannot be used: its estimator needs
$\langle R,1\rangle=0$, which on the Neumann side follows from (A4) with $v=1\in V_{N'}$ but is not
forced here, because $1\notin H^1_0(\Om)$.
\end{remark}

\section{Verified realization}\label{sec:realization}

\subsection{Cosine-spectral discretization}
We discretize on an anisotropic box with Neumann boundary conditions in a tensor
cosine basis, which diagonalizes the kinetic term and, through a parity
decomposition of the mode indices, splits the ground state ($\mu_1$) and the
first excited state ($\mu_2$) into separate symmetry sectors, reducing the
degrees of freedom by a large factor. Assembly therefore rests on a single one-dimensional shifted-Coulomb moment, enclosed rigorously in \Cref{sec:moments} without any error-function implementation.

\subsection{Verified linear algebra}
The assembled interval matrices are passed to verified eigenvalue routines that
return enclosures of $\mu_1$, $\mu_2$, the auxiliary $\eta$, and the
Lehmann--Goerisch quantities. The treatment of the Goerisch constraint, and the classification of every quantity in the pipeline as exact or approximate, are given in \Cref{sec:contract}.

\subsection{Rigorous moment integrals}\label{sec:moments}

Through $1/|x-a|=(2/\sqrt\pi)\int_0^\infty e^{-t^2(x-a)^2}\,dt$, discretized by a Gauss--Legendre grid in $t$, each matrix element reduces to the one-dimensional shifted-cosine moment
$G(\kappa,t)=\int_{-L}^{L}\cos(\kappa(x+L))e^{-t^2(x-s)^2}dx$. Because interval
\texttt{erfcx} is unavailable, $G$ is enclosed with a two-regime scheme using only
interval $\exp$, $\cos$, $\sqrt{\,}$ plus two analytic error terms.

\paragraph{Large $t$ ($t\ge t_\star=1$).} The infinite-domain integral is
elementary,
\[
  G_\infty=\cos(\kappa(s+L))\,\frac{\sqrt\pi}{t}\,e^{-\kappa^2/4t^2},
\]
and the truncation to $[-L,L]$ is controlled by a Gaussian tail bounded below
$2\times10^{-29}$ at $t_\star$.

\paragraph{Small $t$ ($t<t_\star$).} A composite verified Gauss--Legendre rule on
$[-L,L]$ ($n_p=24$ nodes, 48 panels) with a per-panel Bernstein-ellipse remainder
$|E_n|\le(64/15)\,h\,M(\rho_{\rm B})/(\rho_{\rm B}^{2n}(\rho_{\rm B}^2-1))$ gives total remainder
$\le10^{-13}$. The Gauss--Legendre nodes and weights are themselves certified by
interval Newton iteration on the Legendre polynomial (BigFloat seeds refined to
$\sim10^{-70}$, outward-rounded to Float64 intervals).

\paragraph{Validation.} The moment enclosure was verified to contain the
double-precision (\texttt{erfcx}) value at every grid point: \textbf{18{,}624
tests, zero failures, maximum enclosure width $1.9\times10^{-12}$}.

\subsection{Verified assembly and eigenvalue bounds}
\label{sec:assembly}
For each parity sector the interval Hamiltonian $\hat H=P+K$ is formed as a dense
interval matrix via the Kronecker structure of the tensor
basis. Floating-point Ritz vectors (Lanczos) seed the verified
\texttt{lehmann\_behnke} solver, whose verified core acts on the small $V'\hat HV$
projections and is therefore independent of the matrix dimension. The certified
ground-energy bracket is $\lambda_1\in[\inf(L_1),\sup(\mu_{1,N})]$ with the sharp
lower bound
\begin{equation}\label{eq:sharpL}
L_k=\frac{m_k+\sigma}{1+\hat C^2(m_k+\sigma)}-\sigma,\quad \sigma=\Ceps,\ m_k=\inf(\mu_k^N),
\end{equation}
$\hat C^2=(1+g^2)/((1-\epsilon)\nu_*(N))$, $g=\eta_V/\sqrt{m_1+\sigma}$ and $\nu_*(N)=((N+1)\pi/(2L_x))^2$. The form factor is the \emph{tail-refined} constant of \cite[Lemma~4.5]{LiuSpectral2026},
\begin{equation}\label{eq:etaV-tail}
  \eta_V \;=\; \sup_{0\ne v_N\in V_N}  \frac{\norm{(I-\Pi^0_N)Vv_N}_{L^2}}{\norm{v_N}_{a_\sigma}} ,
\end{equation}
which is the sharpest constant admissible in \Cref{app:def:si-etaV}, and for a Coulomb potential the only one of the available constants that is finite, since $Vv_N\in L^2$ while $V\notin L^\infty$. It is evaluated as the largest generalized eigenvalue of $\bigl(Q-P^2,\,K+P+\sigma M\bigr)$ with $Q_{\bm i\bm j}=(V^2\varphi_{\bm i},\varphi_{\bm j})$, and certified in interval arithmetic by a Cholesky positivity test on $\mu S-Q+P^2\succeq0$: on $\Om_2$ at $N_{\rm stageA}=72$ and $\sigma=\Ceps^{\rm cert}$ the certified value is $\eta_V^2\le0.55$, with the trial value $0.50$ failing, so the constant lies in $(0.50,\,0.55]$.

The same constant enters the certified chain at two further places, and it is certified separately at each: on
the trial space of the inner Lehmann--Goerisch separator (\Cref{sec:mu2}; the oee sector at $N'$), where the
projection bound must place $\rho$ above the sector's ground value, and in the rescaled operator $B=-\Delta+V/\epsilon$ of \Cref{prop:ceps-cert}, where it enters the certified shift itself; \Cref{rem:why-scale} explains why that rescaling is needed rather than cosmetic, and \Cref{rem:shift-cost} what moving from an inadmissible shift to a certified one costs. \Cref{tab:formfactor-all} lists the certified brackets for every
configuration of \Cref{tab:certified}; re-certifying $\Ceps^{\rm cert}$ with the tail-refined constant
of $B$ returns a \emph{smaller} value than the one used at every geometry, so every shift used here is
admissible under a certified constant.

\begin{table}[t]\centering\small
\caption{Certified tail-refined form factor $\eta_V^2$ \eqref{eq:etaV-tail} at every place it enters the certified chain on $\Om_2$ and the three-centre configurations; on $\Om_1$ the Stage-A ladder at $N=72$ gives $(0.50,0.52]$. Brackets are (last rung not verified, first rung verified] of the Cholesky ladder
$\mu S-(Q-P^2)\succeq0$. Only the right endpoint is certified: the ladder test is sufficient and
carries a rigorous perturbation margin, so a rung it does not verify is not thereby excluded, and
the bracket is a record of the search rather than an enclosure of $\eta_V^2$. Stage A:
eee sector at $N_{\rm stageA}$, admissible iff $c+\inf L_1>0$. Separator: oee sector at $N'$; the budget is the
largest $\eta_V^2$ for which the projection bound still places $\rho$ above the sector's Galerkin ground value.
$\Ceps$: the rescaled operator $B=-\Delta+V/\epsilon$ at $N_{\rm aux}$, eee; the pivot is the value of the constant
that reproduces the $\Ceps^{\rm cert}$ in use, so a bracket below it certifies that shift as admissible. All entries are at the configuration's own shift $\sigma=\Ceps^{\rm cert}$, as is \Cref{tab:stagea}. Every verified rung, the $\Om_1$ one included, was re-checked with the a~posteriori Cholesky backward-error bound of \Cref{app:stageA-sep}.}
\label{tab:formfactor-all}
\begin{tabular}{llccccccc}
\toprule
 & & \multicolumn{2}{c}{Stage A (eee)} & \multicolumn{3}{c}{separator (oee)} & \multicolumn{2}{c}{$\Ceps$ ($B$, eee)}\\
\cmidrule(lr){3-4}\cmidrule(lr){5-7}\cmidrule(lr){8-9}
system & box & $N$ & $\eta_V^2$ & $N'$ & $\eta_V^2$ & budget & $\eta_{V,B}^2$ & pivot\\
\midrule
H$_2^+$ & $\Om_2$, $L=10$ & 72 & $(0.50,0.55]$ & 48 & $(0.50,0.55]$ & $2.38$ & $(2.6,2.8]$ & $2.87$\\
H$_3^{2+}$, $d=2.8$ & $L=12$ & 76 & $(0.40,0.45]$ & 86 & $(0.40,0.45]$ & $2.59$ & $(2.2,2.4]$ & $3.41$\\
H$_3^{2+}$, $d=4$   & $L=12$ & 64 & $(0.40,0.45]$ & 92 & $(0.40,0.45]$ & $2.41$ & $(2.2,2.4]$ & $3.30$\\
\bottomrule
\end{tabular}
\end{table}

\subsection{Parity-sector reduction}
The cosine (and sine) bases split by mode
parity into eight symmetry sectors. The ground state lives in the (even,even,even) sector (\Cref{lem:aux-sector}), and by \Cref{thm:mu2-candidates} the second eigenvalue is the least of three sector quantities, of which the (odd,even,even) ground level is the binding one at every configuration reported here. Assembling one sector at a time reduces both the degrees of freedom (by $8\times$) and the dense one-sector matrix--vector cost (by $[(N{+}1)^3/(N/2{+}1)^3]^2\approx58\times$ at $N=64$), which is what allows the $N=64$ dense
Dirichlet solve to fit the $251$\,GB host.

\begin{remark}[what a block computation concludes about]\label{rem:sector-role}
A bound proved on one parity block concerns that block's operator. For the ground state the
identification with $\mu_1(\Om)$ is proved (\Cref{lem:aux-sector}). For the second level it is not, and
$\mu_{2,N}$ in \Cref{tab:stagea} is the oee-block ground Ritz value. \Cref{thm:mu2-candidates} removes
this assumption wherever all three candidates are certified (\Cref{rem:sector-assumed}); the costs are
given in \Cref{app:choices:sectors}.
\end{remark}
\section{Truncation error and the truncation gap}
\label{sec:optimalbox}

The truncation theorem of \Cref{sec:truncation} bounds the box-to-space error but
does not fix the box, and the width of a certified enclosure is the sum of that
truncation error and the discretization error of the computation carried out on
the box. The two move in opposite directions as the box grows. We exhibit both contributions on a
pair of nested boxes and quantify the gap between them; how the boxes of
\Cref{tab:configs} were chosen, and the balance that fixes them, is in \Cref{app:choices:box}.

\subsection{Two boxes decompose the error}
\label{sec:trunc}

\Cref{tab:certified} collects the certified whole-space enclosures on the two boxes. The
enclosure width has two sources, interval/discretization error and domain-truncation error, and
holding both a Neumann lower and a Dirichlet upper bracket on the same box separates them: on
$\Om_2$ the interval widths are $\sim10^{-11}$ while the enclosure stalls at $1.38\times10^{-3}$,
the discrete Dirichlet--Neumann Ritz gap, which is nearly $N$-independent, so no increase of $N$
can break it. Doubling every half-side collapses truncation but costs a $\sim4\times$ smaller
spectral-gap factor $\nu^\ast$, leaving a residual dominated by the Dirichlet side
(\Cref{app:choices:bias}).

\subsection{Two-sided Dirichlet brackets and the sign of the truncation bias}
\label{sec:dir}
Domain monotonicity makes the first Dirichlet eigenvalue on a box a rigorous \emph{upper} bound
for $\lambda_1(\R^3)$; a lower bound on that same box eigenvalue brackets it two-sidedly and
quantifies the truncation bias. Both sides are certified for the operator. The upper side is an interval Rayleigh--Ritz value; for the
lower side we take a single-test-vector Lehmann--Goerisch step in the $H^1_0$ (sine) basis on the ooo
sector, in the realization of \Cref{sec:a4} carried over to that sector, so that the
hypotheses are verified on $H^1_0(\Om)$ and the endpoint bounds $\lambda_1^{\mathrm{Dir}}(\Om)$ rather
than a discretization of it. \Cref{rem:dir-aux} records what changes in the construction, the
boundary condition belonging to the form domain and not to the auxiliary variable. The separator is
the one the method generates for itself, the certified second level of the same ooo sector:
$\rho=-0.1615297083$ on $\Om_1$ and $\rho=-0.2830522488$ on $\Om_2$. No external spectral input
enters, and in particular the eoo-sector route to $\lambda_2^{\mathrm{Dir}}$ is not needed.

\begin{table}[t]\centering\small
\caption{Two-sided \emph{Dirichlet} eigenvalue brackets $\lambda_1^{\mathrm{Dir}}\in[L^{\mathrm{Dir}},
U^{\mathrm{Dir}}]$ on the two boxes at $N=48$, $N'=64$, against the certified Neumann upper endpoint
$\sup\mu_{1,N}$ of \Cref{tab:stagea}. Both endpoints are certified for the operator: the upper one is
an interval Rayleigh--Ritz value, and the lower one is the Lehmann--Goerisch realization of
\Cref{sec:a4} carried over to the Dirichlet sector, whose hypotheses are verified on $H^1_0(\Om)$ and
not on a discrete subspace (\Cref{rem:dir-aux}). The last column is therefore a certified statement.
Since $\lambda_1^{\mathrm{Dir}}(\Om)\ge\lambda_1(\R^3)\ge\mu_1(\Om)$, a positive gap separates the two
brackets and exhibits the truncation error; a negative one means they overlap, truncation having fallen
below the width the method resolves at this resolution. Every certificate reports $\rho>\Lambda_n$
(hence $B\succ0$ and $\nu_1<0$) with conjugate-gradient residual $\sim10^{-13}$, and each uses the
separator the method generates for itself, the certified second level of its own sector. Shift $c=1$ and $\epsilon'=0.65$ throughout; $\epsilon'=0.65$ is the smallest admissible grid value and the sharpest, and the gap stays positive over the whole grid $0.65\le\epsilon'\le0.90$ ($7.58$ to $5.43\times10^{-3}$).}
\label{tab:dir}
\begin{adjustbox}{max width=\linewidth}
\begin{tabular}{lccccc}
\toprule
box & $L^{\mathrm{Dir}}$ & $U^{\mathrm{Dir}}$ & width & $\sup\mu_{1,N}$
  & gap $=L^{\mathrm{Dir}}-\sup\mu_{1,N}$\\
\midrule
$\Om_1$ ($L=8$)  & $-0.5482494847$ & $-0.5473483990$ & $9.011\times10^{-4}$ & $-0.5558289240$
  & $+7.579\times10^{-3}$\\
$\Om_2$ ($L=10$) & $-0.5528533308$ & $-0.5507405805$ & $2.113\times10^{-3}$ & $-0.5518514997$
  & $-1.002\times10^{-3}$\\
\bottomrule
\end{tabular}
\end{adjustbox}
\end{table}

\paragraph{The truncation gap.} \Cref{fig:twosided} shows the two-sided brackets for both boundary
conditions on both boxes, every endpoint certified. The two boxes fall on opposite sides of the
resolution threshold, and \Cref{tab:dir} states it as a single number. On $\Om_1$ the Dirichlet
bracket lies entirely above the Neumann one, the gap being
$L^{\mathrm{Dir}}-\sup\mu_{1,N}=+7.579\times10^{-3}$, so the truncation error is exhibited by
certified quantities alone: $\lambda_1^{\mathrm{Dir}}(\Om_1)\ge-0.5482494847$ while
$\mu_1(\Om_1)\le-0.5558289240$, and no computation at any resolution can place the two box
eigenvalues closer. On $\Om_2$ the same difference is $-1.002\times10^{-3}$: the brackets overlap, not
because truncation has vanished but because it has fallen below the width the method resolves --- the
true gap there is $1.111\times10^{-3}$ against a certified Dirichlet bracket of
$2.113\times10^{-3}$, so the bracket is $1.9$ times too wide to fit inside it, where on $\Om_1$ it
fits with a factor $9.4$ to spare. Enlarging the box therefore buys accuracy in $\lambda_1(\R^3)$ and
loses the ability to see truncation at all; the two effects are separated, and quantified, in
\Cref{app:choices:bias}.

\begin{figure}[t]\centering
\includegraphics[width=0.74\textwidth]{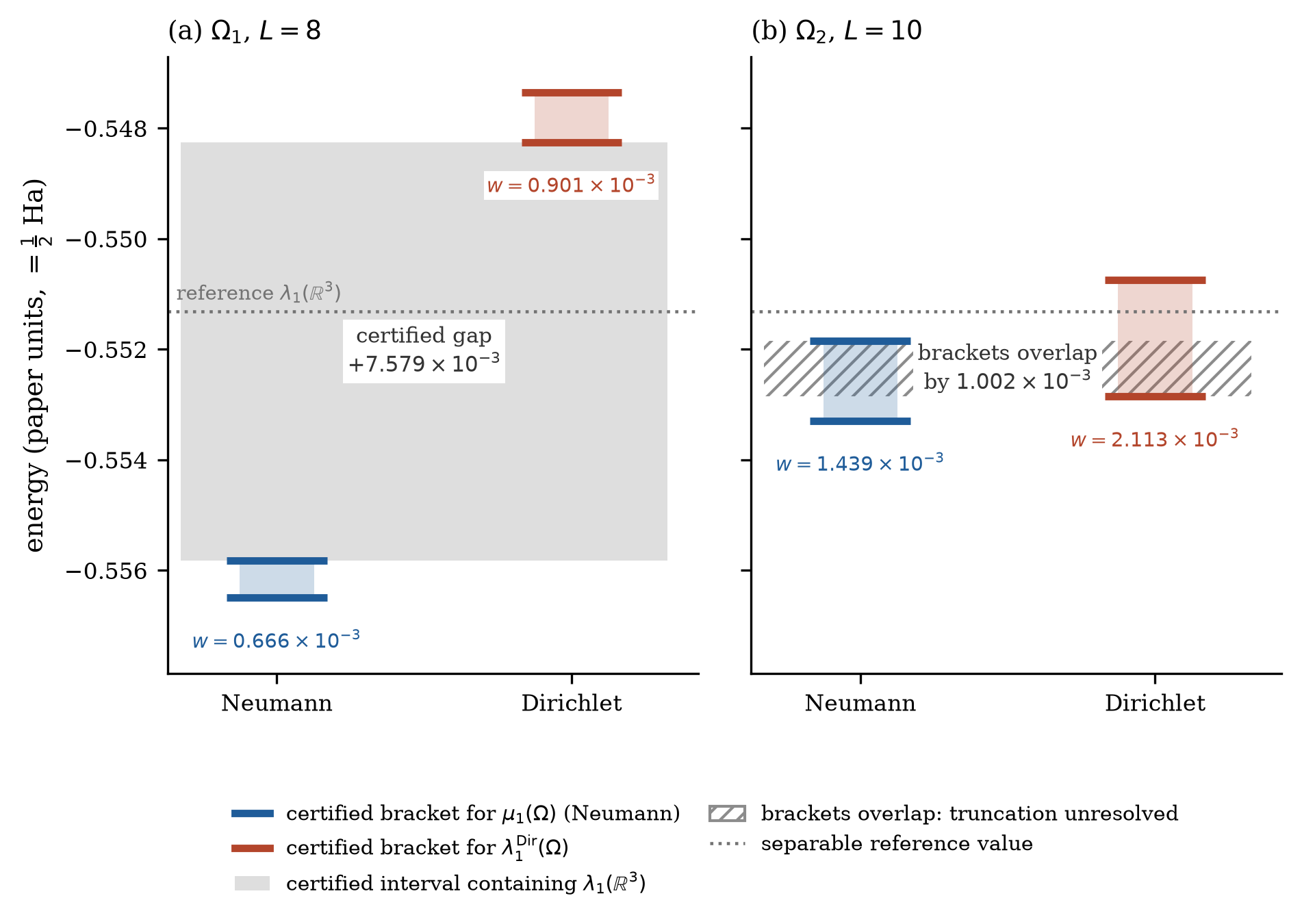}
\caption{Two-sided brackets for \emph{both} boundary conditions on the two truncation boxes, in the units
of \Cref{sec:units}. The Neumann brackets are at $N=64$, $N'=80$ and the Dirichlet brackets at
$N=48$, $N'=64$, the same pair of resolutions on both boxes. \textbf{Every endpoint shown is certified for the
operator}: upper endpoints are interval Rayleigh--Ritz values, lower endpoints the
Lehmann--Goerisch realization of \Cref{sec:a4} --- on the Neumann side as stated there, on the
Dirichlet side carried over to $H^1_0(\Om)$ as in \Cref{rem:dir-aux}. Since
$\lambda_1^{\mathrm{Dir}}(\Om)\ge\lambda_1(\RR^3)\ge\mu_1(\Om)$, the region between the two brackets
is where $\lambda_1(\RR^3)$ must lie. \textbf{(a)} On $\Om_1$ the brackets are disjoint and that region
is a certified interval of width $7.579\times10^{-3}$ (shaded), so truncation is exhibited by certified
quantities alone. \textbf{(b)} On $\Om_2$ they overlap by $1.002\times10^{-3}$: truncation has fallen
below the width the method resolves, the true separation $1.111\times10^{-3}$ being narrower than the
certified Dirichlet bracket $2.113\times10^{-3}$. Widths are printed beneath each bracket. \textbf{Both panels share one
vertical scale and range}, so the shaded region of (a) and the hatched overlap of (b) may be compared
directly: the certified gap on $\Om_1$ is $7.6$ times the overlap on $\Om_2$. The dotted line is the
reference value \eqref{eq:lam1exact}, which lies inside every bracket pair shown.
}
\label{fig:twosided}
\end{figure}

\section{Certified results}\label{sec:results}

\begin{table}[t]\centering\small
\caption{The configurations reported here. Each fixes a system and geometry, a truncation box
$\Om_L=[-L,L]\times[-0.8L,0.8L]^2$, a production basis size $N$ and an auxiliary size $N'$; the
Stage-A and separator resolutions $N_{\rm stageA}$ and $N'_{\rm sep}$ are chosen independently of $N$
(\Cref{sec:cost}) and are listed because the text refers to them. The two H$_2^+$ boxes are named by
the boxes themselves, $\Om_1$ and $\Om_2$ of \eqref{eq:boxes}; the three-centre configurations carry
the letter D and are reported in the Galerkin tier (\Cref{sec:results-h3}). Further boxes appear in the box-size study of
\Cref{app:choices:box} and are not reported as certified enclosures.}
\label{tab:configs}
\begin{tabular}{@{}llcccl@{}}
\toprule
label & system, geometry & $L$ & $N/N'$ & $N_{\rm stageA}$ / $N'_{\rm sep}$ & tiers reported\\
\midrule
$\Om_1$ & H$_2^+$, $a_{1,2}=(\mp2,0,0)$ & $8$ & $64/80$ & $72$ / $48$ & both\\
$\Om_2$ & H$_2^+$, $a_{1,2}=(\mp2,0,0)$ & $10$ & $64/80$ & $72$ / $48$ & both\\
\addlinespace
D1 & H$_3^{2+}$, $d=2.8$ & $12$ & $64/80$ & $76$ / $86$ & Galerkin only\\
D3a & H$_3^{2+}$, $d=4$ & $12$ & $32/48$ & $64$ / $92$ & Galerkin only\\
D3c & H$_3^{2+}$, $d=4$ & $12$ & $24/40$ & $64$ / $92$ & Galerkin only\\
\bottomrule
\end{tabular}
\end{table}

\emph{Certified status.} For the operator-certified rows of \Cref{tab:certified} every hypothesis of the chain is verified for the operator itself: the Lehmann--Goerisch realization of \Cref{sec:a4} satisfies A1--A5 for $\mathcal H'$ on $H^1(\Om)$, and the separator carries no sector assumption: at every operator-certified configuration the three
candidates of \Cref{thm:mu2-candidates} are certified separately and $\rho$ is their minimum
(\Cref{tab:sepfree}), so no identification enters. One disclosure remains. At the three-centre configurations (A4) was not verified on $H^1(\Om)$, so the
value the Lehmann--Goerisch step returns there is \emph{not} a certified lower bound on $\mu_1(\Om)$
or on $\lambda_1(\RR^3)$, and those rows are not reported as bounds; the reason is in
\Cref{app:choices:real}.
The three admissible shifts were each run end to end on $\Om_2$; they are compared in
\Cref{app:choices:shift}.

\subsection{H\texorpdfstring{$_2^+$}{2+}: certified enclosures}
\label{sec:results-h2}

At $L=10$, $N=64$, $N'=80$ --- the box $\Om_2$ of \eqref{eq:boxes} --- the operator-certified
enclosure is
\begin{equation}\label{eq:headline}
\lambda_1 \;\in\; [\,-0.5532907722271889,\; -0.5507405804827743\,],
\qquad W = 2.550192\times10^{-3},
\end{equation}
computed with the Lehmann--Goerisch realization of
\Cref{sec:a4}: every hypothesis A1--A5 is verified for $\mathcal H'$ on $H^1(\Om)$.
\Cref{tab:headline-tiers} lists the certificate quantities of that run. The width is $0.463\%$ of $\abs{\lambda_1}$, inside the adequacy criterion adopted
here, $5\%$ of $\abs{\lambda_1}$, by a factor eleven.

\begin{table}[t]\centering\small
\caption{The certificate for the headline configuration $\Om_2$ ($L=10$, $N=64$, $N'=80$,
$D_{\rm aux}=68\,921$): the reported enclosure and the quantities whose verification licenses it
(\Cref{sec:a4}). Energies in the units of \Cref{sec:units} unless marked; margins are to the
reference \eqref{eq:lam1exact}.}
\label{tab:headline-tiers}
\begin{adjustbox}{max width=\linewidth}
\begin{tabular}{ll}
\toprule
enclosure            & $[-0.5532907722,\; -0.5507405805]$\\
width $W$            & $2.550192\times10^{-3}$ \ ($5.1004\times10^{-3}$\,Ha, $138.8$\,meV)\\
$W/\abs{\lambda_1}$  & $0.463\%$\\
margin below / above & $1.9737\times10^{-3}$ / $5.7653\times10^{-4}$\\
\midrule
separator $\rho$     & $-0.3723080919$ (certified oee $\mu_2$, from the $N{=}32/N'{=}48$ stage of \Cref{sec:mu2})\\
$A_2$                & $2.2334026755$\\
$B$ (outer)          & $7.273\times10^{-2}>0$\\
$\nu_1$ (outer)      & $-2.4682430<1$\\
admissibility gate   & $c=1.0\ge\Ceps'=0.9722$ at $\epsilon'=0.65$\\
\bottomrule
\end{tabular}
\end{adjustbox}
\end{table}

The two boxes' widths, $9.146\times10^{-3}$ on $\Om_1$ against $2.550\times10^{-3}$ on $\Om_2$, are
set by the auxiliary discretization term, which is divided by $\nu_*(N')\propto L_x^{-2}$ and so
costs more on the larger box (\Cref{app:choices:real}).

\subsection{Linear H\texorpdfstring{$_3^{2+}$}{3(2+)}: a three-centre test with
a singularity at the box centre}\label{sec:results-h3}

\paragraph{Realization tier of the H$_3^{2+}$ rows.} D1, D3a and D3c are reported in the Galerkin
tier only: the Lehmann--Goerisch stage is not evaluated for them in the realization of
\Cref{sec:a4}, so their lower endpoints are not certified lower bounds on $\lambda_1(\RR^3)$. Their
admissibility input is certified --- $\Ceps'$ at $N_{\rm aux}=92$ is admissible at $\epsilon'=0.90$
against the recorded $c=1.5$ on both boxes --- and every gate applied to them passed. 

\paragraph{Geometry.}
Linear H$_3^{2+}$ has no bound total-energy minimum at zero field: the
electronic eigenvalue deepens monotonically as the chain contracts while the total energy $E_{\mathrm{elec}} + 2.5/d$ increases without bound as the chain contracts toward the united-atom limit and decreases monotonically toward the dissociation limit and lies above the H$_2^+(R{=}2)+\mathrm{H}^+$ dissociation
channel at every spacing, in agreement with earlier variational studies of the stability of
H$_3^{2+}$ in linear and triangular arrangements \cite{Conroy1964,TurbinerLopezSolis1999,MedelCobaxin2008}.
The object certified here is the electronic
eigenvalue of the fixed-nuclei operator \eqref{eq:H}, which is a genuine
isolated eigenvalue below an essential spectrum beginning at $0$ --- exactly as
for H$_2^+$, and exactly the object the method addresses. We therefore select
two geometries by stated criteria rather than by minimisation: $d=4$ (adjacent
spacing equal to the H$_2^+$ equilibrium) and $d=2.8$ (stronger three-centre
coupling).

Reference values are needed because linear H$_3^{2+}$ admits no separable coordinate system and
hence no analytic benchmark; they are $\lambda_1 = -0.762080024$ at $d=4$ and $-0.936459577$ at
$d=2.8$, both with uncertainty $\pm 3.8\times10^{-7}$. How they were obtained, validated and given
that uncertainty is in \Cref{app:choices:h3}.

\begin{table}[!tp]
\centering
\caption{Certified enclosures of $\lambda_1(\RR^3)$, in the units of \Cref{sec:units}. The two
H$_2^+$ rows use the Lehmann--Goerisch stage realized on $H^1(\Om)$; their lower endpoints are
certified lower bounds on $\lambda_1(\RR^3)$. The three-centre rows are reported in the Galerkin
tier, so their lower endpoints are \emph{not} bounds on $\lambda_1$ and are reported only to show that the pipeline
transfers to a three-centre geometry. ``Margin'' is the distance from the lower endpoint to the
reference --- \eqref{eq:lam1exact} for H$_2^+$, the variational reference of \Cref{sec:results-h3}
for H$_3^{2+}$. $c$ is the Goerisch shift and $\epsilon'/\Ceps'$ the admissibility pair, with
$c\ge\Ceps'$ and $c$ never raised to make a condition hold.}
\label{tab:certified}
\footnotesize\setlength{\tabcolsep}{3pt}
\begin{tabular}{@{}llcccc r r r@{}}
\toprule
\multicolumn{9}{l}{\textit{Certified: lower bounds on $\lambda_1(\RR^3)$}}\\
 & system & $L$ & $N/N'$ & $c$ & $\epsilon'\,/\,\Ceps'$ & width & margin & cost\\
\midrule
$\Om_1$ & H$_2^+$ & 8 & 64/80 & $1.0$ & $0.65\,/\,0.9100$ & $9.146048\times10^{-3}$ & $5.177\times10^{-3}$ & 5.2\,h\\
\textbf{$\Om_2$} & \textbf{H}$_\mathbf{2}^\mathbf{+}$ & \textbf{10} & \textbf{64/80} & $1.0$ & $0.65\,/\,0.9722$ & $\mathbf{2.550192\times10^{-3}}$ & $\mathbf{1.974\times10^{-3}}$ & 5.3\,h\\
\addlinespace
\multicolumn{9}{l}{\textit{Transferability only --- not bounds on $\lambda_1$}}\\
 & system & $L$ & $N/N'$ & $c$ & $D_{\mathrm{aux}}$ & width & margin$^\dagger$ & cost\\
\midrule
D1          & H$_3^{2+}$, $d{=}2.8$ & 12 & 64/80 & $1.5$ & 68\,921 & $2.400853\times10^{-4}$ & $8.479\times10^{-5}$ & 4.3\,h\\
D3a$^\star$ & H$_3^{2+}$, $d{=}4$ & 12 & 32/48 & $1.5$ & 15\,625 & $3.131155\times10^{-3}$ & $2.397\times10^{-3}$ & 2.0\,h\\
D3c$^\star$ & H$_3^{2+}$, $d{=}4$ & 12 & 24/40 & $1.5$ &  9\,261 & $7.186526\times10^{-3}$ & $5.953\times10^{-3}$ & 2.0\,h\\
\bottomrule
\end{tabular}

\smallskip
\footnotesize
$^\dagger$ For the three-centre rows ``margin'' is a diagnostic, not a certified quantity, their
lower endpoints not being certified lower bounds on $\lambda_1$.
$^\star$ Recorded at $c=1.50$; an admissible shift reaches only $c=1.50$ on this geometry at an
affordable Stage-A resolution. These two rows are valid certificates reported at a larger $c$. Reaching $c=1.00$ would need a Stage-A resolution of $120$ (against $64$), and even $\rho=\lambda_2$ leaves them $2.17$ and $2.23$ times wider than at $c=1.00$, so no width ratio is reported.
\end{table}

\paragraph{Results.}
\Cref{tab:certified} reports two certified H$_2^+$ enclosures and three three-centre configurations, every one produced with an admissible coercivity shift. The narrowest
three-centre interval, D1 at $d=2.8$, is
\begin{equation}
  [\,-0.9365443657837298,\; -0.9363042804864575\,], \qquad W = 2.400853\times10^{-4},
\end{equation}
which is $0.0256\%$ of $\abs{\lambda_1}$; its lower endpoint is not a certified bound, for the
reason given at the start of this subsection. All five intervals contain their reference values. D3a and D3c differ only in resolution ($N=32$ and $N=24$) and are
included to show that $\rho>\Lambda_n$, and hence $B\succ0$ and $\nu_1<0$, holds \emph{across} basis size on the
three-centre operator rather than at a single fortunate configuration.

We emphasise what this does and does not claim. It does not claim that
H$_3^{2+}$ is computed as sharply as H$_2^+$; the point of the second system is
transferability, and the relevant statement is that a three-centre operator
with a singularity at the box centre is handled by the same mathematics,
the same quadrature and the same driver, to well inside the accuracy such a
model warrants.

\subsection{Cost}\label{sec:cost}

\begin{table}[t]
\centering
\caption{Per-stage cost of the two operator-certified runs. Peak memory is the maximum over the
stage, and the total row reports the maximum over stages, not a sum. Interval assembly is not a
separate row because each stage assembles the operator it needs. Values are read from the per-stage records of each configuration's certificate. Stage A and the separator run at resolutions decoupled from
the production basis, $N_{\rm stageA}=72$ and $N'_{\rm sep}=48$ on both boxes (\Cref{sec:mu2}).
``Stage A'' is the projection lower bound, ``separator'' the auxiliary-sector Lehmann--Goerisch run
producing $\rho$, ``Lehmann--Goerisch'' the fully-even bound on the Neumann side, and ``Dirichlet''
the two-sided Dirichlet bracket. The two runs used different thread counts, so wall times are not
directly comparable between the boxes; peak memory is.}
\label{tab:cost}
\begin{tabular}{lrrrr}
\toprule
 & \multicolumn{2}{c}{$\Om_1$ ($L{=}8$)} & \multicolumn{2}{c}{$\Om_2$ ($L{=}10$)} \\
\cmidrule(r){2-3}\cmidrule(l){4-5}
stage & time (s) & peak (GB) & time (s) & peak (GB) \\
\midrule
certified $\Ceps$      & 3\,898 & 638.2 & 3\,558 & 638.2 \\
Stage A (projection)   & 2\,371 & 422.5 & 3\,102 & 422.5 \\
separator (oee)        &    534 & 422.5 &    724 &  31.3 \\
Lehmann--Goerisch      & 9\,629 & 422.5 & 9\,396 & 123.5 \\
Dirichlet (both sides) & 2\,283 & 241.7 & 2\,298 & 241.7 \\
\midrule
total                  & 18\,716 & 638.2 & 19\,078 & 638.2 \\
\bottomrule
\end{tabular}
\end{table}

Verified computation is routinely dismissed as too expensive. The measured picture (\Cref{tab:cost}) is more favourable than the reputation, and one observation is responsible for most of the difference.

End to end the two operator-certified runs take $5.2$ and $5.3$ hours with a $638$\,GB peak, the peak
being set by the certified coercivity stage rather than by the Lehmann--Goerisch stage. Where that time goes, the
thread-exact assembly that made it possible, and one caveat on reading \Cref{tab:cost} are in
\Cref{app:choices:cost}.
\section{Outlook: global trial functions}\label{sec:global}

The truncation box is forced by generality. When a separator and the residual of a trial are both
available in closed form, one trial on $\RR^3$ suffices, and Temple's inequality, the second-order
member of \eqref{eq:hierarchy}, gives bounds limited only by the trial. We show this for the H$_2^+$
configuration of \Cref{sec:units}; proofs and details are in \Cref{app:global}. The separator of
\Cref{lem:even-sep}, $-\tfrac14$ (i.e.\ $-\tfrac12$\,hartree), is the one used for H$_2^+$ in
\cite{GHT1978}, where it is taken from the quoted rise of the second even level with $R$; the lemma proves it directly. Bazley and Fox split the H$_2^+$ Hamiltonian into two hydrogen-like operators
for the same purpose \cite{BazleyFox1963}.

\begin{lemma}[an analytic separator for the even sector]\label{lem:even-sep}
Let $\mathcal H_+=\{u\in L^2(\RR^3):u(-x)=u(x)\}$ and let $H_+$ be the restriction of $H$ to
$\mathcal H_+$. Then $\inf\sigma(H_+)=\lambda_1(\RR^3)$, and the second min--max value of $H_+$
satisfies $\lambda_2(H_+)\ge-\tfrac14$.
\end{lemma}

\begin{theorem}[Temple's bound for an even global trial]\label{thm:temple-global}
Let $0\ne\psi\in D(H)\cap\mathcal H_+$, $\bar\lambda=(H\psi,\psi)/\norm{\psi}^2$, and
$\sigma_+^2\ge\norm{(H-\bar\lambda)\psi}^2/\norm{\psi}^2$. If $\bar\lambda<-\tfrac14$, then
\begin{equation}\label{eq:temple-global}
  \bar\lambda-\frac{\sigma_+^2}{-\tfrac14-\bar\lambda}\ \le\ \lambda_1(\RR^3)\ \le\ \bar\lambda .
\end{equation}
\end{theorem}

\paragraph{Trial and result.} With $a=2$ and prolate coordinates $\xi=(r_1+r_2)/(2a)$,
$\eta=(r_1-r_2)/(2a)$, $r_k=\abs{x-a_k}$, take
\begin{equation}\label{eq:prolate-trial}
  \psi=e^{-p(\xi-1)}P(\xi,\eta),
\end{equation}
with $p>0$ rational and $P$ a polynomial, even in $\eta$, with rational coefficients. Then
$\psi\in D(H)$ is even in each coordinate, and every moment is closed form: $\norm{\psi}^2$ and
$\bar\lambda$ are exact rationals and a residual bound $\sigma_+^2$ is an exact rational combination of
$1$, $\gamma+\log(4p)$ and $e^{4p}E_1(4p)$, the only two constants enclosed with outward rounding
(\Cref{app:global,app:prolate}). \Cref{tab:global} shows the width falling by more than three orders
of magnitude for each ten additional radial degrees. For the $1\,701$-term trial,
\begin{equation}\label{eq:global-enclosure}
\begin{aligned}
  \lambda_1\ &\ge\ -0.5513171072\,4747323075\,4484472659\,1705800970\,3763951332\,405114,\\
  \lambda_1\ &\le\ -0.5513171072\,4747323075\,4484472659\,1705800970\,3763951331\,662658,
\end{aligned}
\end{equation}
a width below $7.5\times10^{-51}$ ($1.5\times10^{-50}$\,hartree), certified in about three
seconds; the $100$-digit Riccati--Pad\'e value of \cite{FernandezGarcia2021} lies inside it.

\begin{table}[t]\centering\small
\caption{Temple enclosures for even prolate trials \eqref{eq:prolate-trial} with $p=3$, with the
analytic separator of \Cref{lem:even-sep}, as in \eqref{eq:temple-global}, and with the Stage-A separator
$\rho_A$ of \Cref{cor:temple-eee}. Degrees are the maximal powers of $\xi$ and of $\eta$ in $P$;
widths are upper bounds on $\lambda^+-\lambda^-$, rounded upward, in the units of \Cref{sec:units}
and in hartree. The $\rho_A$ columns use the separator certified by Stage~A (\Cref{app:stageA-sep}).}
\label{tab:global}
\begin{tabular}{@{}rrrcccc@{}}
\toprule
 & & & \multicolumn{2}{c}{$\rho=-\tfrac14$} & \multicolumn{2}{c}{$\rho=\rho_A=-0.2035$}\\
\cmidrule(lr){4-5}\cmidrule(lr){6-7}
deg.\ $\xi$ & deg.\ $\eta$ & terms & width & (hartree) & width & (hartree)\\
\midrule
40 & 24 & 533 & $7.52\times10^{-34}$ & $1.51\times10^{-33}$ & $6.52\times10^{-34}$ & $1.31\times10^{-33}$\\
60 & 32 & 1\,037 & $6.92\times10^{-44}$ & $1.39\times10^{-43}$ & $6.00\times10^{-44}$ & $1.20\times10^{-43}$\\
70 & 36 & 1\,349 & $1.73\times10^{-47}$ & $3.45\times10^{-47}$ & $1.50\times10^{-47}$ & $3.00\times10^{-47}$\\
80 & 40 & 1\,701 & $7.43\times10^{-51}$ & $1.49\times10^{-50}$ & $1.30\times10^{3}$ & $2.59\times10^{3}$\\
\bottomrule
\end{tabular}
\end{table}

\paragraph{A sharper separator from Stage~A.} On the fully even sector of a box that confines its second
level, the projection bound of \Cref{sec:stageA} gives
$\lambda_2^{\rm eee}(\RR^3)\ge\rho_A=-0.2035$ in place of $\lambda_2(H_+)\ge-\tfrac14$, where
$\lambda_2^{\rm eee}$ is the second level of the fully even part of $H_+$ (for H$_2^+$ the $2s\sigma_g$
level, which is also $\lambda_2(H_+)$); the value is certified (\Cref{app:stageA-sep}). Since \eqref{eq:prolate-trial} is fully even, \eqref{eq:temple-global} holds
with $\rho_A$ in place of $-\tfrac14$ (\Cref{cor:temple-eee}) and the lower endpoint rises to
\begin{equation}\label{eq:global-enclosure-A}
  \lambda_1\ \ge\ -0.3517994403\,2762331597\,6320471693\,3797388001\,7960666394\,628138,
\end{equation}
narrowing every row of \Cref{tab:global} by the factor $0.866$ (the exact level would give $0.812$). The
gain in width is bounded; what the computed separator adds is that it uses no property specific to
H$_2^+$, and so applies where \Cref{lem:even-sep} has no counterpart (\Cref{rem:global-scope}).

\paragraph{Three centres.} For linear H$_3^{2+}$ neither ingredient is closed form. Floating-point
Temple widths with a few hundred multi-centre Slater functions reach $10^{-10}$, six to seven orders
of magnitude below the box widths of \Cref{tab:certified} (\Cref{tab:global-h3}). A rigorous version
needs a certified separator, which the Stage-A route supplies in the fully even sector, and a validated
integration of the three-centre residual; both are left to further work.

\section{Conclusions, limitations, and reproducibility}\label{sec:concl}

The chain from a Coulomb-singular Schr\"odinger operator on $\RR^3$ to a machine-checkable two-sided
certificate is closed with no external spectral input: a localized Hardy inequality restores
coercivity, the Lehmann--Goerisch separator is certified within the method, and a Neumann/Dirichlet
argument lifts box bounds to $\RR^3$ with a computable remainder. Because the verified cost depends on
the basis and not on the box, the discretization--truncation balance is reached at no cost, and the
thread-exact interval assembly removes a factor $21$ from the dominant stage; together they give an
operator-certified enclosure of width $139$\,meV in $5.3$ hours at $\Om_2$
(\Cref{tab:certified,tab:cost}), the remaining gap being set by the auxiliary resolution $N'$. A lower
bound for $\lambda_2$ also makes global trial functions usable: for H$_2^+$ the width falls to
$6.5\times10^{-51}$ with a separator that Stage~A certifies itself, using no property specific to
H$_2^+$.

\paragraph{Limitations.} The systems are one-electron, and neither the moment integrals nor the parity
reduction carries over to an electron--electron term. The parity reduction needs nuclei placed
symmetrically on a coordinate axis; a non-symmetric geometry multiplies the dense interval matrix by
$4$--$16$. The coercivity shift is rigorous but not sharp: the certified shift exceeds the optimum by a factor $1.95$ (\Cref{rem:sharpness}), and the Ritz value of $\mathcal A_\epsilon$ cannot replace it because it lies on the inadmissible side. Floating-point surrogates select cells and
are never quoted as bounds (\Cref{rem:float-not-bound}). Memory, not arithmetic, is the binding constraint. Interval rounding contributes about $10^{-7}$ of the widths, while the verification routine allocates $0.09$\,GB plus $7.03$ times the dense interval matrix. A matrix-free or blocked interval assembly is therefore the most valuable next step.

\paragraph{Reproducibility.}\label{sec:repro} The verified pipeline, the sweep drivers, the certificate
of every configuration in \Cref{tab:certified}, the box-size $\times$ basis-size surface, the cost table,
the figure scripts and a provenance table mapping every reported number to the driver output that holds
it are archived at \url{https://github.com/xfliu/schrodinger}. Certified runs used Julia 1.12.7 with \texttt{IntervalArithmetic.jl} 1.0.11 on one 224-core, 3\,022\,GB node (Intel Xeon Platinum 8380H, OpenBLAS ILP64); the $L=20$ comparison used Julia 1.12.6 with \texttt{IntervalArithmetic.jl} 1.0.8 on a 96-core, 251\,GB node. Thread counts ($32$ for $\Om_1$, $16$ for $\Om_2$, $48$ for three centres) are recorded in each certificate. They do not affect validity, but bit-for-bit reproduction of a certificate requires the recorded count.

\section*{Acknowledgments and declarations}
This work was supported by JSPS KAKENHI Grant-in-Aid for Challenging Research (Exploratory) 26K22276, Grant-in-Aid for Scientific Research (B) 24K00538, and Grant-in-Aid for Scientific Research (A) 26H02490.

\emph{Competing interests:} none.
\emph{Data availability:} code, certificates and provenance table are archived at
\url{https://github.com/xfliu/schrodinger} (\Cref{sec:repro}).

\appendix
\crefalias{section}{appendix}
\crefalias{subsection}{appendix}
\crefalias{subsubsection}{appendix}
\crefname{appendix}{Appendix}{Appendices}
\Crefname{appendix}{Appendix}{Appendices}

\section*{Appendices}
The appendices are supporting material and may be skipped. \Cref{app:si:hardy,app:si:proj,app:si:formfactor,app:si:proj-proof,sec:sharp}
carry the proofs of the statements quoted in \Cref{sec:truncation,sec:stageA,sec:stageB} together with the
constants those proofs require; a reader who accepts the statements as displayed needs none of it.
\Cref{app:choices} records how the settings of the reported runs were chosen, which alternatives the
method admits at each choice, and where the residual width and the cost come from --- none of which is
needed to reproduce a bound, since the certificates carry every input. 

\section{The admissible shift, and what it costs}\label{app:si:hardy}

Throughout this section the potential is the multi-centre attractive Coulomb
potential
\begin{equation}\label{app:eq:V-explicit}
  V(x)=-\sum_{i=1}^{K}\frac{Z_i}{|x-a_i|},\qquad Z_i>0,\quad
  Z_{\rm tot}=\sum_{i=1}^{K}Z_i,
\end{equation}
with nuclei $a_i\in\Omega$ each satisfying
$\operatorname{dist}(a_i,\partial\Omega)\ge\ell$ for a common inradius $\ell>0$.
Thus $V\le0$, its only singularities are the nuclei $a_i$, and the associated form
is $a(u,v)=(\nabla u,\nabla v)+(Vu,v)$. Recall $-(Vu,u)=(V^-u,u)$. \Cref{lem:hardy,lem:sharp-form} of the main text make $-\Delta+V$ coercive after a shift, with all constants explicit in $\ell$, $Z_{\rm tot}$ and free proof parameters. This section proves the shift formula \eqref{eq:sigma} they feed, and records what an admissible shift costs.

\begin{remark}[the cost of an admissible shift]\label{rem:shift-cost}
Replacing the inadmissible $-\eta_{1,N}$ by a certified $\Ceps$ widens the enclosure by $2.2\%$, $8.0\%$
and $3.8\%$ at $L=12$, $L=14$ and D1, the configurations where $c$ can be matched. The Lehmann--Goerisch
step itself is unaffected: $A_0$ and $A_2$ do not depend on $\Ceps$, and feeding the earlier separator
through the new $A_0,A_2$ reproduces the earlier endpoint to $1.7\times10^{-15}$. The shift acts through
two channels. The separator is certified slightly weaker, and, the larger effect, the admissibility rule
for $c$ reads Stage~A's $\inf L_1$.
\end{remark}

\begin{proof}[Proof of \eqref{eq:sigma}]
We bound the attractive part $(V^-v,v)=\sum_{i=1}^{K}Z_i\int_\Omega
|v|^2/|x-a_i|\,dx$ from above, one nucleus at a time. At the $i$-th centre apply
the pointwise Young bound $1/r\le(s/2)/r^2+1/(2s)$ with $r=|x-a_i|$ and a common
$s=2\epsilon/(Z_{\rm tot}C_{\rm grad})>0$, then control the resulting
$\int_\Omega|v|^2/|x-a_i|^2$ by the localized Hardy inequality \eqref{eq:hardy}
(each $a_i$ has $\operatorname{dist}(a_i,\partial\Omega)\ge\ell$, so the same
$C_{\rm grad},C_{L^2}$ apply):
\[
  Z_i\!\int_\Omega\!\frac{|v|^2}{|x-a_i|}
  \le \frac{Z_i s}{2}\Bigl(C_{\rm grad}\|\nabla v\|^2+C_{L^2}\|v\|^2\Bigr)
     +\frac{Z_i}{2s}\|v\|^2 .
\]
\emph{Summing over the $K$ centres} — i.e.\ over the terms of the potential
\eqref{app:eq:V-explicit} — replaces each $Z_i$ by $\sum_i Z_i=Z_{\rm tot}$, so
\[
  (V^-v,v)\ \le\
  \underbrace{\frac{Z_{\rm tot}s\,C_{\rm grad}}{2}}_{=\ \epsilon}\,\|\nabla v\|^2
  +\Bigl(\frac{Z_{\rm tot}s\,C_{L^2}}{2}+\frac{Z_{\rm tot}}{2s}\Bigr)\|v\|^2 ,
\]
the gradient coefficient collapsing to $\epsilon$ by the choice of $s$. Hence in
$a_\sigma(v,v)=\|\nabla v\|^2-(V^-v,v)+\sigma\|v\|^2$ the coefficient of
$\|\nabla v\|^2$ becomes $1-\epsilon$, and the coefficient of $\|v\|^2$ is
$\sigma-\bigl(\tfrac12 Z_{\rm tot}s\,C_{L^2}+\tfrac12 Z_{\rm tot}/s\bigr)
=\sigma-\bigl(\epsilon C_{L^2}/C_{\rm grad}+Z_{\rm tot}^2C_{\rm grad}/(4\epsilon)\bigr)=0$
by the definition \eqref{eq:sigma} of $\sigma$ and $s=2\epsilon/(Z_{\rm tot}C_{\rm grad})$.
Therefore $a_\sigma(v,v)\ge(1-\epsilon)\|\nabla v\|^2$, as claimed. The form factor
follows from Cauchy--Schwarz applied to $(Vv_N,v_N)$ together with \eqref{eq:hardy},
using $\|\nabla v_N\|^2\le(1-\epsilon)^{-1}\|v_N\|_{a_\sigma}^2$ and
$\|v_N\|^2\le(\mu_{1,N}+\sigma)^{-1}\|v_N\|_{a_\sigma}^2$.
\end{proof}

\section{The projection framework and its constants}\label{app:si:proj}

Stage~A rests on the projection (Rayleigh--Ritz complementary) lower bound of
\cite{Liu2015}, treated systematically in \cite[Chapter~3]{Liu2024}, in the
spectral-Galerkin realization of \cite{LiuSpectral2026}. As that last reference is
still under review, we restate here the part on which the present
paper depends---the abstract lower-bound theorem and the sharp spectral
constant---and give complete proofs, so that the present paper is self-contained.

Throughout this section $\widehat a$ and $\widehat b$ are two symmetric,
positive-definite bilinear forms on a Hilbert space $\widehat V$ (for us
$\widehat a=a_\sigma$ the shifted energy form and $\widehat b=(\cdot,\cdot)$ the
$L^2$ inner product), with induced norms $\snorm{\cdot}{\widehat a}$,
$\snorm{\cdot}{\widehat b}$. The generalized eigenproblem
$\widehat a(u,v)=\lambda\,\widehat b(u,v)$ has eigenpairs $(\lambda_j,u_j)$.

\begin{assumption}\label{app:as:abstract}\leavevmode
\begin{enumerate}
\item[\textup{(A1)}] There are exact eigenpairs
$(\lambda_j,u_j)$, $j\ge1$, with $0<\lambda_1\le\lambda_2\le\cdots$ and the Gram
relations $\widehat b(u_i,u_j)=\delta_{ij}$, $\widehat a(u_i,u_j)=\lambda_j
\delta_{ij}$.
\item[\textup{(A2)}] $\widehat V_N\subset\widehat V$ is finite-dimensional, with
$\widehat a,\widehat b$ positive definite on it; the \emph{discrete eigenvalues}
$0<\lambda_{1,N}\le\cdots\le\lambda_{M,N}$ ($M=\dim\widehat V_N$) solve the matrix
eigenproblem for $(\widehat a,\widehat b)$ restricted to $\widehat V_N$.
\item[\textup{(A3)}] For each $k\le M$ there is a linear map
$\Pi_N:E_k\to\widehat V_N$ on $E_k:=\operatorname{span}\{u_1,\dots,u_k\}$ with the
\emph{Pythagoras property}
$\widehat a(\Pi_N\phi,\phi-\Pi_N\phi)=0$ and the \emph{projection-error estimate}
$\snorm{\phi-\Pi_N\phi}{\widehat b}\le C_N\snorm{\phi-\Pi_N\phi}{\widehat a}$ for
all $\phi\in E_k$.
\end{enumerate}
\end{assumption}
In the conforming case $\Pi_N$ is the $\widehat a$-orthogonal (Ritz) projection and
(A3) holds automatically; only the Gram relations (A1) are used below.

\begin{theorem}[Projection-based lower bound; {\rm Liu \cite{Liu2015}, \cite[Chapter~3]{Liu2024}}]
\label{app:thm:si-liu}
Under \Cref{app:as:abstract}, for each $k=1,\dots,M$,
\begin{equation}\label{app:eq:si-liu}
  \lambda_k\ \ge\ \frac{\lambda_{k,N}}{1+C_N^2\,\lambda_{k,N}} .
\end{equation}
\end{theorem}

\begin{proof}
Write $\kappa_j=1/\lambda_j$, $\kappa_{j,N}=1/\lambda_{j,N}$ for the reciprocal
eigenvalues; after inversion, \eqref{app:eq:si-liu} is equivalent to
$\kappa_k\le\kappa_{k,N}+C_N^2$. Introduce the reciprocal Rayleigh quotient
$\widehat R(v):=\snorm{v}{\widehat b}^2/\snorm{v}{\widehat a}^2$ (for
$\snorm{v}{\widehat a}>0$). By the Gram relations (A1), any
$\phi=\sum_{j\le k}c_ju_j\in E_k$ has $\snorm{\phi}{\widehat b}^2=\sum c_j^2$ and
$\snorm{\phi}{\widehat a}^2=\sum\lambda_jc_j^2$, so
\begin{equation}\label{app:eq:si-exact-minmax}
  \min_{0\ne\phi\in E_k}\widehat R(\phi)=\kappa_k ,
\end{equation}
while the max--min principle for the pencil $(\widehat b,\widehat a)$ on
$\widehat V_N$ gives
\[
  \kappa_{k,N}=\max_{\dim S=k}\ \min_{0\ne v\in S}\widehat R(v).
\]

\emph{Case 1: $\Pi_N$ is not injective on $E_k$.} Choose $0\ne\phi\in E_k$ with
$\Pi_N\phi=0$. The projection-error estimate (A3) gives
$\snorm{\phi}{\widehat b}\le C_N\snorm{\phi}{\widehat a}$, i.e.\
$\widehat R(\phi)\le C_N^2$, so by \eqref{app:eq:si-exact-minmax}
$\kappa_k\le C_N^2\le\kappa_{k,N}+C_N^2$.

\emph{Case 2: $\Pi_N$ is injective on $E_k$.} Then $S:=\Pi_N(E_k)$ has dimension
$k$, so the discrete max--min principle yields $\phi^*\in E_k\setminus\{0\}$ with
$\widehat R(\Pi_N\phi^*)\le\kappa_{k,N}$, i.e.\
$\snorm{\Pi_N\phi^*}{\widehat b}\le\sqrt{\kappa_{k,N}}\,
\snorm{\Pi_N\phi^*}{\widehat a}$. By the triangle inequality in
$\snorm{\cdot}{\widehat b}$, then (A3), then Cauchy--Schwarz in $\R^2$,
\[
  \snorm{\phi^*}{\widehat b}
  \le\snorm{\Pi_N\phi^*}{\widehat b}+\snorm{\phi^*-\Pi_N\phi^*}{\widehat b}
  \le\sqrt{\kappa_{k,N}}\,\snorm{\Pi_N\phi^*}{\widehat a}
    +C_N\,\snorm{\phi^*-\Pi_N\phi^*}{\widehat a}
\]
\[
  \le\sqrt{\kappa_{k,N}+C_N^2}\,
    \Bigl(\snorm{\Pi_N\phi^*}{\widehat a}^2
    +\snorm{\phi^*-\Pi_N\phi^*}{\widehat a}^2\Bigr)^{1/2}
  =\sqrt{\kappa_{k,N}+C_N^2}\,\snorm{\phi^*}{\widehat a},
\]
the last equality being the Pythagoras property (A3). Hence
$\kappa_k\le\widehat R(\phi^*)\le\kappa_{k,N}+C_N^2$ by
\eqref{app:eq:si-exact-minmax}. In both cases $\kappa_k\le\kappa_{k,N}+C_N^2$, which is
\eqref{app:eq:si-liu}.
\end{proof}

The bound is unconditional: unlike Temple--Kato or Lehmann--Goerisch it needs no
lower bound on $\lambda_{k+1}$ and no spectral localization. Its quality is
governed by the dimensionless product $C_N^2\lambda_{k,N}$---useful precisely when
$C_N^2\lambda_{k,N}\ll1$---so everything rests on making $C_N$ explicit and small.
For a trial space spanned by exact Laplacian eigenmodes this constant is not only
explicit but optimal.

\begin{lemma}[Sharp spectral constant]\label{app:lem:si-CN}
Let $\widehat V_N$ be spanned by the eigenmodes with index in a retained set
$\Lambda_N$, let $\Pi_N$ be the modal ($\widehat a$-orthogonal) truncation, and let
$\lambda_*:=\min\{\lambda_j:j\notin\Lambda_N\}$ be the first omitted eigenvalue.
Then $C_N=\lambda_*^{-1/2}$ is admissible in \textup{(A3)}, and it is optimal
\textup(the smallest constant for which the projection-error estimate holds\textup).
\end{lemma}

\begin{proof}
Expand $v$ in the $\widehat b$-orthonormal basis (A1): $v=\sum_j c_ju_j$ with
$c_j=\widehat b(v,u_j)$. Because the $u_j$ are simultaneously $\widehat b$- and
$\widehat a$-orthogonal, the modal truncation is
$\Pi_N v=\sum_{j\in\Lambda_N}c_ju_j$, so
$(I-\Pi_N)v=\sum_{j\notin\Lambda_N}c_ju_j$. Using
$\snorm{u_j}{\widehat b}^2=1$ and $\snorm{u_j}{\widehat a}^2=\lambda_j$ from the
Gram relations, Parseval in both norms gives
\[
  \snorm{(I-\Pi_N)v}{\widehat b}^2
  =\sum_{j\notin\Lambda_N}c_j^2
  \le\frac1{\lambda_*}\sum_{j\notin\Lambda_N}\lambda_j\,c_j^2
  =\frac1{\lambda_*}\snorm{(I-\Pi_N)v}{\widehat a}^2 ,
\]
since $\lambda_j\ge\lambda_*$ for every omitted index $j\notin\Lambda_N$. Taking
$v=u_{j_*}$ with $\lambda_{j_*}=\lambda_*$ ($j_*\notin\Lambda_N$) gives
$\Pi_N v=0$ and equality, so the constant is optimal.
\end{proof}

For the cosine trial space the omitted modes have Laplacian eigenvalue at least
$\nu_*=\big((N+1)\pi/(2L_x)\big)^2$, so $\lambda_*=\nu_*$ and the sharp constant is
$C_N=\nu_*^{-1/2}=O(N^{-1})$; this is the value used in \Cref{app:si:proj-proof} once
the coercivity shift has restored positivity. For the Coulomb operator the form is
\emph{not} positive definite, and \Cref{app:si:hardy}--\Cref{app:si:proj-proof} restore
positivity by a coercivity shift and propagate the constant through the shifted
form.

\section{The projection-error constant via the potential form factor}\label{app:si:formfactor}

The abstract bound \eqref{app:eq:si-liu} is only as good as the projection constant
$C_N$. For the Schr\"odinger form on the cosine space, \cite{LiuSpectral2026} makes $C_N$ explicit
through a single scalar---the \emph{potential form factor}---and this is the
mechanism Stage~A uses. We summarize it here for a positive-definite form (the
case to which the coercivity shift of \Cref{app:si:hardy} reduces the Coulomb
operator) and give the short proofs.

Throughout, $a(u,v)=(\nabla u,\nabla v)+(Vu,v)$ with $V\ge0$, $\|v\|_a^2=a(v,v)$,
and $V_N$ is the cosine trial space with omitted-mode floor
$\nu_*=\min_{\bm m\notin\Lambda_N}\nu_{\bm m}$. The standard Galerkin matrix is
$K+P$ ($K$ diagonal kinetic, $P_{\bm i\bm j}=(V\varphi_{\bm i},\varphi_{\bm j})$),
whose eigenvalues $\lambda_{k,N}\ge\mu_k$ are the Rayleigh--Ritz upper bounds. Two
distinct projections onto $V_N$ enter:
\begin{itemize}\itemsep2pt
\item the \emph{modal} projection $\Pi_N^0$, the $L^2$-orthogonal truncation of the
cosine expansion; it is also $a_0$-orthogonal ($a_0=(\nabla\cdot,\nabla\cdot)$) and
obeys the exact spectral-gap estimate
\begin{equation}\label{app:eq:si-modaltail}
  \|u-\Pi_N^0u\|_{L^2}\le\nu_*^{-1/2}\,\|\nabla(u-\Pi_N^0u)\|_{L^2},
\end{equation}
since $u-\Pi_N^0u$ contains only modes with $\nu_{\bm m}\ge\nu_*$;
\item the \emph{Ritz} projection $\Pi_N$ of the full form,
$a(u-\Pi_Nu,v_N)=0$ for all $v_N\in V_N$---the projection required by
\eqref{app:eq:si-liu}, but for which no modal tail estimate is directly available
because the cosines are not eigenfunctions of $-\Delta+V$.
\end{itemize}
The entire difficulty is to control the \emph{projection gap}
$e:=\Pi_N^0u-\Pi_Nu\in V_N$, which an algebraic identity transfers to data seen by
the \emph{modal} projection.

\begin{lemma}[Gap identity]\label{app:lem:si-gapid}
For every $u\in\mathcal V$ and $v_N\in V_N$, with $e=\Pi_N^0u-\Pi_Nu$,
\begin{equation}\label{app:eq:si-gapid}
  a(e,v_N)=\bigl(V(\Pi_N^0u-u),\,v_N\bigr)_{L^2}.
\end{equation}
\end{lemma}
\begin{proof}
$a(e,v_N)=a(\Pi_N^0u-u,v_N)+a(u-\Pi_Nu,v_N)$; the second term vanishes by the
definition of $\Pi_N$, and in the first the gradient part
$a_0(\Pi_N^0u-u,v_N)=0$ by $a_0$-orthogonality of the modal projection, leaving the
potential part \eqref{app:eq:si-gapid}.
\end{proof}

\begin{definition}[Potential form factor]\label{app:def:si-etaV}
$\eta_V\ge0$ is any constant with
$|(Vw,v_N)|\le\eta_V\|w\|_{L^2}\|v_N\|_a$ for all $w\in L^2(\Omega)$, $v_N\in V_N$.
\end{definition}

\begin{lemma}[Projection gap]\label{app:lem:si-projgap}
With $\lambda_{1,N}>0$ the smallest Galerkin eigenvalue and
$g:=\lambda_{1,N}^{-1/2}\eta_V$,
\begin{equation}\label{app:eq:si-projgap}
  \|\Pi_N^0u-\Pi_Nu\|_{L^2}\le g\,\nu_*^{-1/2}\,\|\nabla(u-\Pi_N^0u)\|_{L^2}.
\end{equation}
\end{lemma}
\begin{proof}
Test \eqref{app:eq:si-gapid} with $v_N=e$ and use \Cref{app:def:si-etaV} with
$w=\Pi_N^0u-u$: $\|e\|_a^2\le\eta_V\|u-\Pi_N^0u\|_{L^2}\|e\|_a$, so
$\|e\|_a\le\eta_V\|u-\Pi_N^0u\|_{L^2}$. The discrete Rayleigh bound
$\|e\|_{L^2}\le\lambda_{1,N}^{-1/2}\|e\|_a$ and the modal-tail estimate
\eqref{app:eq:si-modaltail} give \eqref{app:eq:si-projgap}.
\end{proof}

\begin{theorem}[Projection-error constant for $-\Delta+V$, {\rm Liu \cite{LiuSpectral2026}}]
\label{app:thm:si-CNV}
Let $V\ge0$. Then for all $u\in\mathcal V$,
\begin{equation}\label{app:eq:si-CNV}
  \|(I-\Pi_N)u\|_{L^2}\le C_N^{(V)}\|(I-\Pi_N)u\|_a,
  \qquad C_N^{(V)}=\sqrt{\frac{1+g^2}{\nu_*}},
\end{equation}
and consequently $\mu_k\ge\lambda_{k,N}/\bigl(1+(C_N^{(V)})^2\lambda_{k,N}\bigr)$ for
$k=1,\dots,\dim V_N$.
\end{theorem}
\begin{proof}
Write $w=(I-\Pi_N)u=(u-\Pi_N^0u)+e$ with the two parts $L^2$-orthogonal. By
\eqref{app:eq:si-modaltail} and \Cref{app:lem:si-projgap},
$\|w\|_{L^2}^2=\|u-\Pi_N^0u\|_{L^2}^2+\|e\|_{L^2}^2\le
(1+g^2)\nu_*^{-1}\|\nabla(u-\Pi_N^0u)\|^2$. Since $\Pi_N^0$ is the best
$a_0$-approximation and $V\ge0$,
$\|\nabla(u-\Pi_N^0u)\|\le\|\nabla(u-\Pi_Nu)\|\le\|w\|_a$, giving
\eqref{app:eq:si-CNV}. The eigenvalue bound is \eqref{app:eq:si-liu} with $\Pi_N$ the Ritz
projection.
\end{proof}

At $V=0$, read as the limit $V\to0$ (for $V\equiv0$ the Neumann Ritz value $\lambda_{1,N}$ vanishes and $g$ is $0/0$), one has $g=0$ and $C_N^{(V)}=\nu_*^{-1/2}$, the optimal constant of
\Cref{app:lem:si-CN}: the mechanism degrades gracefully to the sharp eigenbasis value,
and the potential enters only through the factor $\sqrt{1+g^2}$. Everything comes
from one assembly and one eigensolve of the standard matrix $K+P$.

\paragraph{Application to the Coulomb operator.} The Coulomb potential is neither
bounded nor nonnegative, so $\|V\|_{L^\infty}=\infty$ and the form is not positive
definite: \Cref{app:thm:si-CNV} does not apply directly. The next section removes both
obstructions with a single device---a coercivity shift $\sigma$ built from a
localized Hardy inequality---after which the shifted form $a_\sigma$ is positive
definite with $a_\sigma(v,v)\ge(1-\epsilon)\|\nabla v\|^2$, the role of
$\lambda_{1,N}$ is played by $\mu_{1,N}+\sigma$, and the abstract form factor of
\Cref{app:def:si-etaV} stays \emph{finite} despite the singularity. Propagating these
through the proof above turns the constant \eqref{app:eq:si-CNV} into
$\widehat C_N^2=(1+g^2)/[(1-\epsilon)\nu_*]$ with
$g=(\mu_{1,N}+\sigma)^{-1/2}\eta_V$---the constant used throughout this paper; it is proved in \Cref{app:si:proj-proof}.

\section{Proof of the projection lower bound}\label{app:si:proj-proof}

The proof combines the abstract lower bound \eqref{app:eq:si-liu} with the
positive-definite form-factor constant \eqref{app:eq:si-CNV}, but with one change of
metric: for the Coulomb operator the natural energy is not $a_0=(\nabla\cdot,
\nabla\cdot)$ but the shifted form $a_\sigma$, and it is here that the factor
$(1-\epsilon)$ enters. We isolate that step as a lemma, because it is the only
place the coercivity constant of \Cref{lem:coercive} is used, and its role is
easy to misread.

\begin{lemma}[Shifted modal-tail estimate]\label{app:lem:shifted-tail}
Let $\sigma$ and $\epsilon\in(0,1)$ be as in \Cref{lem:coercive}, so that
$a_\sigma(v,v)=a(v,v)+\sigma\|v\|^2\ge(1-\epsilon)\|\nabla v\|^2$ for all
$v\in H^1(\Omega)$. Let $\Pi_N^0$ be the modal ($L^2$-orthogonal) truncation onto
the cosine space $V_N$, whose omitted modes have Laplacian eigenvalue
$\ge\nu_*$. Then for every $v\in H^1(\Omega)$
\begin{equation}\label{app:eq:shifted-tail}
  \|(I-\Pi_N^0)v\|_{L^2}^2\ \le\ \frac{1}{\nu_*}\,\|\nabla(I-\Pi_N^0)v\|^2
  \ \le\ \frac{1}{(1-\epsilon)\,\nu_*}\,\|(I-\Pi_N^0)v\|_{a_\sigma}^2 .
\end{equation}
\end{lemma}
\begin{proof}
The first inequality is the exact modal spectral-gap estimate
\eqref{app:eq:si-modaltail}: $w:=(I-\Pi_N^0)v$ contains only modes with Laplacian
eigenvalue $\ge\nu_*$, so $\|w\|_{L^2}^2\le\nu_*^{-1}\|\nabla w\|^2$. For the
second, apply the coercivity bound of \Cref{lem:coercive} to $w$ itself,
$\|\nabla w\|^2\le(1-\epsilon)^{-1}a_\sigma(w,w)=(1-\epsilon)^{-1}\|w\|_{a_\sigma}^2$.
Chaining the two gives \eqref{app:eq:shifted-tail}. The single factor $(1-\epsilon)^{-1}$
is exactly the price of measuring the modal tail in the shifted energy norm
$\|\cdot\|_{a_\sigma}$ instead of the pure Dirichlet norm $\|\nabla\cdot\|$.
\end{proof}

\begin{proof}[Proof of \Cref{thm:coulomb-bound}]
Pass to the shifted operator $H'=H+\sigma$ with form $a_\sigma$. By
\Cref{lem:coercive}, $a_\sigma$ is symmetric positive definite with eigenvalues
$\mu_k+\sigma>0$, so the abstract principle \eqref{app:eq:si-liu} applies to the pencil
$(a_\sigma,(\cdot,\cdot))$ on $V_N$:
\begin{equation}\label{app:eq:si-liu-shift}
  \mu_k+\sigma\ \ge\ \frac{\mu_{k,N}+\sigma}{1+C_N^2(\mu_{k,N}+\sigma)},
\end{equation}
with $C_N$ the Ritz projection constant \emph{for the form $a_\sigma$}. We bound
$C_N$ exactly as in \Cref{app:thm:si-CNV}, but with $\|\cdot\|_a$ there replaced by the
shifted norm $\|\cdot\|_{a_\sigma}$ and $\lambda_{1,N}$ by $\mu_{1,N}+\sigma$.
Writing $w=(I-\Pi_N)v=(v-\Pi_N^0v)+e$ with the two parts $L^2$-orthogonal and
$e=\Pi_N^0v-\Pi_Nv\in V_N$, the projection-gap bound \eqref{app:eq:si-projgap} gives
$\|e\|_{L^2}\le g\,\nu_*^{-1/2}\|\nabla(v-\Pi_N^0v)\|$ with
$g=(\mu_{1,N}+\sigma)^{-1/2}\eta_V$, while the \emph{shifted} modal-tail estimate
\eqref{app:eq:shifted-tail} of \Cref{app:lem:shifted-tail} controls the first part. Hence
\[
  \|w\|_{L^2}^2=\|v-\Pi_N^0v\|_{L^2}^2+\|e\|_{L^2}^2
  \le\frac{1+g^2}{\nu_*}\,\|\nabla(v-\Pi_N^0v)\|^2
  \le\frac{1+g^2}{(1-\epsilon)\,\nu_*}\,\|w\|_{a_\sigma}^2 ,
\]
the last step by \eqref{app:eq:shifted-tail} together with
$\|\nabla(v-\Pi_N^0v)\|\le\|\nabla(v-\Pi_Nv)\|$ ($\Pi_N^0$ is the best
$a_0$-approximation) and $\|\nabla(v-\Pi_Nv)\|^2\le(1-\epsilon)^{-1}\|w\|_{a_\sigma}^2$.
Therefore
\begin{equation}\label{app:eq:CNhat}
  C_N^2\ \le\ \widehat C_N^2:=\frac{1+g^2}{(1-\epsilon)\,\nu_*}
  \ =\ O(\nu_*^{-1})=O(N^{-2}),
\end{equation}
which is \eqref{app:eq:si-CNV} amplified by the single coercivity factor
$(1-\epsilon)^{-1}$ traced in \Cref{app:lem:shifted-tail}. Substituting \eqref{app:eq:CNhat}
into \eqref{app:eq:si-liu-shift} and subtracting $\sigma$ gives the bound stated in \Cref{thm:coulomb-bound}. Since $\sigma$ and $\epsilon$ are independent of $N$, and $g$ depends on $N$ only through $\mu_{1,N}$ and is bounded uniformly in $N$, the gap $\mu_{k,N}-\mu_k$ closes at rate $O(N^{-2})$ with no floor.
\end{proof}

\section{Stage B: proofs and comparisons}\label{app:stageB}

This appendix collects the material behind \Cref{sec:stageB} that the algorithm does not execute: the
place of the bound among the classical ones, the status of its hypotheses, the proofs of the parity
reduction, and the comparisons that justify the separator route of \Cref{sec:mu2}.

\subsection{The bound in context}\label{app:stageB-context}
\Cref{thm:lg} is classical. It is Lehmann's bound \cite{Lehmann1963} in the
operator form given by Goerisch, and complete proofs are available in
\cite{GoerischHaunhorst1985,BeattieGoerisch1995} and in the verified-computation
literature \cite{Liu2015,LiuPlum2025}; we therefore state it and do not
reproduce a proof. What is \emph{not} classical is the price paid for solving the auxiliary
constraint inexactly --- the realization actually used, and that price, are the subject of
\Cref{sec:a4} --- and the position of the bound within the Weinstein--Temple hierarchy, addressed in
\Cref{rem:hierarchy} of \Cref{sec:lg-bound}.

The optimal test functions are $w_i=T\tilde u_i$, where
\begin{equation}\label{eq:lg-w}
  M(\tilde u_i,\phi)=N(v_i,\phi)\qquad\text{for all }\phi\in D .
\end{equation}
This minimization need not be solved exactly---any $w_i$ satisfying (A4) approximately
still yields a valid bound once its Goerisch residual is enclosed
(\Cref{sec:contract}). The one condition that must hold \emph{exactly} is the
positive-definiteness of $B$, and by \Cref{lem:B-positive} it is implied by
$\rho>\Lambda_n$.

The separator of \eqref{eq:separation} is certified by the method whenever the certified lower bound on
$\mu_2(\Om)$ exceeds $U_1$. The final inequality of \eqref{eq:separation} is the $k{=}2$ case of
\Cref{thm:enc-main}: a certified lower bound on the \emph{box} eigenvalue $\mu_2^{\mathrm{Neu}}(\Omega)$
lifts, through $\mu_2^{\mathrm{Neu}}(\Omega)\le\lambda_2(\R^3)$, to the whole-space second eigenvalue.

\subsection{Hypotheses and outputs}\label{app:stageB-hyp}

\begin{remark}[a bound formula evaluated in floating point is not a bound]
\label{rem:float-not-bound}
\Cref{thm:lg-sharp} and \eqref{eq:liu-bound} are inequalities about exactly
computed quantities. Evaluating the same expressions in floating point yields a
number that resembles a bound and is not one: the hypotheses are then checked
only approximately, and $A_0,A_1,A_2$, the separator and the closing transform
all carry rounding error of unknown sign. This is not a remote hazard. At configuration D1 of \Cref{tab:configs} ($d=2.8$, $L=12$, $N=64$) the floating-point evaluation of the Lehmann--Goerisch
chain returned a purported lower bound lying $4.3\times10^{-5}$ \emph{above} an
independently computed upper bound for the same eigenvalue, an inequality that cannot hold. The
interval-arithmetic computation at the same configuration satisfied every hypothesis and produced a
valid enclosure. Nothing internal to the floating-point sequence indicates the error. Floating-point
Lehmann--Goerisch values are therefore used only to choose which cells to certify, and are never quoted
as bounds or widths. Likewise, the customary error bar of a converged Rayleigh--Ritz value, built from
successive decrements along one basis ladder, is not a bound; such values are quoted here only as
one-sided bounds.
\end{remark}

\subsection{The parity reduction}\label{app:stageB-parity}
\begin{proof}[Proof of \Cref{lem:parity-decomp}]
Each $U_g$ is unitary on $L^2(\Om)$, maps $H^1(\Om)$ onto itself, and satisfies
$a(U_gu,U_gv)=a(u,v)$ because $\Om$ and $V$ are $G$-invariant; hence $U_g$ commutes
with $\mathcal H$. Since $G$ is abelian of order $8$ it has exactly eight characters,
the $P_p$ are orthogonal projections summing to the identity, and each commutes with
$\mathcal H$. The direct-sum decomposition follows, and the spectrum of a direct sum
is the union of the spectra with multiplicities added.
\end{proof}

The statement is classical and is recorded in \Cref{sec:mu2} only to fix the sector labels used
there. For a single mirror axis it is the even/odd splitting of
\cite[Sec.~7]{KuttlerSigillito1984}, where the several-axis case is noted --- a rectangle
carries four symmetry classes, a box eight --- along with the two consequences this
paper relies on: the classes are mutually orthogonal, so trial functions may be
restricted to one of them, and the decomposition increases the separation between
successive eigenvalues within a class, which is what makes the separator of
\Cref{sec:mu2} attainable sector by sector. For a general finite symmetry group,
abelian or not, the same construction with the isotypic projections and the resulting
splitting of a differential eigenvalue problem into independent subproblems is
\cite[Thm.~2.1, Cor.~2.2]{FangGaoZhou2013}.

\begin{lemma}[monotonicity in the parity index]\label{lem:parity-monotone}
Write $\lambda_k^{(p)}$ for the $k$-th eigenvalue of $\mathcal H_p$ and
$\mathrm{odd}(p)=\{i:p_i=\mathrm o\}$. If $\mathrm{odd}(p)\subseteq\mathrm{odd}(q)$
then $\lambda_k^{(p)}\le\lambda_k^{(q)}$ for every $k$.
\end{lemma}

\begin{proof}
Let $O=\Om\cap\{x_i>0\ \forall i\}$ and $\Gamma_i=\partial O\cap\{x_i=0\}$. Restriction
$u\mapsto u|_O$ maps the form domain of $\mathcal H_p$ bijectively onto
\[
  W_p=\{v\in H^1(O):\ v=0 \text{ on } \Gamma_i \text{ for each } i \text{ with } p_i=\mathrm o\},
\]
because the even reflection of an $H^1$ function across a flat boundary is again $H^1$
with no condition, while the odd reflection is $H^1$ precisely when the trace vanishes.
For $u$ in a parity sector the integrands $\abs{u}^2$, $\abs{\nabla u}^2$ and
$V\abs{u}^2$ are all $G$-invariant, so each of the three integrals over $\Om$ equals
eight times its counterpart over $O$ and the Rayleigh quotient is unchanged. Thus
$\lambda_k^{(p)}$ is the $k$-th min--max value of that quotient over $W_p$. If
$\mathrm{odd}(p)\subseteq\mathrm{odd}(q)$ then $W_q\subseteq W_p$, and min--max over a
subspace is no smaller.
\end{proof}

\begin{proof}[Proof of \Cref{thm:mu2-candidates}]
By \Cref{lem:aux-sector} the ground state is strictly positive, hence totally
symmetric, and simple; so $\mu_1=\lambda_1^{\mathrm{eee}}$ and it occurs once in the
merged multiset of \Cref{lem:parity-decomp}. After removing that one copy, the smallest
remaining element is either the next element of the eee sequence,
$\lambda_2^{\mathrm{eee}}$, or the first element $\lambda_1^{(p)}$ of one of the seven
other sequences; hence $\mu_2$ is the minimum of those eight numbers. For $p$ with two
or three odd directions there is a single-odd $p'$ with
$\mathrm{odd}(p')\subseteq\mathrm{odd}(p)$, so $\lambda_1^{(p)}\ge\lambda_1^{(p')}$ by
\Cref{lem:parity-monotone} and $p$ may be dropped from the minimum. The last claim holds
because $(x_1,x_2,x_3)\mapsto(x_1,x_3,x_2)$ preserves both $\Om$ and $V$ and conjugates
$\mathcal H_{\mathrm{eoe}}$ to $\mathcal H_{\mathrm{eeo}}$.
\end{proof}

\subsection{Status of the three candidates}\label{app:stageB-cand}
\begin{remark}[the separator carries no assumption]\label{rem:sector-free-separator}
Certified lower bounds for $\lambda_2^{\mathrm{eee}}$, $\lambda_1^{\mathrm{oee}}$ and
$\lambda_1^{\mathrm{eoe}}$ give $\rho:=\min\{\cdot\}\le\mu_2(\Om)$ by \Cref{thm:mu2-candidates}, with no
sector identification. Taking the minimum, rather than naming the sector that attains it, is what
removes the hypothesis. $\lambda_2^{\mathrm{eee}}$ cannot be dropped, since $\mu_2$ may lie in the eee
sector. No general principle decides which candidate is least. Sturm oscillation settles it in one
dimension, and the explicit spectrum settles it for $V\equiv0$ on a box, but neither argument extends to a
three-dimensional Coulomb potential.
\end{remark}

\begin{remark}[status of the three candidates]\label{rem:sector-assumed}
On $\Om_2$ the pipeline's separator, the certified oee bound, carries no sector assumption. The eee and
eoe candidates of \Cref{thm:mu2-candidates} were certified by \Cref{thm:coulomb-bound} ($k=2$ on eee,
$k=1$ on eoe, $\eta_V^2$ re-certified on each, $\sigma=1.7656061570$), and the weaker clears the oee value
by $1.017\times10^{-1}$ (\Cref{tab:sepfree}). On $\Om_1$ the eee candidate is certified, with margin
$1.042\times10^{-1}$, but the eoe candidate is not; in approximate mode it clears by $7.31\times10^{-2}$.
The $\Om_1$ bound therefore rests on one uncertified candidate. The Stage-A value
$L_2=-0.3853219389$ (\Cref{tab:mu3-needed}) is an unconditional alternative, at a cost of $7.93\times10^{-5}$ on the lower
endpoint. The $k=2$ lift holds on both boxes: $\sigma(\Om_1)=-0.2983$ and $\sigma(\Om_2)=-0.2425$ exceed
$\lambda_2(\R^3)=-0.3338$.
\end{remark}

\begin{table}[t]\centering\small
\caption{The three candidates of \Cref{thm:mu2-candidates} on $\Om_2$, each a certified lower
bound; the last column is the margin by which the weaker of the two further candidates clears the
oee value. Energies in the units of \Cref{sec:units}; provenance in \Cref{rem:sector-assumed}.}
\label{tab:sepfree}
\begin{tabular}{lccccccc}
\toprule
 & $N$ & \multicolumn{2}{c}{$\eta_V^2$ certified} & $L(\lambda_2^{\mathrm{eee}})$ & $L(\lambda_1^{\mathrm{eoe}})$ & $\rho$ (oee, LG) & clears $\rho$ by\\
\cmidrule(lr){3-4}
 & & eee & eoe & & & & \\
\midrule
$\Om_2$ & $72$ & $0.55$ & $0.04$ & $-0.2631316246$ & $-0.2706534666$ & $-0.3723080919$ & $1.017\times10^{-1}$\\
\bottomrule
\end{tabular}
\end{table}

\subsection{What the intermediate $\mu_2$ stage buys}\label{app:stageB-route}
\paragraph{Provenance of the Stage-B inputs.} Stage~B on either box consumes exactly one
spectral input, the separator $\rho$, and it is obtained as follows. The oee
Lehmann--Goerisch computation runs at $N=32$, $N'=48$, $\epsilon'=0.65$, $c=1$, in the
realization of \Cref{sec:a4}; its own separator is the projection bound on the oee
second level computed within that stage ($-0.2331440156$ on $\Om_1$, $-0.3023216607$ on
$\Om_2$), and its output is the certified lower bound on $\mu_2(\Om)$ that Stage~B then uses:
$\rho=-0.3561459574$ on $\Om_1$ and $\rho=-0.3723080919$ on $\Om_2$. The Stage-A value $L_2$
does \emph{not} enter Stage~B on the reported route; each Stage-B certificate records it only
as the fallback candidate of \Cref{rem:mu3-needed}, taken from the Stage-A run at the shift
$\Ceps^{\rm cert}$ certified at $N_{\rm aux}=80$ ($1.4701145303$ on $\Om_1$, $1.6326092627$ on
$\Om_2$; $L_2=-0.3853$ and $-0.3982$). \Cref{tab:stagea} reports Stage~A on $\Om_2$ at the
shift certified at $N_{\rm aux}=72$, $\Ceps^{\rm cert}=1.7656061570$, which is a separate
valid certificate: the Stage-A results are reported in their own right and the reported
enclosure depends on neither Stage-A run, only on $\rho$ and on the admissibility gate
$c\ge\Ceps'$ that Stage~B verifies for itself (\Cref{tab:headline-tiers}).

\begin{remark}[what the intermediate $\mu_2$ stage buys]\label{rem:mu3-needed}
\Cref{tab:mu3-needed} compares the three-stage route of \Cref{sec:mu2} with the direct route, which uses
the Stage-A value $L_2$ as $\rho$. The H$_2^+$ rows are the two separator variants each certificate
records. The three-centre $L_1$ are reconstructed from each configuration's own $A_0,A_2$ through
\Cref{thm:lg-sharp} and agree with the recorded values to $9\times10^{-11}$. At the two $d=4$
configurations the direct route certifies nothing. There $\rho=L_2=-0.9727$, against
$\Lambda_n\approx-0.7620$, drives $\nu_1$ to $+3.48$ and $+3.45$. With $\nu_1>1$, \Cref{thm:lg} asserts
nothing, yet the M\"obius map still returns a number, which can lie above the eigenvalue. The driver
therefore asserts $B\succ0$, $\nu_1<1$ and $\rho>\Lambda_n$ at run time. In the other configurations the
stage buys sharpness: $7.9\times10^{-5}$ ($0.9\%$ of width) on $\Om_1$, $1.7\times10^{-4}$ ($6.8\%$) on
$\Om_2$, and a factor $11.3$ in width at $d=2.8$. It costs $2.9$--$3.8\%$ of the run time
(\Cref{tab:cost}).
\end{remark}

\begin{table}[t]
\centering
\caption{The $\mu_1$ lower bound with and without the intermediate $\mu_2$ stage of
\Cref{sec:mu2}. ``With'' seeds the eee Lehmann--Goerisch step with the certified oee bound;
``without'' seeds it with the Stage-A value $L_2$ directly. Both routes use the certified shift and its
tail-refined form factor throughout. Four rows admit the direct route; the two $d=4$ rows fail
$\nu_1<1$ and $\rho>\Lambda_n$ together, and the numbers their arithmetic returns are not
bounds. $B\succ0$ holds in all twelve cases. Widths are the certified $\R^3$ enclosure widths of
\Cref{tab:certified}; the ``valid'' outcomes give the factor by which the direct route's enclosure
is wider.}
\label{tab:mu3-needed}
\setlength{\tabcolsep}{4pt}
\begin{adjustbox}{max width=\linewidth}
\begin{tabular}{@{}llrrrrrl@{}}
\toprule
& & \multicolumn{2}{c}{with the $\mu_2$ stage} & \multicolumn{4}{c}{without it (direct $\rho=L_2$)}\\
\cmidrule(lr){3-4}\cmidrule(lr){5-8}
system & box, $N_{\rm stageA}$/$N'$ & $\rho$ & $L_1$ & $\rho$ & $\nu_1$ & $L_1$ & outcome\\
\midrule
H$_2^+$            & $\Om_1$, $L{=}8$, $72/48$  & $-0.3561$ & $-0.556494$ & $-0.3853$ & $-2.59$ & $-0.556574$ & valid, $0.9\%$ wider\\
H$_2^+$            & $\Om_2$, $L{=}10$, $72/48$ & $-0.3723$ & $-0.553291$ & $-0.3982$ & $-2.88$ & $-0.553465$ & valid, $6.8\%$ wider\\
H$_3^{2+}$, $d{=}2.8$& $L{=}12$, $76/80$        & $-0.6650$ & $-0.936544$ & $-0.9197$ & $-29.1$ & $-0.939007$ & valid, $11.3\times$ wider\\
H$_3^{2+}$, $d{=}4$  & $L{=}12$, $64/48$        & $-0.6092$ & $-0.764477$ & $-0.9727$ & $+3.48$ & --- & no certificate\\
H$_3^{2+}$, $d{=}4$  & $L{=}12$, $64/40$        & $-0.6362$ & $-0.768033$ & $-0.9727$ & $+3.45$ & --- & no certificate\\
\bottomrule
\end{tabular}
\end{adjustbox}
\end{table}
How much a sharper separator would buy on $\Om_2$ follows from \eqref{eq:lg-closed} with the
certificate's own $A_0,A_2$ (a sub-second recomputation, no re-assembly): $\rho=-0.3396$ would raise the
lower endpoint from $-0.5532907722$ to $-0.5531311076$, and even the exact $\lambda_2^{\rm oee}=-0.3337670$
only to $-0.5531076567$, a gain of $1.8\times10^{-4}$ on a width of $2.55\times10^{-3}$. On $\Om_2$ the
remaining width is set by $A_2$, that is by the auxiliary resolution $N'$, not by the separator.

\section{The optimal constant, and how it is certified}\label{sec:sharp}

The optimal admissible shift is not given by any inequality: it is
$\Ceps^{\rm opt}=-\eta(\epsilon)$, where $\eta(\epsilon)$ is the lowest Neumann eigenvalue of
the auxiliary operator $\mathcal A_\epsilon=-\epsilon\Delta+V$ on $H^1(\Om)$. Locating it is
cheap in the present discretization, because $\mathcal A_\epsilon$ is obtained from the
assembled moment matrix by rescaling the kinetic diagonal, so no new assembly is required; on
$\Om_2$ the verified Rayleigh--Ritz enclosures at $N=32,48,64$ are, at $\epsilon=0.4$,
$0.8977781273$, $0.9025639887$, $0.9042438878$, each of width below $2\times10^{-11}$, and an
independent floating-Gaussian calculation of $\lambda_{\min}(\mathcal A_\epsilon)$ gives
$0.9050944$.

A Rayleigh--Ritz value is not itself an admissible shift, and the obstruction is one of
direction rather than accuracy. The cosine trial space is a subspace of $H^1(\Om)$, so
Rayleigh--Ritz bounds the eigenvalue from above, $\eta_{1,N}\ge\eta$, hence
\begin{equation}\label{eq:ritz-wrong-side}
  -\eta_{1,N}\;\le\;-\eta\;=\;\Ceps^{\rm opt}.
\end{equation}
Admissibility requires a shift \emph{at least} $\Ceps^{\rm opt}$, so a Ritz value lies on the
inadmissible side at every $N$; the recorded sequence increases monotonically towards the independent
value without reaching it.

What is admissible is a certified \emph{lower} bound on $\eta$, and the machinery of this
paper supplies one --- provided it is applied to the operator in the normalization the
projection bound assumes.

\begin{proposition}[certified upper bound for the optimal shift]\label{prop:ceps-cert}
Fix $\epsilon\in(0,1)$ and $\alpha\in(0,\epsilon)$, and let $\sigma_{\rm loc}(\alpha)$ be the
constant of \Cref{lem:sharp-form}. Put
\[
  \epsilon_*:=\frac{\alpha}{\epsilon}\in(0,1),\qquad
  C_*:=\frac{\sigma_{\rm loc}(\alpha)}{\epsilon}.
\]
Let $\nu_{1,N}$ be the first Rayleigh--Ritz value on $V_N$ of the rescaled operator
$\mathcal B_\epsilon:=\epsilon^{-1}\mathcal A_\epsilon=-\Delta+\epsilon^{-1}V$, and let
$\hat C_N$ be the projection constant of \Cref{thm:coulomb-bound} formed with $\epsilon_*$
and $C_*$. Then
\begin{equation}\label{eq:ceps-cert}
  \Ceps^{\rm opt}\;\le\;\Ceps^{\rm cert}
  \;:=\;-\,\epsilon\left[\frac{\nu_{1,N}+C_*}
       {1+\hat C_N^{\,2}\,(\nu_{1,N}+C_*)}-C_*\right],
\end{equation}
and $\Ceps^{\rm cert}$ is an admissible shift in \eqref{eq:sharp-form}, for every such
$\alpha$.
\end{proposition}

\begin{proof}
Dividing the form bound of \Cref{lem:sharp-form} through by $\epsilon$ gives
$\epsilon^{-1}(V^-u,u)\le\epsilon_*\norm{\nabla u}^2+C_*\norm{u}^2$, which is exactly the
hypothesis of \Cref{thm:coulomb-bound} for $\mathcal B_\epsilon$, an operator whose Laplacian
carries the coefficient one. The theorem at $k=1$ bounds
$\nu_1=\lambda_{\min}(\mathcal B_\epsilon)$ from below by the bracket in
\eqref{eq:ceps-cert}. Since $\mathcal A_\epsilon=\epsilon\,\mathcal B_\epsilon$ with
$\epsilon>0$, the eigenvalues scale, $\eta=\epsilon\,\nu_1$, and
$\Ceps^{\rm opt}=-\eta$ gives the claim. Admissibility is immediate: any
$C\ge\Ceps^{\rm opt}$ satisfies \eqref{eq:sharp-form}.
\end{proof}

\begin{remark}[the rescaling is not cosmetic]\label{rem:why-scale}
Without rescaling, the coercivity left after the form bound may be misread as $\epsilon-\alpha$, which
vanishes as the shift approaches its analytic floor. The quantity used by \Cref{thm:coulomb-bound} is the
fraction $1-\alpha/\epsilon\in(0,1)$. What grows as $\alpha\downarrow0$ is
$C_*=\sigma_{\rm loc}(\alpha)/\epsilon$, and the trade-off has an interior optimum in $\alpha$. The same
construction appears in \cite[Appendix~D]{LiuPlum2025}.
\end{remark}

\begin{table}[t]\centering\small
\caption{The certified coercivity shift $\Ceps^{\rm cert}$ of \Cref{prop:ceps-cert} on
$\Om_2$, at three auxiliary resolutions. Column $\sigma_{\rm loc}^\star$ is the analytic
constant \eqref{eq:delta-star} at the same $\epsilon$, which is both the preliminary input to
\eqref{eq:ceps-cert} and the fallback shift. The $\epsilon=1$ row is a validation rather than
a production case: there $\mathcal B_\epsilon=H$, and the certified auxiliary enclosure is
bit-identical to the Stage-A enclosure of $\mu_{1,N}$, while the construction reproduces the
recorded Stage-A bound $L_1$ to $7\times10^{-13}$. Every entry is certified in interval
arithmetic; $\alpha$ is chosen by floating-point optimization and then certified, which costs
nothing in rigour because every $\alpha\in(0,\epsilon)$ yields a valid bound.}
\label{tab:ceps-cert}
\begin{tabular}{ccccccc}
\toprule
$N_{\rm aux}$ & $D_{\rm eee}$ & $\epsilon$ & $\alpha$ & $\Ceps^{\rm cert}$ & $\sigma_{\rm loc}^\star$ & reduction\\
\midrule
$48$ & $15\,625$ & $0.4$ & $0.2616$ & $2.4181812209$ & $3.5000$ & $1.45\times$\\
$48$ & $15\,625$ & $0.3$ & $0.2145$ & $3.5089574878$ & $4.3333$ & $1.24\times$\\
$64$ & $35\,937$ & $0.4$ & $0.2518$ & $1.9336986512$ & $3.5000$ & $1.81\times$\\
$64$ & $35\,937$ & $0.3$ & $0.2033$ & $2.8594577573$ & $4.3333$ & $1.52\times$\\
$72$ & $50\,653$ & $0.4$ & $0.2491$ & $\mathbf{1.7656061570}$ & $3.5000$ & $\mathbf{1.98\times}$\\
$72$ & $50\,653$ & $0.3$ & $0.2002$ & $2.6079442571$ & $4.3333$ & $1.66\times$\\
\midrule
$72$ & $50\,653$ & $1.0$ & $0.4778$ & $0.7137771573$ & $2.0000$ & validation\\
\bottomrule
\end{tabular}
\end{table}

\begin{lemma}[the auxiliary ground state is totally symmetric]\label{lem:aux-sector}
Let $\Om$ and $V$ be invariant under the three coordinate reflections. Then the lowest
Neumann eigenvalue of $\mathcal A_\epsilon$ is simple and its eigenfunction is even in each
coordinate.
\end{lemma}

\begin{proof}
$\mathcal A_\epsilon$ is bounded below with form domain $H^1(\Om)$ and $\Om$ is connected, so
the semigroup $e^{-t\mathcal A_\epsilon}$ is positivity improving; the ground eigenvalue is
therefore simple and its eigenfunction may be taken strictly positive. Each reflection $R$
commutes with $\mathcal A_\epsilon$ and so maps the one-dimensional ground eigenspace to
itself, whence $Ru=\pm u$; strict positivity excludes the minus sign for each of the three
reflections.
\end{proof}

\noindent
\Cref{lem:aux-sector} is what licenses evaluating $\nu_{1,N}$ on the fully even block alone:
the reduction is a proof, not a symmetry assignment. No corresponding argument is available
for $\lambda_2$, which is why the separator of \Cref{sec:mu2} is generated and certified
within the method rather than assigned. Finally, the constant of \Cref{lem:sharp-form}
remains in use --- as the preliminary $\sigma_{\rm loc}(\alpha)$ that seeds
\eqref{eq:ceps-cert}, and as the fallback shift whenever the auxiliary solve is not run.
The certified values are collected in \Cref{tab:ceps-cert}; \Cref{rem:sharpness} records how far they stand from the optimum and what governs the remaining margin.

\begin{remark}[how far from optimal, and what governs the margin]
\label{rem:sharpness}
The optimal constant in \eqref{eq:sharp-form} is $\Ceps^{\rm opt}=-\eta(\epsilon)$. On $\Om_2$ the Hardy
shift \eqref{eq:sigma} exceeds it by a factor $11.6$, and the analytic $\sigma_{\rm loc}$ of
\Cref{lem:sharp-form} by $3.87$ at $\epsilon=0.4$ and $3.96$ at $\epsilon=0.3$. The certified
$\Ceps^{\rm cert}$ of \Cref{prop:ceps-cert} exceeds it by $1.95$ and $2.38$ at $N_{\rm aux}=72$. The
remaining margin is governed by the auxiliary projection constant $\hat C_N$, hence by the auxiliary
resolution. This resolution is independent of the production resolution, since admissibility does not
depend on $N$. A finite element discretization \cite{LiuPlum2025} reaches a sharper constant at equal
effort. Its projection constant $\max_K h_K/\pi$ is local, and mesh grading puts resolution where $V^-$ is
large, whereas a spectral basis pays for the singularity everywhere. The gap is a property of the basis,
not of the method.
\end{remark}

\subsection{The sharp coercivity constant in detail}\label{sec:sharp-detail}\label{app:si:sharp-ceps}

The dominant term of the projection bound is $\sigma^2/\nu_*$ (\Cref{app:si:proj-proof}), and the
Galerkin values $\mu_{k,N}$ have long converged, so reducing $\sigma$ is the lever that sharpens the bound
at fixed $N$. For $\epsilon\in(0,1)$ the coercivity step needs $C_\epsilon$ with
\begin{equation}\label{app:eq:form-bound-abstract}
  (V^-u,u)\ \le\ \epsilon\,\|\nabla u\|^2+C_\epsilon\,\|u\|^2\qquad\forall u\in H^1(\Omega),
\end{equation}
and the smallest such constant is
\begin{equation}\label{app:eq:ceps-eta}
  C_\epsilon^{\mathrm{opt}}=-\,\eta(\epsilon),\qquad
  \eta(\epsilon):=\inf_{0\ne u\in H^1(\Omega)}
    \frac{\epsilon\,\|\nabla u\|^2-(V^-u,u)}{\|u\|^2},
\end{equation}
the lowest Neumann eigenvalue of
\begin{equation}\label{app:eq:aux-operator}
  \mathcal A_\epsilon:=\epsilon\,(-\Delta)+V\qquad\text{on }H^1(\Omega).
\end{equation}
By \Cref{lem:hardy}, $\eta(\epsilon)$ is finite. The constant mode gives $\eta(\epsilon)<0$, so the shift
is positive. Any $C_\epsilon\ge C_\epsilon^{\rm opt}$ is admissible, and the singular $V$ enters
directly, with no regularization.

\paragraph{No circularity.} The input to \eqref{eq:ceps-cert} is the closed-form
$\sigma_{\rm loc}(\alpha)$, and every $\alpha\in(0,\epsilon)$ yields a valid output. No step consumes a
quantity it produces. $\alpha$ is optimized in floating point and only the selected value is
certified. The Ritz value $\eta_{1,N}$ lies on the wrong side \eqref{eq:ritz-wrong-side} at every $N$.
It enters \eqref{eq:ceps-cert} only as the input that the projection bound converts into a certified
lower bound on $\eta$. In the cosine basis $\mathcal A_\epsilon$ needs no new assembly:
\begin{equation}\label{app:eq:aux-matvec}
  \mathcal A_\epsilon X=[(-\Delta)+V]\,X+(\epsilon-1)\,K_{\mathrm{diag}}\!\circ X ,
\end{equation}
evaluated on the fully even block by \Cref{lem:aux-sector}.

\paragraph{Magnitude.} For H$_2^+$ on $\Om_2$ ($Z_{\rm tot}=2$, $\ell=8$) at $\epsilon=0.4$, the analytic
$\sigma_{\rm loc}=7/2$ of \Cref{lem:sharp-form} improves on the optimized Hardy shift \eqref{eq:sigma}
($10.65$) by $3.04\times$. This shrinks $\mu_{1,N}-L_1$ by $7.6$ at $N=64$ and $8.0$ at $N=96$, and
brings the separation certificate down to $N=68$. $\Ceps^{\rm cert}$ gains a further $1.98\times$, and
$\Ceps^{\rm opt}\approx0.9051$ would gain $3.87\times$ in all (\Cref{rem:sharpness}).

\section{Parameter choices, alternatives, and sources of error}\label{app:choices}
The main text states the settings of \Cref{tab:configs} and the bounds they produce. This appendix
records how those settings were chosen, which alternatives the method offers at each choice, and
where the residual width and the cost come from. None of it is needed to reproduce a reported bound: the released certificates carry every input.

\subsection{The box, and the balance that fixes it}\label{app:choices:box}\label{sec:box}
Let $\Om_L = [-L,L]\times[-0.8L,0.8L]^2$ denote the truncation box, and let
$N$ be the number of modes per axis in the production basis and $N'>N$ in the
auxiliary basis. The certified enclosure width decomposes as
\begin{equation}\label{eq:width-decomp}
  W(L,N)
  \;=\; \underbrace{\lambda_1^{D}(\Om_L) - \lambda_1(\RR^3)}_{\text{truncation}}
  \;+\; \underbrace{\lambda_1(\RR^3) - \lambda_1^{\mathrm{LG}}(\Om_L,N,N')}_{\text{discretization}} ,
\end{equation}
in which the first term decreases with $L$ --- a larger box is a better
approximation to $\RR^3$ --- and the second increases with $L$, because at
fixed $N$ a larger box means fewer modes per unit length and hence poorer
resolution of the Coulomb cusps. An optimum in $L$ therefore exists.

What makes the balance worth knowing is a cost asymmetry. The dimension of the
fully-even Cartesian parity sector in which the ground state lives is
\begin{equation}\label{eq:Deee}
  D_{\mathrm{eee}} \;=\; (N/2+1)^3 ,
\end{equation}
which is \emph{independent of $L$}. Every dense interval matrix in the
pipeline, and therefore every unit of assembly time and every byte of memory,
is set by $N$ and $N'$ alone. \textbf{Moving the truncation box costs
nothing:} any accuracy available from a better choice of $L$ is available for
free, and a box chosen conservatively large is paid for entirely in accuracy.

\begin{remark}[how the boxes below were chosen, and what is out of scope]\label{rem:box-scope}
The balance region was located by evaluating the two Ritz quantities of \eqref{eq:width-decomp} over a
grid in $(L,N)$. A floating-point Lehmann--Goerisch surrogate was used only to select runs
(\Cref{rem:float-not-bound}); \Cref{fig:optimalbox} shows the sweep. The truncation branch decays in $L$
at a rate within $13\%$ of the Agmon-type prediction $0.8\cdot2\sqrt{\abs{\lambda_1}}$ for H$_2^+$ and
$d=4$. The discretization exponent grows with box size, so $L$ is chosen near the balance point; the
$L=20$ comparison configuration is over-truncated in this sense. Locating the optimum sharply is out of
scope, and a certificate is valid on any box. Configurations with $\max_k\abs{a_k}/L>0.8$ are rejected,
because the large-$t$ moment integral degrades near a face. All reported configurations have
$\max_k\abs{a_k}/L\le0.5$.
\end{remark}

\begin{figure}[t]
\centering
\includegraphics[width=\textwidth]{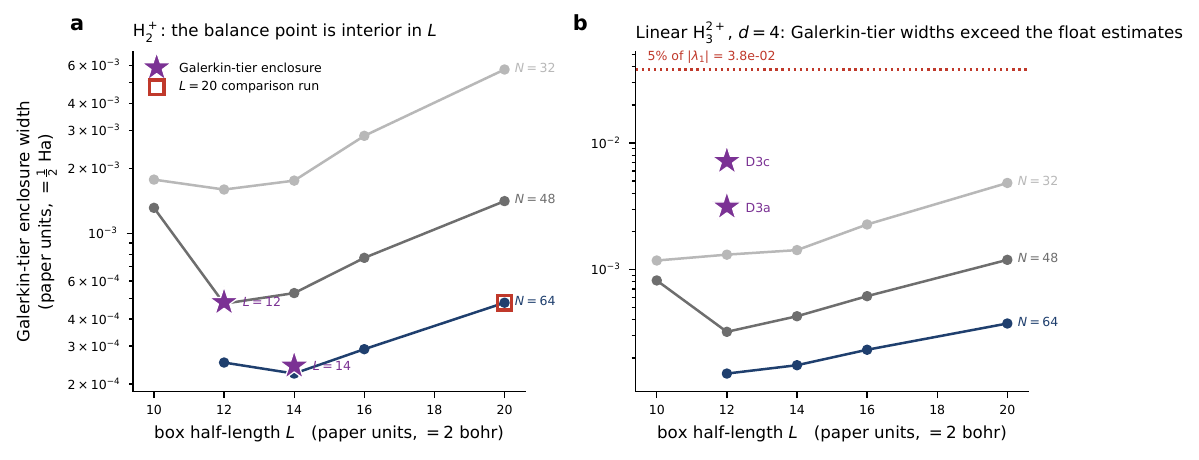}
\caption{Galerkin-tier enclosure width against truncation box size (the sweep and the balance law are Galerkin-tier statements; see \Cref{sec:results-h2}). Widths are in the units of \Cref{sec:units}. Curves show the Lehmann--Goerisch tier of the sweep, which was run at $L\in\{10,12,14,16,20\}$ and $N\in\{32,48,64\}$; the cheap dense tier covers the full $L\in\{8,\dots,28\}\times N\in\{24,\dots,64\}$ grid but does not yield a certified width, so no curve is drawn outside the certified grid. \textbf{(a)} H$_2^+$, with the $L=20$ comparison run marked and the certified points at $L=12$ and $L=14$ overlaid; the Galerkin-tier value at $L=14$ is $1.96\times$ narrower than the comparison run at identical $N$, $N'$ and $D_{\mathrm{aux}}$, hence at unchanged operation count. \textbf{(b)} Linear H$_3^{2+}$ at $d=4$ with the two certified points; D3c is at $N=24$, a basis size outside the swept set, and is therefore shown as a point without a curve. Both certified widths lie \emph{above} the float curve at the same $(L,N)$ --- by factors $2.43$ and $2.97$ against the recorded values obtained before the coercivity shift was certified --- which is why the float tier is a cell-selection tool and not a bound (\Cref{rem:float-not-bound}); both remain an order of magnitude inside the adequacy criterion. Both panels are generated from the released sweep surface and the configuration certificates.}
\label{fig:optimalbox}
\end{figure}

\subsection{The sign of the truncation bias}\label{app:choices:bias}
\Cref{tab:dir} states the bias on the two boxes of the main text as a single certified number, the gap
$L^{\mathrm{Dir}}-\sup\mu_{1,N}$, and its sign is what distinguishes them. It is positive on $\Om_1$,
where the truncation error exceeds the width the method resolves and the two brackets are disjoint,
and negative on $\Om_2$, where truncation has collapsed below that width --- not because it is absent
but because $2.113\times10^{-3}$ of certified Dirichlet bracket cannot resolve a true separation of
$1.111\times10^{-3}$. The third box, twice $\Om_2$ in every half-side,
\begin{equation}\label{eq:box3}
  \Om_3=[-20,20]\times[-16,16]^2,\qquad \ell=16,
\end{equation}
appears in this appendix only, as a comparison: it is discretization-limited rather than
truncation-limited, and no result reported in the main text rests on it. It continues the trend past the point where certification is
affordable, and is reported at the Galerkin tier because no operator-certified run was made there
(\Cref{app:choices:real}): its two-sided Dirichlet bracket at $N=64$, $N'=80$ is
$[-0.5514408191,-0.5509672618]$, of width $4.74\times10^{-4}$ --- an enclosure of one and the same box
eigenvalue, hence a pure discretization measure, and essentially the whole enclosure there. The three
boxes thus sit at three points of one balance: $\Om_1$ truncation-dominated, where the gap is visible
and the enclosure of $\lambda_1(\R^3)$ is correspondingly loose; $\Om_2$ at the crossover, giving the
sharpest certified enclosure reported here; and $\Om_3$ discretization-limited, where the lever is a
\emph{smaller} box at the same spectral order.

\subsection{The coercivity shift}\label{app:choices:shift}
All three shifts the method admits were run end to end on $\Om_2$; \Cref{tab:sigmaloc-cert} reports
them. The sharpest uses the certified $\Ceps^{\rm cert}$ of \Cref{prop:ceps-cert} and returns
\begin{equation}\label{eq:sigmaloc-enclosure}
  \lambda_1\in[-0.5519142771258511,\,-0.5507405804827743],
\end{equation}
of width $1.1736966\times10^{-3}$, or $0.213\%$ of $\abs{\lambda_1}$ --- a Galerkin-tier enclosure, not the operator-certified one of \Cref{tab:certified}. Three properties of that choice are worth recording. It costs nothing
and gains $1.8\%$: because $A_0$, $A_2$ and $c$ are independent of the separator, the final step can
be re-evaluated at the separator the inadmissible shift would have produced, holding box,
resolution, moments and Goerisch shift fixed, and the admissible shift is the narrower of the two.
It is also cheaper, making the auxiliary separator stage succeed at $N'_{\rm sep}=48$ where the
analytic shift required $80$, and its separator $-0.3468177613$ improves on the Stage-A $L_2$ of
\Cref{tab:stagea} by $2.687\times10^{-2}$; using Stage-A $L_2$ alone gives $1.181846\times10^{-3}$,
forfeiting about a third of the gain rather than all of it. And it is not converged: it falls
monotonically with the auxiliary resolution and is still decreasing at $N_{\rm aux}=72$, where the
auxiliary enclosure peaks at $345.2$\,GB against the limit $\Ceps^{\rm opt}=0.9051$. Roughly a
further factor two is available in the shift; what binds is the auxiliary cost, not the
construction.

\begin{table}[t]\centering\small
\caption{Certified enclosures on $\Om_2$ with $L=10$, $N=64$, $N'=80$, under the two shifts
available to the method. Column~1 uses the analytic $\sigma_{\rm loc}$ of
\Cref{lem:sharp-form}; column~2 uses the certified $\Ceps^{\rm cert}=1.7656061570$ of
\Cref{prop:ceps-cert} at $N_{\rm aux}=72$. Both are valid certificates and both contain
\eqref{eq:lam1exact}; the differences are sharpness and cost, and the better shift is both
sharper and cheaper, because it makes the auxiliary separator stage succeed at
$N'_{\rm sep}=48$ where under the analytic shift it had to be raised to $80$. Stage-A
quantities are reported once, in \Cref{tab:stagea}.}
\label{tab:sigmaloc-cert}
\begin{tabular}{lll}
\toprule
 & $\sigma_{\rm loc}$, $N'_{\rm sep}=80$ & $\Ceps^{\rm cert}$, $N'_{\rm sep}=48$\\
\midrule
shift $\sigma$          & $3.5000$ & $\mathbf{1.7656061570}$\\
separator $\rho$, certified $\mu_2(\Om)$ & $-0.4784635650$ & $\mathbf{-0.3468177613}$\\
enclosure               & $[-0.5520111921,$ & $[\mathbf{-0.5519142771},$\\
                        & $\phantom{[}-0.5507405805]$ & $\phantom{[}\mathbf{-0.5507405805}]$\\
width                   & $1.2706116\times10^{-3}$ & $\mathbf{1.1736966\times10^{-3}}$\\
relative                & $0.230\%$ & $\mathbf{0.213\%}$\\
\midrule
$\rho>\Lambda_n$ margin & $+7.337\times10^{-2}$ & $+2.050\times10^{-1}$\\
$B$ (outer)             & $1.2046\times10^{-2}$ & $\mathbf{9.384\times10^{-2}}$\\
$\nu$ (outer)           & $-6.0911$ & $-2.1848$\\
$B$ (inner)             & $3.7154\times10^{-4}$ & $8.443\times10^{-4}$\\
\midrule
cost                    & $13\,896$\,s, $607.8$\,GB & $\mathbf{11\,419}$\,s, $\mathbf{422.5}$\,GB\\
\midrule
matched comparison: width at the & \multicolumn{2}{l}{$1.1955593\times10^{-3}$ \ \ (ratio $0.9817$: the admissible}\\
separator the inadmissible shift & \multicolumn{2}{l}{shift is the \emph{narrower} of the two, by $2.19\times10^{-5}$)}\\
would have produced & \\
$\Ceps^{\rm cert}$ at $N_{\rm aux}=48,64,72$ & \multicolumn{2}{l}{$2.4182$, $1.9337$, $1.7656$ \ \ (limit $\Ceps^{\rm opt}=0.9051$)}\\
\bottomrule
\end{tabular}
\end{table}

\subsection{Operator-certified and Galerkin tiers}\label{app:choices:real}
The operator-certified H$_2^+$ widths $2.72\times10^{-3}$ ($\Om_2$), $6.00\times10^{-3}$ ($L=12$) and
$6.78\times10^{-3}$ ($L=14$) \emph{invert} the ordering of the Galerkin tier. The operator-certified
$A_2^G$ exceeds the Galerkin one by the auxiliary discretization term ($6.8\times10^{-4}$,
$2.2\times10^{-3}$ and $1.9\times10^{-3}$). That term is set by $N'$ and the box, not by the sharpness of
the value it corrects. It is divided by $\nu_*(N')=((N'+1)\pi/2L_x)^2\propto L_x^{-2}$, so at fixed
$N'=80$ the operator-to-Galerkin width ratio is $1.07$ on $\Om_1$ and $2.16$ on $\Om_2$. The $L=20$ box
has $\nu_*=40.5$, against $161.9$ on $\Om_2$. Hence the box study of \Cref{app:choices:box} is a
Galerkin-tier result whose optimum does not survive certification. Sharpening an operator-certified
enclosure is a matter of raising $N'$ (\Cref{rem:inexact-solve}), not of moving the box.

\subsection{The separator, and why the hypothesis tests are binding}\label{app:choices:sep}\label{sec:binding}
Run-time hypothesis checks matter only if they can fail, and two do on the configurations reported.
The cheap Stage-A candidate separator (\Cref{tab:mu3-needed}) clears $\rho>\Lambda_n$ at only two of six
configurations; at the other four, the closing transform returns a value \emph{above} $\lambda_1$. The
driver uses the best valid variant, and the failure is detected from $A_0$ and $\rho$ alone. On $\Om_3$
the separation test $L_2>U_1$ fails at $N=32$ and first passes at $N=48$. An assumed hypothesis would
have been satisfied silently in each case.

\subsection{Why linear H\texorpdfstring{$_3^{2+}$}{H3(2+)}, and its reference values}\label{app:choices:h3}
A symmetric linear three-centre configuration keeps all three coordinate reflections, so H$_3^{2+}$
certifies at the cost of H$_2^+$. An equilateral triangle keeps at most two and multiplies the dense
interval matrix by $4$--$16$. The central nucleus puts a singularity at $s=0$, where every even mode is
maximal, and this tests the moment construction of \Cref{sec:moments}. Interval containment against a
$256$-bit reference passed $112$ of $112$ tests at $s=0$ (maximal width $1.926\times10^{-12}$, against
$1.922\times10^{-12}$ off-centre). With the third charge set to zero, the assembly is bit-identical to
the two-centre code. The reference values of \Cref{sec:results-h3} come from a floating-Gaussian
Rayleigh--Ritz calculation, validated on the exact H$_2^+$ levels, on the large-$d$ asymptote
$-0.5+0.5/d$ (total energy, hartree), and on four charge-switching limits including the $Z=3$ united
atom. Their uncertainty $\pm3.8\times10^{-7}$ is transferred from an H$_2^+$ control run through the
identical chain (residual $1.903\times10^{-7}$, doubled), not inferred from convergence history
(\Cref{rem:float-not-bound}).

\subsection{Convergence, and where the cost goes}\label{app:choices:cost}
The fused, threaded interval Kronecker accumulation reproduces both endpoints bit-for-bit against the
serial and unfused forms, checked on all $46\,656$ entries of a three-centre moment product, and
production is gated on this check. Assembly at $D_{\rm aux}=68\,921$ takes $1\,328$\,s, against the
$27\,883$\,s predicted by the earlier single-threaded scaling ($5.87\times10^{-6}D^2$\,s; a prediction,
not part of \Cref{tab:cost}). The Lehmann--Goerisch stage is $64\%$ ($L=12$) and $71\%$ ($L=14$) of the
total. Stage~A and the separator are valid at any resolution and are raised only as far as \Cref{sec:mu2}
requires ($48/60$ at $L=12$, $56/72$ at $L=14$), so their share falls from $30\%$ to $19\%$ while the
production basis grows from $N=48$ to $64$. \Cref{fig:convergence} shows the convergence diagnostics.

\begin{figure}[tbp]\centering
\includegraphics[width=0.8\textwidth]{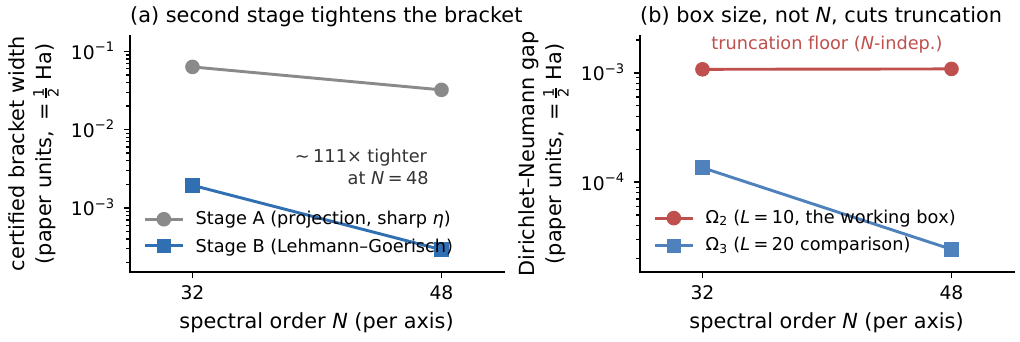}
\caption{Two convergence diagnostics at spectral orders $N=32,48$.
\textbf{(a)} Certified bracket width on $\Omega_2$: the Stage-B
Lehmann--Goerisch step tightens the Stage-A projection bracket by $\sim\!10^2$ at
fixed $N$. \textbf{(b)} Discrete Dirichlet--Neumann gap (a truncation-error proxy):
$N$-independent on $\Omega_2$ (floor $\approx1.1\times10^{-3}$), collapsing one to two orders on the $L=20$ comparison box $\Omega_3$. On $\Omega_2$ the enclosure is
truncation-limited and refining $N$ cannot help; on $\Omega_3$ it is discretization-limited. The optimum lies between (\Cref{rem:box-scope}).}
\label{fig:convergence}
\end{figure}

\subsection{The parity reduction, and the symmetry-free form}\label{app:choices:sectors}
The parity decomposition saves memory and is not part of the theory. \Cref{thm:coulomb-bound} holds
for every $k\le\dim V_N$, so on the unreduced cosine space $k=1,2$ bound $\mu_1(\Om),\mu_2(\Om)$ with no
symmetry input. The symmetry-free pipeline certifies $\sigma$ and $\eta_V$ on $V_N$, takes the Ritz
values of one pencil, applies \Cref{thm:coulomb-bound} at $k=1,2$, checks $L_2>\sup\mu_{1,N}$, and runs
Lehmann--Goerisch with $\rho=L_2$. No block is named and the hypothesis of \Cref{rem:sector-assumed}
does not arise. At equal memory it costs a factor two in $N$, and hence ${\approx}3.9$ in
$\widehat C_N^2$. The saving of the reduction requires invariance under all three reflections; with $r$
surviving reflections the dense matrix--vector cost falls by $58$, $16$, $4$ and $1$ for $r=3,2,1,0$.

\Cref{tab:symfree} runs this form on $\Om_2$. At $N=36$, the dimension of the production eee block,
$L_2=-0.4662$ is a separator with no identification. Fed through the same eee Lehmann--Goerisch step
(same $A_0,A_2$), it gives $\mu_1(\Om)\ge-0.551990$ against $-0.551914$ from the reduced oee separator
$-0.3468$. The loss of $7.5\times10^{-5}$, ${\approx}16\%$ of the widths of \Cref{tab:certified}, is
$1.7$ times the linear estimate from $\mathrm{d}L_1^{\rm LG}/\mathrm{d}\rho$, because the map is convex in
$\rho$. A separator at the midpoint $\tfrac12(\lambda_1+\lambda_2)$ already gives a width within a factor
$1.50$--$1.84$ of the best any separator could give, and the certified separators lie within $1.03$--$1.33$
of it. The certified shift is load-bearing here: under $\sigma_{\rm loc}=7/2$ the unreduced space first
separates at $N=58$ ($L_2=-0.8251$ at $N=36$). Since $\eta_V^2$ is a supremum over the trial space, it is
larger on the unreduced space and grows with $N$.

\begin{table}[t]
\centering
\caption{The certified chain run \emph{without} the parity-sector reduction, on $\Om_2$ at the
certified shift $\sigma=\Ceps^{\rm cert}=1.7656$. All quantities are computed on the unreduced space $V_N$, $\dim=(N{+}1)^3$, with $\eta_V^2$
re-certified there; $\mu_{1,N},\mu_{2,N}$ are the first two Ritz values of one pencil, and $L_1,L_2$
follow from \Cref{thm:coulomb-bound} at $k=1,2$. Single node, $32$ threads; peak memory
${\approx}60$ bytes per matrix entry.}
\label{tab:symfree}
\begin{adjustbox}{max width=\linewidth}
\begin{tabular}{rrlllrrl}
\toprule
$N$ & $D=(N{+}1)^3$ & $\eta_V^2$ & $\mu_{2,N}$ & $L_2$ & $L_2-U_1$ & peak & wall\\
\midrule
$24$ & $15\,625$ & $(0.450,0.475]$ & $-0.336200$ & $-0.589047$ & $-3.80\times10^{-2}$ & $15$\,GB & $754$\,s\\
$32$ & $35\,937$ & $(0.475,0.500]$ & $-0.336702$ & $-0.495618$ & $+5.59\times10^{-2}$ & $77$\,GB & $3\,105$\,s\\
$36$ & $50\,653$ & $(0.475,0.500]$ & $-0.336846$ & $-0.466179$ & $+8.54\times10^{-2}$ & $153$\,GB & $6\,524$\,s\\
\bottomrule
\end{tabular}
\end{adjustbox}
\end{table}

\section{Global trial functions: proofs and details}\label{app:global}

\begin{proof}[Proof of \Cref{lem:even-sep}]
For $u\in H^1(\RR^3)$ put $v(x)=\bigl((\abs{u(x)}^2+\abs{u(-x)}^2)/2\bigr)^{1/2}$. Then
$v\in H^1(\RR^3)\cap\mathcal H_+$, $\norm{v}=\norm{u}$, $\abs{\nabla v(x)}^2\le
(\abs{\nabla u(x)}^2+\abs{\nabla u(-x)}^2)/2$, and $V$ is inversion invariant, so the quadratic form
satisfies $q[v]\le q[u]$ and the two infima agree. Write $H=h_1+h_2$ with
$h_k=-\tfrac12\Delta-1/\abs{x-a_k}$. Each $h_k$ is a hydrogen operator: its ground state
$\phi_k\propto e^{-\abs{x-a_k}}$ is simple with energy $-\tfrac12$, and $q_{h_k}[u]\ge-\tfrac18\norm{u}^2$
for $u\perp\phi_k$. Inversion exchanges $\phi_1$ and $\phi_2$, so an even $u$ has
$(u,\phi_1)=(u,\phi_2)$, and $u\perp\phi_1+\phi_2$ forces $u\perp\phi_1$ and $u\perp\phi_2$. On this
codimension-one subspace of the even form domain $q[u]=q_{h_1}[u]+q_{h_2}[u]\ge-\tfrac14\norm{u}^2$, and
the max--min principle gives $\lambda_2(H_+)\ge-\tfrac14$.
\end{proof}

\begin{proof}[Proof of \Cref{thm:temple-global}]
The upper bound is the Rayleigh principle. By \Cref{lem:even-sep} and $\lambda_1\le\bar\lambda<-\tfrac14$,
the spectrum of $H_+$ lies in $\{\lambda_1\}\cup[-\tfrac14,\infty)$, so
$((H-\lambda_1)\psi,(H+\tfrac14)\psi)\ge0$. Expanding about $\bar\lambda$ gives
$\norm{(H-\bar\lambda)\psi}^2\ge(\bar\lambda-\lambda_1)(-\tfrac14-\bar\lambda)\norm{\psi}^2$
\cite{Temple1928}.
\end{proof}

\paragraph{The trial \eqref{eq:prolate-trial} and its residual.} With $a=2$, $r_k=\abs{x-a_k}$, let
$\xi=(r_1+r_2)/(2a)\in[1,\infty)$, $\eta=(r_1-r_2)/(2a)\in[-1,1]$ and $D=\xi^2-\eta^2$, and take
\[
  \psi=e^{-p(\xi-1)}P(\xi,\eta),
\]
with $p>0$ rational and $P$ a polynomial, even in $\eta$, with rational coefficients. Inversion fixes
$\xi$ and reverses $\eta$, so $\psi\in\mathcal H_+$; the reflections $x_2\mapsto-x_2$ and
$x_3\mapsto-x_3$ fix both $r_1$ and $r_2$, so $\psi$ is in fact even in each coordinate,
$\psi\in\mathcal H_{\rm eee}$, and \Cref{cor:temple-eee} applies to it. The coordinates are sums and differences of
distances, smooth away from the nuclei; near a nucleus $\psi$ and its gradient are bounded and its
second derivatives are $O(1/r_k)$, which is square integrable in three dimensions; and $\psi$ decays
exponentially. Hence $\psi\in H^2(\RR^3)=D(H)$. In these coordinates
$H\psi=e^{-p(\xi-1)}N/(a^2D)$ with the polynomial
\[
  N=-(\partial_\xi-p)\bigl[(\xi^2-1)(\partial_\xi-p)P\bigr]-\partial_\eta\bigl[(1-\eta^2)\partial_\eta P\bigr]-2a\xi P .
\]
With all integrals over $\xi\ge1$, $\abs{\eta}\le1$ and the common factor $2\pi a^3$ removed,
\begin{gather*}
  A=\iint e^{-2p(\xi-1)}P^2D,\qquad \bar\lambda=\frac{\iint e^{-2p(\xi-1)}PN}{a^2A},\\
  \sigma_+^2=\frac{\iint e^{-2p(\xi-1)}\,\xi R^2/D}{a^4A},\qquad R=N-a^2\bar\lambda DP .
\end{gather*}
The weight $\xi\ge1$ makes $\sigma_+^2$ an upper bound for $\norm{(H-\bar\lambda)\psi}^2/\norm{\psi}^2$
and keeps every moment in closed form: $A$ and $\bar\lambda$ are exact rationals, and $\sigma_+^2$ is an
exact rational combination of $1$, $\gamma+\log(4p)$ and $e^{4p}E_1(4p)$ (\Cref{app:prolate}). Only
these two constants are enclosed, with outward rounding. The coefficients of $P$ are proposed by a
floating-point computation and then fixed as exact rationals; the certificate does not depend on how
they were obtained.

\paragraph{Three centres: the estimate of \Cref{sec:global}.} For linear H$_3^{2+}$ neither the separator nor the residual
moments are available in closed form, but the efficiency of a global trial can be estimated.
\Cref{tab:global-h3} uses multi-centre Slater functions
$\abs{x-a_k}^n e^{-\zeta\abs{x-a_k}}P_l(\cos\theta_k)$ about the three nuclei, coefficients minimising the
residual, and floating-point quadrature, and replaces the separator by the Ritz estimate of
$\lambda_2$. The widths $\sigma^2/(\lambda_2-\bar\lambda)$ are therefore estimates of what Temple's
inequality would give with a sharp separator, not bounds. With a few hundred functions they reach
$10^{-10}$, six to seven orders of magnitude below the box widths of \Cref{tab:certified}; at that level
the conditioning of the basis in double precision, not the method, limits them. A rigorous version
needs a certified lower bound for $\lambda_2$ --- for a fully even trial, on $\lambda_2^{\rm eee}$, which
the Stage-A route of \Cref{app:stageA-sep} supplies without any analogue of \Cref{lem:even-sep} ---
and a validated integration of the three-centre residual; we leave both to further work.

\begin{table}[t]\centering\small
\caption{Floating-point estimates for linear H$_3^{2+}$ with multi-centre Slater trials on
$\RR^3$: residual $\sigma^2=\norm{(H-\bar\lambda)\psi}^2/\norm{\psi}^2$ and Temple width
$\sigma^2/(\lambda_2-\bar\lambda)$ with $\lambda_2$ replaced by its Ritz estimate, in the units of
\Cref{sec:units}. Estimates, not bounds.}
\label{tab:global-h3}
\begin{tabular}{@{}rcccc@{}}
\toprule
 & \multicolumn{2}{c}{$d=4$} & \multicolumn{2}{c}{$d=2.8$}\\
\cmidrule(lr){2-3}\cmidrule(lr){4-5}
functions & $\sigma^2$ & width & $\sigma^2$ & width\\
\midrule
90 & $1.68\times10^{-7}$ & $9.57\times10^{-7}$ & $2.94\times10^{-7}$ & $1.02\times10^{-6}$\\
252 & $1.29\times10^{-9}$ & $7.35\times10^{-9}$ & $1.22\times10^{-8}$ & $4.23\times10^{-8}$\\
486 & $4.22\times10^{-11}$ & $2.41\times10^{-10}$ & $2.72\times10^{-11}$ & $9.40\times10^{-11}$\\
\bottomrule
\end{tabular}
\end{table}

\begin{remark}[scope]\label{rem:global-scope}
Three properties of H$_2^+$ are used, and none holds in general:
\begin{itemize}
\item Inversion exchanges two equal nuclei, so one orthogonality condition removes both atomic ground
states. Unequal charges, or linear H$_3^{2+}$, need one condition per nucleus.
\item Prolate coordinates make the moments closed-form, and three centres have no analogue.
\item \Cref{lem:even-sep} uses the hydrogen spectrum on $\RR^3$.
\end{itemize}
The two routes combine. \Cref{thm:enc-main} at $k=2$ gives $\mu_2(\Om)\le\lambda_2(\RR^3)$ whenever
$U_2^{\rm Dir}<\sigma(\Om)$, so a certified lower bound on $\mu_2(\Om)$ may replace $-\tfrac14$ in
\Cref{thm:temple-global} for any arrangement of nuclei. The same holds in the fully even sector for
reflection-invariant potentials (\Cref{cor:temple-eee}, \Cref{tab:global-sep}), which covers linear
H$_3^{2+}$. What three centres lack is a trial in $D(H)$ whose residual can be integrated with a
certified error.
\end{remark}

\begin{remark}[Temple's inequality on the box]\label{rem:temple-box}
For a single trial vector, \eqref{eq:lg-closed} is Temple's inequality in weak form. Let $H'$ be the
operator of $M$ relative to $N$, with eigenvalues $\kappa_1<\kappa_2\le\cdots$ and $\kappa_1<\rho\le\kappa_2$.
Evaluating $(H'-\kappa_1)(I-\rho H'^{-1})\succeq0$ at $v$ and using $\langle v,H'^{-1}v\rangle\le A_2$ gives
\eqref{eq:lg-closed}. The strong form needs $\norm{H'v}^2$ instead. For a Ritz vector its residual
$\norm{(I-\Pi_N)(Vv)}^2$ is $O(N^{-1})$, because near a nucleus $Vv\approx-Z\,v(a)/\abs{x-a}$ has Fourier
decay $\abs{k}^{-2}$. The weak form divides this component by $\nu_*\propto N'^2$ (measured decay
$N'^{-2.83}$, \Cref{rem:inexact-solve}), so Stage~B uses the weak form. On a box neither form reaches the
precision of \Cref{tab:global}. The global trial does, because $H\psi$ is known in closed form on $\RR^3$
and no box enters.
\end{remark}

\section{Closed-form moments for prolate trials}\label{app:prolate}
Put $s=2p$, $K_n=\int_1^\infty e^{-s(\xi-1)}\xi^n\,d\xi$, so that $K_0=1/s$ and
$K_n=(1+nK_{n-1})/s$. The $\eta$-integrals of polynomial terms are $\int_{-1}^1\eta^{2j}\,d\eta=2/(2j+1)$,
so $A$ and $\int e^{-s(\xi-1)}PN$ are exact rationals. The residual needs, for $j\ge0$ and $\xi>1$,
\[
  \int_{-1}^{1}\frac{\eta^{2j}}{\xi^2-\eta^2}\,d\eta=\xi^{2j-1}\log\frac{\xi+1}{\xi-1}
  -2\sum_{r=0}^{j-1}\frac{\xi^{2r}}{2j-2r-1},
\]
and therefore the moments $J_m=\int_1^\infty e^{-s(\xi-1)}\xi^m\log\frac{\xi+1}{\xi-1}\,d\xi$. The weight
$\xi$ in $\sigma_+^2$ makes every radial exponent at least one, so only $m\ge0$ occurs. With
$C=\gamma+\log(2s)$ and $F=e^{2s}E_1(2s)$,
\[
  J_0=\frac{C+F}{s},\qquad J_{m+1}=J_0+\frac{m+1}{s}J_m-\frac2s Q_m,\qquad
  Q_m=\int_1^\infty e^{-s(\xi-1)}\frac{\sum_{k=0}^m\xi^k}{\xi+1}\,d\xi ,
\]
the recursion following from integrating the derivative of $e^{-s(\xi-1)}(\xi^{m+1}-1)\log\frac{\xi+1}{\xi-1}$.
Division by $\xi+1$ reduces $Q_m$ to the $K_n$ and $\int_1^\infty e^{-s(\xi-1)}/(\xi+1)\,d\xi=F$. Every
residual moment is thus an exact rational combination of $C$, $F$ and $1$. The coefficients of $C$ and
$F$ and the rational part are formed in exact arithmetic; only $C$ and $F$ are evaluated, as
outward-rounded enclosures.

\section{A computed separator from Stage~A}\label{app:stageA-sep}

The separator in \Cref{thm:temple-global} need not be analytic. Write $\mathcal H_{\rm eee}\subset\mathcal H_+$
for the functions even in each coordinate and $H_{\rm eee}$ for the restriction of $H$ to it; its
min--max values are $\lambda_k^{\rm eee}(\RR^3)$.

\begin{table}[!t]\centering\small
\caption{Lower bounds on $\lambda_2^{\rm eee}(\RR^3)$ for H$_2^+$, and the Temple width they give for
the $1\,701$-term trial of \Cref{tab:global} through \eqref{eq:temple-eee}. The certified row is a certificate; the other Stage-A rows are floating-point evaluations of
\eqref{eq:coulomb-bound} on the eee block of $\Om_s$ at $\sigma=1.764$, rounded down, and are estimates.}
\label{tab:global-sep}
\begin{tabular}{@{}llccc@{}}
\toprule
separator & $D$ & $\rho$ & width & relative width\\
\midrule
analytic, \Cref{lem:even-sep} & --- & $-0.250$ & $7.43\times10^{-51}$ & $1.000$\\
\textbf{Stage~A, eigenvalue-ordered, certified} & $210\,000$ & $\mathbf{-0.2035}$ & $6.44\times10^{-51}$ & $0.866$\\
\midrule
Stage~A, eigenvalue-ordered basis & $103\,823$ & $-0.219$ & $6.74\times10^{-51}$ & $0.907$\\
Stage~A, eigenvalue-ordered basis & $210\,000$ & $-0.211$ & $6.58\times10^{-51}$ & $0.885$\\
Stage~A, eigenvalue-ordered basis & $420\,000$ & $-0.205$ & $6.46\times10^{-51}$ & $0.870$\\
box level $\mu_2^{\rm eee}(\Om_s)$ (limit of this route) & --- & $\approx-0.1953$ & $6.29\times10^{-51}$ & $0.846$\\
exact $\lambda_2^{\rm eee}(\RR^3)$ ($2s\sigma_g$) & --- & $\approx-0.1804$ & $6.04\times10^{-51}$ & $0.812$\\
\bottomrule
\end{tabular}
\end{table}

\begin{corollary}[Temple's bound with a computed separator]\label{cor:temple-eee}
Let $0\ne\psi\in D(H)\cap\mathcal H_{\rm eee}$, with $\bar\lambda$ and $\sigma_+^2$ as in
\Cref{thm:temple-global}, and let $\bar\lambda<\rho\le\lambda_2^{\rm eee}(\RR^3)$. Then
\begin{equation}\label{eq:temple-eee}
  \bar\lambda-\frac{\sigma_+^2}{\rho-\bar\lambda}\ \le\ \lambda_1(\RR^3)\ \le\ \bar\lambda .
\end{equation}
Since $\lambda_2^{\rm eee}(\RR^3)\ge\lambda_2(H_+)\ge-\tfrac14$, the choice $\rho=-\tfrac14$ is admissible
and \eqref{eq:temple-eee} contains \eqref{eq:temple-global} for fully even trials.
\end{corollary}
\begin{proof}
The symmetrisation in the proof of \Cref{lem:even-sep}, taken over the group generated by the three
coordinate reflections, $v=(\tfrac18\sum_g\abs{u\circ g}^2)^{1/2}$, gives
$\inf\sigma(H_{\rm eee})=\lambda_1(\RR^3)$. The spectrum of $H_{\rm eee}$ therefore lies in
$\{\lambda_1\}\cup[\rho,\infty)$, and the proof of \Cref{thm:temple-global} applies with $-\tfrac14$
replaced by $\rho$.
\end{proof}

\Cref{lem:even-sep} leaves a gap. The second level of the fully even sector is $2s\sigma_g$,
$\lambda_2^{\rm eee}(\RR^3)\approx-0.1804$, well above $-\tfrac14$. Stage~A can close much of that gap
without the hydrogen decomposition or equal charges. \Cref{thm:enc-main} holds sector by sector, so its
proof with eee trial functions gives $\mu_k^{\rm eee}(\Om)\le\lambda_k^{\rm eee}(\RR^3)$ whenever
$\sigma(\Om)>\lambda_k^{\rm eee}(\RR^3)$. \Cref{thm:coulomb-bound} on the eee block at $k=2$ then bounds
$\mu_2^{\rm eee}(\Om)$ from below. The level is diffuse, so confinement fails on the boxes of
\eqref{eq:boxes}. It holds on $\Om_s=[-12,12]\times[-11.5,11.5]^2$ ($\sigma(\Om_s)=-0.1714$), certified
by the interval Dirichlet bound $\lambda_2^{\rm eee}(\RR^3)\le U_2^{\rm eee}=-0.18022$ on $[-20,20]^3$. The box level $\mu_2^{\rm eee}(\Om_s)\approx-0.1953$ is the ceiling of this route.

\paragraph{The certified separator.} We ran this route on $\Om_s$ with the $D=210\,000$ eee cosine
modes of lowest Laplacian eigenvalue ($\nu_*\ge386.5845114$, enclosed). The shift is certified there:
a positivity test gives $\eta_{1,N}\ge-0.9080000420$ for the auxiliary operator at $\epsilon=0.4$,
and since the bracket of \eqref{eq:ceps-cert} increases with $\nu_{1,N}$, \Cref{prop:ceps-cert}
gives $\sigma=\Ceps^{\rm cert}(\Om_s)=1.2195258303$. At this shift
$\mu_{1,N}\ge-0.5513178709$, $\mu_{2,N}\ge-0.1953002923$, $\eta_V^2\le0.5445$, and
\eqref{eq:coulomb-bound} gives $\mu_2^{\rm eee}(\Om_s)\ge-0.2034431$; as $U_2^{\rm eee}<\sigma(\Om_s)=-6/35$,
\Cref{thm:enc-main} in the eee sector gives $\lambda_2^{\rm eee}(\RR^3)\ge\rho_A=-0.2035$. The
$353$\,GB matrices $P$ and $Q=(V^2\varphi_{\bm i},\varphi_{\bm j})$ are assembled in floating point
from the interval moment matrices of \Cref{sec:moments}, with a~priori bounds
$\norm{\Delta P}_2\le1.5\times10^{-8}$, $\norm{\Delta Q}_2\le9.3\times10^{-9}$. Each positivity statement is a
floating-point Cholesky factorization whose backward error is bounded, after completion, by
$\gamma_{D+3}\norm{|L|}_1\norm{|L|}_\infty\le2\times10^{-7}$ \cite[Thm.~10.3]{Higham2002}, and is subtracted with
the assembly bounds; $\mu_{2,N}$ uses rank-one interlacing ($H+xx^{\mathsf T}-sI\succ0\Rightarrow\mu_{2,N}>s$). At
$D=50\,653$ the same code certifies $\mu_{2,N}\ge-0.19529228$ (Ritz $-0.19529222$) and
$\eta_V^2\le0.500$ at $\sigma=1.764$ (floating-point $0.4967$); on $\Om_2$ it confirms the rung
$\eta_V^2\le0.55$ of \Cref{tab:formfactor-all}.

The certified row, $\rho_A=-0.2035$, gives \eqref{eq:global-enclosure-A}. The last two rows bound what any separator can do.

\section{Notation}\label{app:notation}
The main symbols are listed below, with the place of definition. Energies and lengths are in the units of
\Cref{sec:units}; configuration labels are defined in \Cref{tab:configs}.

\vspace{-2pt}
{\footnotesize\setlength{\tabcolsep}{4pt}\renewcommand{\arraystretch}{0.98}
\begin{longtable}{@{}>{\raggedright\arraybackslash}p{0.20\textwidth}p{0.77\textwidth}@{}}
\toprule
\multicolumn{2}{@{}l}{\emph{Operator, domains and spaces}}\\
\midrule
$H=-\Delta+V$, $V^-$, $Z_{\rm tot}$ & operator on $\R^3$; attractive part $V^-=\max(-V,0)$; total charge \eqref{eq:H}\\
$\Om_1,\Om_2,\Om_3,\Om_s$ & boxes \eqref{eq:boxes}; $\Om_3$: $L=20$ comparison box; $\Om_s$: box of \Cref{app:stageA-sep}\\
$\ell$, $\sigma(\Om)$ & nucleus--boundary clearance (\Cref{lem:hardy}); confinement constant $\inf_{x\notin\Om}V$ (\Cref{thm:enc-main})\\
$V_N$, $V_{N'}$; $\nu_*$ & cosine trial spaces ($N$, $N'$ modes per axis); smallest omitted Laplacian eigenvalue\\
eee, oee, eoe, eeo & parity sectors, one letter per axis (e even, o odd) (\Cref{lem:parity-decomp})\\
\midrule
\multicolumn{2}{@{}l}{\emph{Eigenvalues and bounds}}\\
\midrule
$\lambda_k(\R^3)$ & $k$-th eigenvalue of $H$ on $\R^3$ (enclosed); $\lambda_k^{p}$: its value in sector $p$\\
$\mu_k(\Om)$, $\lambda_k^{\rm Dir}(\Om)$ & Neumann, Dirichlet eigenvalues on $\Om$; $\mu_k(\Om)\le\lambda_k(\R^3)\le\lambda_k^{\rm Dir}(\Om)$ \eqref{eq:thm-trunc}\\
$\mu_{k,N}$ & $k$-th Ritz value on $V_N$, an upper bound on $\mu_k(\Om)$\\
$L_k$ ($L_k^{\rm Neu}$) & Stage-A certified lower bound on $\mu_k(\Om)$ (\Cref{thm:coulomb-bound})\\
$L_{k,{\rm LG}}^{\rm Neu}$, $U_k^{\rm Dir}$ & Stage-B lower bound; certified Dirichlet upper bound (\Cref{fig:framework})\\
$U_1$ & certified upper bound $\sup\mu_{1,N}$ used in the separation test $L_2>U_1$\\
\midrule
\multicolumn{2}{@{}l}{\emph{Coercivity and the projection bound (Stage~A)}}\\
\midrule
$\epsilon$, $\Ceps$ & gradient fraction in the form bound \eqref{eq:sharp-form}; an admissible constant\\
$\Ceps^{\rm opt}$, $\Ceps^{\rm cert}$ & optimal constant $-\eta(\epsilon)$; its certified bound (\Cref{prop:ceps-cert}), the shift $\sigma$ of every run\\
$\sigma_{\rm loc}$, $\sigma_{\rm H}$ & analytic shifts from \Cref{lem:sharp-form} and from the Hardy route \eqref{eq:sigma}\\
$C_{\rm grad}$, $C_{L^2}$; $r_1,\delta$ & localized Hardy constants; free cutoff parameters (\Cref{lem:hardy})\\
$\eta_V$ & potential form factor \eqref{eq:etaV-def}, tail-refined form \eqref{eq:etaV-tail}\\
$\widehat C_N$, $g$ & projection constant (\Cref{thm:coulomb-bound}); $g=\eta_V/\sqrt{\mu_{1,N}+\sigma}$\\
\midrule
\multicolumn{2}{@{}l}{\emph{Lehmann--Goerisch stage (Stage~B)}}\\
\midrule
$\rho$ & separator, $\Lambda_n<\rho\le\lambda_{m+1}$ \eqref{eq:separation}, certified in-method (\Cref{thm:mu2-candidates})\\
$A_0,A_1,A_2$; $B$ & moment matrices; $B=A_0-2\rho A_1+\rho^2A_2\succ0$ if $\rho>\Lambda_n$ (\eqref{eq:A-B}, \Cref{lem:B-positive})\\
$\Lambda_n$; $\nu_k$ & largest trial Rayleigh quotient \eqref{eq:Lambda-n}; eigenvalues of $Az=\nu Bz$ (\Cref{thm:lg})\\
$c$, $\mathcal H'$ & Goerisch shift and $\mathcal H'=\mathcal H+c\succ0$\\
$\epsilon'$, $A_2^G$ & realization's gradient fraction; bound $A_2^G\ge\langle v,\mathcal H'^{-1}v\rangle$ \eqref{eq:goerisch-corrected}\\
\midrule
\multicolumn{2}{@{}l}{\emph{Resolutions}}\\
\midrule
$N$, $N'$ & modes per axis of the production and auxiliary bases\\
$N_{\rm stageA}$, $N'_{\rm sep}$, $N_{\rm aux}$ & resolutions of Stage~A, the separator stage, and $\Ceps^{\rm cert}$\\
$D_{\rm eee}$, $D_{\rm aux}$ & dimensions of the eee block, $(N/2+1)^3$, and of the auxiliary block\\
\bottomrule
\end{longtable}}

\end{document}